\documentclass[11pt]{amsart}

\usepackage{amssymb,amsthm,amsmath,url,relsize,etoolbox,cite,nameref,enumerate,fancyhdr}
\usepackage[utf8]{inputenc}
\usepackage[hidelinks]{hyperref}

\theoremstyle{plain}
\newtheorem*{theorem*}{Theorem}
\newtheorem{theorem}[subsection]{Theorem}

\newtheorem{lemma}[subsection]{Lemma}

\newtheorem{corollary}[subsection]{Corollary}

\newtheorem{definition}[subsection]{Definition}

\newtheorem{remark}[subsection]{Remark}

\newtheorem{conjecture}[subsection]{Conjecture}

\DeclareMathOperator{\degree}{deg}
\DeclareMathOperator{\modulus}{mod}
\DeclareMathOperator{\residue}{Res}
\DeclareMathOperator{\rad}{rad}

\renewcommand{\Re}{\operatorname{Re}}
\renewcommand{\Im}{\operatorname{Im}}
\newcommand{\Ci}{\operatorname{Ci}}
\newcommand*\conj[1]{\overline{#1}}
\newcommand{\sumstar}{\sideset{}{^*} \sum}

\begin{document}

\allowdisplaybreaks

\title{The Hybrid Euler-Hadamard Product Formula for Dirichlet $L$-functions in $\mathbb{F}_q [T]$}
\author{M. Yiasemides}
\date{\today}
\address{Department of Mathematics, University of Exeter, Exeter, EX4 4QF, UK}
\email{my298@exeter.ac.uk}


\thanks{\textbf{Acknowledgments:} The author is grateful for an Engineering and Physical Sciences Research Council (UK) DTP Standard Research Studentship (grant number EP/M506527/1). The author would also like to thank Hung Bui and Nigel Byott for various comments and corrections, and Julio Andrade for suggesting this problem and for his comments.}

\maketitle

\begin{abstract}
For Dirichlet $L$-functions in $\mathbb{F}_q [T]$ we obtain a hybrid Euler-Hadamard product formula. We make a splitting conjecture, namely that the $2k$-th moment of the Dirichlet $L$-functions at $\frac{1}{2}$, averaged over primitive characters of modulus $R$, is asymptotic to (as $\degree R \longrightarrow \infty$) the $2k$-th moment of the Euler product multiplied by the $2k$-th moment of the Hadamard product. We explicitly obtain the main term of the $2k$-th moment of the Euler product, and we conjecture via random matrix theory the main term of the $2k$-th moment of the Hadamard product. With the splitting conjecture, this directly leads to a conjecture for the $2k$-th moment of Dirichlet $L$-functions. Finally, we lend support for the splitting conjecture by proving the cases $k=1,2$. This work is the function field analogue of the work of Bui and Keating. A notable difference in the function field setting is that the Euler-Hadamard product formula is exact, in that there is no error term.
\end{abstract}


\tableofcontents


\section{Introduction and Statement of Results} \label{section, introduction}

Mean values, or moments, of $L$-functions have many powerful applications in number theory, from the non-vanishing of $L$-functions at certain points, to zero-density estimates, positive lower bounds for the number of zeros of $\zeta (s)$ that lie on the critical line, and the proportion of simple zeros on the critical line (see \cite{Gonek2005_AppMVTTheoRZF} for a summary). They also have intrinsic interest because results on moments higher than the fourth have not been obtained, and instead we rely on conjectures. \\

Consider the Riemann zeta-function. For $\Re (s) > 1$,
\begin{align*}
\zeta (s) := \sum_{n=1}^{\infty} \frac{1}{n^{s}} ,
\end{align*}
which can be extended meromorphically to $\mathbb{C}$ with a simple pole at $s=1$. It was shown by Hardy and Littlewood \cite{HardyLittlewood1918_ContrTheoRZFTheoDistrPrime} that
\begin{align*}
\frac{1}{T} \int_{t=0}^{T} \Big\lvert \zeta \Big( \frac{1}{2} + it \Big) \Big\rvert^{2} \mathrm{d} t
\sim \log T ,
\end{align*}
as $\degree T \longrightarrow \infty$, and it was shown by Ingham \cite{MeanValueTheoRMZ_Ingham1926} that
\begin{align*}
\frac{1}{T} \int_{t=0}^{T} \Big\lvert \zeta \Big( \frac{1}{2} + it \Big) \Big\rvert^{4} \mathrm{d} t
\sim \frac{1}{12} \frac{6}{\pi^2} (\log T)^4 
\end{align*}
as $\degree T \longrightarrow \infty$. For higher moments it has been conjectured (see \cite[equation (4)]{RMTZeta_KeatingSnaith2000}) that, for integers $k \geq 0$,
\begin{align} \label{conjecture of 2k-th moment of RZF on crit line, statement}
\lim_{T \longrightarrow \infty} \frac{1}{(\log T)^{k^2}} \frac{1}{T} \int_{t=0}^{T} \Big\lvert \zeta \Big( \frac{1}{2} + it \Big) \Big\rvert^{2k} \mathrm{d} t 
= f(k) a(k) ,
\end{align}
where $f(k)$ is a real-valued function and
\begin{align*}
a(k)
:= \prod_{p} \Bigg( \bigg( 1 - \frac{1}{p} \bigg) ^{k^2} \sum_{m=0}^{\infty} \frac{d_k (p^m)^2}{p^m} \Bigg) .
\end{align*}
We have $a(0) = 1 , a(1) = 1 , a(2) = \frac{1}{\zeta(2)} = \frac{6}{\pi^2}$, and we have an understanding of $a(k)$ for higher values of $k$. The factor $f(k)$ is more elusive. Clearly, from the results described above, we have $f(0) = 1$, $f(1) = 1$, $f(2)= \frac{1}{12}$. It has been conjectured via number-theoretic means that $f(3) = \frac{42}{9!}$ \cite{Conj6thMomRZF_ConreyGhos1998} and $f(4) = \frac{24024}{16!}$ \cite{HighMomRZF_ConreyGonek1999}. For higher powers one must look at random matrix theory for conjectures. \\

We point out that one can also obtain conjectures for higher powers by using the recipe developed by Conrey, Farmer, Keating, Rubinstein, and Snaith \cite{IntegralMomLFunc_CFKRS}. Again these conjectures are in agreement with those obtained via random matrix theory. However, in this paper we focus on the latter. \\

It has been known for some time that there is a relationship between the Riemann zeta-function and eigenvalues of random unitary matrices. In 1972 it was observed by Montgomery and Dyson that the pair correlations of the non-trivial zeros of the Riemann zeta-function appear to behave similarly to the pair correlations of eigenvalues of a random Hermitian matrix \cite{Montgomery1973_PairCorrZeroZetaFunc}. Later, Odlyzko produced numerical evidence in support of this \cite{Odlyzko1987_DistrSpaceZeroZetaFunc}. \\

Given that the eigenvalues of a matrix are the zeros of its characteristic polynomial, it is reasonable to expect a relationship between $\zeta (s)$ on the critical line and the characteristic polynomials of unitary matrices. Keating and Snaith \cite{RMTZeta_KeatingSnaith2000} modeled $\zeta (s)$ at around height $T$ on the critical line by the characteristic polynomial of a random $N \times N$ unitary matrix. (Here, $N$ is chosen such that the mean spacing between the eigenphases of an $N \times N$ unitary matrix is the same as the mean spacing of the zeros of the Riemann zeta-function at around height $T$ on the critical line). They obtained the following result for integers $k \geq 0$:
\begin{align} \label{2k-th moment of char poly of unitary matrices, statement}
\int_{U \in U(N)} \lvert Z (U, \theta) \rvert^{2k} \mathrm{d} U
\sim f_{CUE} (k) N^{k^2}
\end{align}
as $N \longrightarrow \infty$. Here $U(N)$ is the set of all unitary $N \times N$ matrices; for all $U \in U(N)$, we take $Z(U,\theta) := \det \big( I_N - U e^{-i \theta} \big)$ to be the characteristic polynomial of $U$; the integral is with respect to the Haar measure on $U(N)$; and $f_{CUE} (k) := \prod_{j=0}^{k-1} \frac{j!}{(j+k)!}$. (The fact that (\ref{2k-th moment of char poly of unitary matrices, statement}) is independent of $\theta$ is not immediately obvious, and so we remark that this lack of dependency is not an error). Now, we note that
\begin{align*}
f_{CUE} (k) 
= \begin{cases}
1 &\text{ if $k=1$} \\
\frac{1}{12} &\text{ if $k=2$} \\
\frac{42}{9!} &\text{ if $k=3$} \\
\frac{24024}{16!} &\text{ if $k=4$} .
\end{cases}
\end{align*}
That is, $f_{CUE} (k) $ agrees with the established values of $f(k)$, as well as the values that have been conjectured by alternative means. This provides support for the conjecture that $f(k) = f_{CUE} (k)$ for all integers $k \geq 0$. We remark that the results of Keating and Snaith apply to values of $k$ that are not necessarily integer-valued; however, in this paper we are concerned with the integer case. \\

Note that this conjecture, using random matrix theory, does not introduce the factor $a(k)$ in (\ref{conjecture of 2k-th moment of RZF on crit line, statement}) in any natural way. This was addressed by Gonek, Hughes, and Keating \cite{HybridEulerHadProdRZF_GonekHughesKeating} who expressed $\zeta (s)$ as a hybrid Euler-Hadamard product: $\zeta (s) \approx P_X (s) Z_X (s)$, where $P_X (s)$ is a roughly a partial Euler product and $Z_X (s)$ is roughly a partial Hadamard product (a product over the zeros of $\zeta (s)$). The variable $X$ determines the contribution of each factor. They conjectured that, asymptotically, the $2k$-th moment of $\zeta (s)$ on the critical line can be factored into the $2k$-th moment of $P_X (s)$ multiplied by the $2k$-th moment of $Z_X (s)$ (known as the splitting conjecture); and they showed that the former contributes the factor $a(k)$ in (\ref{conjecture of 2k-th moment of RZF on crit line, statement}) and conjectured via random matrix theory that the latter contributes the fator $f(k)$. That is, they obtained a conjecture for the $2k$-th moment of $\zeta (s)$ in a way that the factor $a(k)$ appears naturally. They also lent support for the splitting conjecture by demonstrating that it holds for the cases $k=1,2$. \\

This approach, using an Euler-Hadamard hybrid formula, has been applied to discrete moments of the derivative of the Riemann zeta-function by Bui, Gonek, and Milinovich \cite{BuiGonekMilinovich2015_HybrEulerHadaProdMomDeivZetaRho}. \\

The relationship between random matrix theory and the Riemann zeta-function extends to other $L$-functions, particularly certain families of $L$-functions \cite{ZeroesZetaFuncSym_KatzSarnak}. For example, one aspect of the relationship is that the proportion of $L$-functions of a certain family with conductor $q$ that have $j$-th zero in some interval $[a,b]$ appears to be the same as the proportion of matrices of a certain matrix ensemble (the precise ensemble is dependent on the family) of size $N \times N$ ($N = N(q)$ is chosen so that the mean spacing of the eigenvalues is the same as the mean spacing of the zeros of the $L$-functions of conductor $q$) that have $j$-th eigenvalue in $[a,b]$. At least, this appears to be the case as $q \longrightarrow \infty$. \\

Let us consider the family of Dirichlet $L$-functions. The associated ensemble of matrices is the unitary matrices \cite[page 887]{ConreyFarmer2000_MeanValLFuncSymm}. By making use of this relationship, Bui and Keating \cite{MeanValDLF_BuiKeating} obtained an analogue of \cite{HybridEulerHadProdRZF_GonekHughesKeating} where they considered the $2k$-th moment of Dirichlet $L$-functions at $s=\frac{1}{2}$, averaged over all primitive Dirichlet $L$-functions of modulus $q$, instead of the Riemann zeta-function averaged over the critical line. That is, using a hybrid Euler-Hadamard product for the Dirichlet $L$-functions, they conjectured (among other results) that
\begin{align} \label{statement, NF DLF moments RMT conjecture}
\frac{1}{\phi^* (q)} \sumstar_{\chi \modulus q} \Big\lvert L \Big( \frac{1}{2} , \chi \Big) \Big\rvert^{2k}
\sim a(k) \frac{ G^2 (k+1)}{G(2k+1)} \prod_{p \mid q} \bigg( \sum_{m=0}^{\infty} \frac{d_k (p^m)^2 }{p^m} \bigg)^{-1} (\log q)^{k^2} 
\end{align}
as $\degree q \longrightarrow \infty$, where $\phi^* (q)$ is the number of primitive Dirichlet characters of modulus $q$, the star in the sum indicates the sum is over primitive characters only, and $G(z)$ is the Barnes $G$-function. This had been conjectured previously (see \cite{RMTLFuncSHalf_KeatingSnaith2000}), but this approach allows for all the factors to appear naturally. \\

One can consider the above problems in the function field setting. In fact, it is the function field analogues that give some insight into the relationship between random matrix theory and $L$-functions (be they in number fields or function fields). See \cite[Section 3]{ZeroesZetaFuncSym_KatzSarnak} for details. In function fields, Bui and Florea \cite{BuiFlorea2018_HybridEulerHadaProdQuadrDLFFuncFields} developed the hybrid Euler-Hadamard product model for the family of quadratic Dirichlet $L$-functions. In this paper we do the same for Dirichlet $L$-functions of any primitive character, which is the function field analogue of the work of Bui and Keating described above. The aim is to provide support for the following conjecture (see \cite[page 887]{ConreyFarmer2000_MeanValLFuncSymm}), which is the analogue of (\ref{statement, NF DLF moments RMT conjecture}), in such a way that all factors appear naturally: 

\begin{conjecture} \label{conjecture, FF DLF moments RMT conjecture}
For all non-negative integers $k$, it is conjectured that
\begin{align*} 
\frac{1}{\phi^* (R)}  \sumstar_{\chi \modulus R} \Big\lvert L \Big( \frac{1}{2} , \chi \Big) \Big\rvert^{2k}
\sim f(k) a (k) \prod_{P \mid R} \bigg( \sum_{m=0}^{\infty} \frac{ d_k (P^m )^2}{\lvert P \rvert^m } \bigg)^{-1} (\degree R )^{k^2} ,
\end{align*}
as $\degree R\longrightarrow \infty$, where
\begin{align*}
a (k)
:= \prod_{ P \in \mathcal{P} } \Bigg( \bigg( 1 - \frac{1}{\lvert P \rvert} \bigg)^{k^2} \sum_{m=0}^{\infty} \frac{ d_k \big( P^m \big)^2 }{\lvert P \rvert^m } \Bigg) .
\end{align*}
and
\begin{align*}
f(k)
:= \frac{G^2 (k+1)}{G(2k+1)}
= \prod_{i=0}^{k-1} \frac{i!}{(i+k)!} ,
\end{align*}
where $G$ is the Barnes $G$-function. 
\end{conjecture}

This conjecture has been verified for the cases $k=1,2$ by Andrade and Yiasemides \cite{AndradeYiasemides2021_4thPowMeanDLFuncField_Final}:
\begin{align} \label{statement, 2nd moment DLF in FF}
\frac{1}{\phi^* (R)} \sumstar_{\chi \modulus R} \Big\lvert L \Big( \frac{1}{2}, \chi \Big) \Big\rvert^2 
\sim \frac{\phi (R)}{\lvert R \rvert} \degree R
\end{align}
as $\degree R \longrightarrow \infty$, and
\begin{align} \label{statement, 4th moment DLF in FF}
\frac{1}{\phi^* (R)} \sumstar_{\chi \modulus R} \Big\lvert L \Big( \frac{1}{2} , \chi \Big) \Big\lvert^4 
\sim \frac{1-q^{-1}}{12} \prod_{P \mid R} \bigg( \frac{ \big( 1 - \lvert P \rvert^{-1} \big)^3}{1 + \lvert P \rvert^{-1}} \bigg) (\degree R)^4
\end{align}
as $\degree R \longrightarrow \infty$. \\

We now state our results, but we refer the unfamiliar to reader to the beginning of Section \ref{section, notation and background} for definitions and notational remarks relating to $L$-functions in function fields. In particular,  henceforth, the letter $q$ is reserved for the order of the finite field $\mathbb{F}_q$. We begin with the Euler-Hadamard hybrid formula, which we prove in Section \ref{section, hybrid E-H formula}.

\begin{theorem} \label{theorem, L(s, chi) as hybrid E-H product, Intro version}
Let $X \geq 1$ be an integer and let $u(x)$ be a positive $C^{\infty}$-function with support in $[e , e^{1 + q^{-X} }]$. Let
\begin{align*}
v(x)
= \int_{t=x}^{\infty} u(t) \mathrm{d} t
\end{align*}
and take $u$ to be normalised so that $v(0) = 1$. Furthermore, for $y \in \mathbb{C} \backslash \{ 0 \}$ with $\arg (y) \neq \pi$, we define $E_1 (y) := \int_{w=y}^{y+ \infty} \frac{e^{-w}}{w} \mathrm{d} w$; and for $z \in \mathbb{C} \backslash \{ 0 \}$ with $\arg (z) \neq \pi$, we define
\begin{align*}
U (z) := \int_{x=0}^{\infty} u(x) E_1 (z \log x ) \mathrm{d} x .
\end{align*}
Let $\chi$ be a primitive Dirichlet character of modulus $R \in \mathcal{M} \backslash \{ 1 \}$, and let $\rho_n = \frac{1}{2} + i \gamma_n$ be the $n$-th zero of $L (s , \chi )$. Then, for all $s \in \mathbb{C}$ we have
\begin{align} \label{statement, dis of res, EH hybrid formula, L=PZ}
L (s , \chi )
= P_{X} (s , \chi) Z_{X} (s , \chi) ,
\end{align}
where
\begin{align*}
P_{X} (s , \chi)
= \exp \bigg( \sum_{\substack{A \in \mathcal{M} \\ \degree A \leq X }} \frac{ \chi (A) \Lambda (A)}{ \lvert A \rvert^{s} \log \lvert A \rvert} \bigg) 
\end{align*}
and
\begin{align*}
Z_{X} (s , \chi)
= \exp \bigg( - \sum_{\rho_n} U \Big( (s - \rho_n ) (\log q) X \Big) \bigg) .
\end{align*}
Strictly speaking, if $s = \rho$ or $\arg (s - \rho ) = \pi$ for some zero $\rho$ of $L (s , \chi)$, then $Z_{X} (s , \chi)$ is not well defined. In this case, we take
\begin{align*}
Z_{X} (s , \chi)
= \lim_{s_0 \longrightarrow s} Z_{X} (s_0 , \chi)
\end{align*}
and we show that this is well defined.
\end{theorem}

\begin{remark}
We note that our hybrid Euler-Hadamard product formula, (\ref{statement, dis of res, EH hybrid formula, L=PZ}), does not involve an error term, unlike the analogous Theorem 1 in \cite{HybridEulerHadProdRZF_GonekHughesKeating} and Theorem 1 in \cite{MeanValDLF_BuiKeating}. This is due to the fact that we are working in the function field setting. \\

We also note that $Z_{X} (s , \chi)$ is expressed in terms of $u(x)$. Whereas, $P_{X} (s , \chi)$ and $L (s , \chi)$ are independent of $u(x)$. Thus, given the equality (\ref{statement, dis of res, EH hybrid formula, L=PZ}), we can see that, as long as $u(x)$ satisfies the conditions in the theorem, the value of $Z_{X} (s , \chi)$ is independent of any further restrictions made on $u(x)$. Ultimately, this is due to the fact that we are working in the function field setting and due to our choice of support for $u(x)$. Indeed, this is why our support for $u(x)$ is not quite the exact analogy to the support of $u(x)$ in Theorem 1 of \cite{MeanValDLF_BuiKeating}. We note that in Theorem 1 in \cite{MeanValDLF_BuiKeating}, $P_{X} (s , \chi)$ and $L (s , \chi)$ also do not depend on $u(x)$, but this is because the dependency exists in the error term.
\end{remark}

We conjecture that the $2k$-th moment of the $L$-functions can be split into the $2k$-th moment of their partial Euler products multiplied by $2k$-th moment of their partial Hadamard products:

\begin{conjecture}[Splitting Conjecture] \label{splitting conjecture for DLF in FF, Intro version}
For integers $k \geq 0$, we have
\begin{align*}
&\frac{1}{\phi^* (R)} \sumstar_{\chi \modulus R} \Big\lvert L \Big( \frac{1}{2} , \chi \Big) \Big\rvert^{2k }\\
\sim &\bigg( \frac{1}{\phi^* (R)} \sumstar_{\chi \modulus R} \Big\rvert P_X \Big( \frac{1}{2} , \chi \Big) \Big\lvert ^{2k} \bigg)
	\cdot \bigg( \frac{1}{\phi^* (R)} \sumstar_{\chi \modulus R} \Big\lvert Z_{X} \Big( \frac{1}{2} , \chi \Big) \Big\rvert^{2k} \bigg)
\end{align*}
as $X , \degree R \longrightarrow \infty$ with $X \leq \log_q \degree R$.
\end{conjecture}

We then obtain the $2k$-th moment of the partial Euler products in Section \ref{section, Euler product moments}, and we use a random matrix theory model to conjecture the $2k$-th moment of the Hadamard products in Section \ref{section, Hadamard moments conjecture}:

\begin{theorem}\label{theorem, Euler product 2k-th moment, Intro version}
For positive integers $k$, we have
\begin{align*}
\frac{1}{\phi^* (R)} \sumstar_{\chi \modulus R} \Big\rvert P_X \Big( \frac{1}{2} , \chi \Big) \Big\lvert ^{2k} 
\sim a(k) \Bigg[ \prod_{\substack{ \degree P \leq X \\ P \mid R }} \Bigg( \sum_{m=0}^{\infty} \frac{ d_k \big( P^m \big)^2 }{\lvert P \rvert^m }  \Bigg)^{-1} \Bigg] \Big( e^{\gamma} X \Big)^{k^2}
\end{align*}
as $X , \degree R \longrightarrow \infty$ with $X \leq \log_q \degree R$. Here, $\gamma$ is the Euler-Mascheroni constant, and 
\begin{align*}
a (k)
= \prod_{ P \in \mathcal{P} } \Bigg( \bigg( 1 - \frac{1}{\lvert P \rvert} \bigg)^{k^2} \sum_{m=0}^{\infty} \frac{ d_k \big( P^m \big)^2 }{\lvert P \rvert^m } \Bigg) .
\end{align*}
\end{theorem}

\begin{conjecture}\label{Hadamard moment conjecture, Intro version}
For integers $k \geq 0$, we have
\begin{align*}
\frac{1}{\phi^* (R)} \sumstar_{\chi \modulus R} \Big\lvert Z_{X} \Big( \frac{1}{2} , \chi \Big) \Big\rvert^{2k}
\sim \frac{G^2 (k+1)}{G (2k+1)} \bigg( \frac{\degree R}{e^{\gamma} X} \bigg)^{k^2} ,
\end{align*}
as $\degree R \longrightarrow \infty$, where $\gamma$ is the Euler-Mascheroni constant and $G$ is the Barnes $G$-function. For our purposes, it suffices to note that
\begin{align*}
\frac{G^2 (k+1)}{G (2k+1)}
= \prod_{i=0}^{k-1} \frac{i !}{(i+k)!} .
\end{align*}
\end{conjecture}

Note that Conjecture \ref{splitting conjecture for DLF in FF, Intro version}, Theorem \ref{theorem, Euler product 2k-th moment, Intro version}, and Conjecture \ref{Hadamard moment conjecture, Intro version} together reproduce Conjecture \ref{conjecture, FF DLF moments RMT conjecture} as desired, but only for certain cases, such as when the largest prime divisor of $R$ has degree less than $X$, or when $P$ is prime. \\

In Section \ref{section, second twisted moment} we rigorously obtain the second moment of the Hadamard product:

\begin{theorem}\label{theorem, second twisted moment, Intro version}
We have that
\begin{align*}
\frac{1}{\phi^* (R)} \sumstar_{\chi \modulus R} \Big\lvert Z_{X} \Big( \frac{1}{2} , \chi \Big) \Big\rvert^{2}
= &\frac{1}{\phi^* (R)} \sumstar_{\chi \modulus R} \Big\lvert L \Big( \frac{1}{2} , \chi \Big) P_X \Big( \frac{1}{2} , \chi \Big)^{-1} \Big\rvert^2 \\
\sim &\frac{\degree R}{e^{\gamma} X} \prod_{\substack{\degree P > X \\ P \mid R}} \bigg( 1 - \frac{1}{\lvert P \rvert} \bigg) 
\end{align*}
as $X , \degree R \longrightarrow \infty$ with $X \leq \log_q \degree R$.
\end{theorem}

In Section \ref{section, fourth twisted moment} we rigorously obtain the fourth moment of the Hadamard product:

\begin{theorem} \label{Theorem, fourth moment of Hadamard Product, Intro version}
We have
\begin{align*}
\frac{1}{\phi^* (R)} \sumstar_{\chi \modulus R} \Big\lvert Z_X \Big( \frac{1}{2} , \chi \Big) \Big\rvert^4 
= &\frac{1}{\phi^* (R)} \sumstar_{\chi \modulus R} \Big\lvert L \Big( \frac{1}{2} , \chi \Big) P_X \Big( \frac{1}{2} , \chi \Big)^{-1} \Big\rvert^4 \\
\sim &\frac{1}{12} \Big( \frac{\degree R }{ e^{\gamma} X } \Big)^4
	 \prod_{\substack{\degree P > X \\ P \mid R}} \frac{\big(1 - \lvert P \rvert^{-1} \big)^3}{1 + \lvert P \rvert^{-1}} 
\end{align*}
as $X , \degree R \longrightarrow \infty$ with $X \leq \log_q \log \degree R$.
\end{theorem}

We can see that Theorems \ref{theorem, second twisted moment, Intro version} and \ref{theorem, Euler product 2k-th moment, Intro version}, and (\ref{statement, 2nd moment DLF in FF}) verify the Splitting Conjecture for the case $k=1$. This can be seen from the fact that $a(1) = 1$ and
\begin{align*}
\prod_{\substack{ \degree P \leq X \\ P \mid R }} \Bigg( \sum_{m=0}^{\infty} \frac{ d_1 \big( P^m \big)^2 }{\lvert P \rvert^m }  \Bigg)^{-1}
= \prod_{\substack{ \degree P \leq X \\ P \mid R }} \bigg( 1 - \frac{1}{\lvert P \rvert} \bigg) .
\end{align*}

We can also see that Theorems \ref{Theorem, fourth moment of Hadamard Product, Intro version} and \ref{theorem, Euler product 2k-th moment, Intro version}, and (\ref{statement, 4th moment DLF in FF}) verify the Splitting Conjecture for the case $k=2$. This can be seen from the fact that $a (2) = 1-q^{-1}$ and
\begin{align*}
\prod_{\substack{ \degree P \leq X \\ P \mid R }} \Bigg( \sum_{m=0}^{\infty} \frac{ d_2 \big( P^m \big)^2 }{\lvert P \rvert^m }  \Bigg)^{-1} 
= \prod_{\substack{ \degree P \leq X \\ P \mid R }} \frac{\big(1 - \lvert P \rvert^{-1} \big)^3}{1 + \lvert P \rvert^{-1}} .
\end{align*}
However, in Theorem \ref{Theorem, fourth moment of Hadamard Product, Intro version} we required the condition $X \leq \log_q \log \degree R$ which is more restrictive than the condition $X \leq \log_q \degree R$ in the Splitting Conjecture. However, given the results that have been establish in the area of twisted moments (see, for example, \cite{GaggeroJara1997_PhDThesis, HughesYoung2010_Twist4thMomRZF, BettinEtAl2020_QuadrDivProbMomRZF, Motohashi2009_RZFHeckeCongrSubgroupII} for $\zeta (s)$ and \cite{Hough2016_AngleLargeValLFunc, Zacharias2019_Moll4thMomDLF} for Dirichlet $L$-functions), we expect that one can improve upon this restriction for Theorem \ref{Theorem, fourth moment of Hadamard Product, Intro version}. \\


\section{Notation and Background} \label{section, notation and background}

Let $q$ be a prime power in $\mathbb{N}$ and define $\mathcal{A} := \mathbb{F}_q [T]$, the polynomial ring over the finite field of order $q$. We define $\mathcal{M} \subseteq \mathcal{A}$ to be the set of monic polynomials, and we define $\mathcal{P}$ to be the set of monic primes. Henceforth, ``prime" shall mean ``monic prime", and the upper-case letter $P$ is reserved for primes even when it is not explicitly stated. In the limits of summations and products, unless otherwise stated, the polynomials appearing therein should be taken to be monic. We define 
\begin{align*}
\mathcal{S} (X) := &\{ A \in \mathcal{A} : P \mid A \Rightarrow \degree P \leq X \} , \\
\mathcal{S}_{\mathcal{M}} (X) := &\{ A \in \mathcal{M} : P \mid A \Rightarrow \degree P \leq X \};
\end{align*}
and for $\mathcal{B} \subseteq \mathcal{A}$ and integers $n \geq 0$, we define 
\begin{align*}
\mathcal{B}_n := \{ B \in \mathcal{B} : \degree B = n \} .
\end{align*}
For $A \in \mathcal{A} \backslash \{ 0 \}$ we define $\lvert A \rvert := q^{\degree A}$, and for the zero polynomial we define $\lvert 0 \rvert := 0$. For $A,B \in \mathcal{A}$ we define $(A,B)$ to be the greatest common (monic) divisor of $A$ and $B$, and $[A,B]$ is the lowest common (monic) multiple. Suppose $A \in \mathcal{A}$ has prime factorisation $A = {P_1}^{e_1} \ldots {P_n}^{e_n}$, then we define the radical of $A$ by $\rad (A) := P_1 \ldots P_n$. \\

As usual, we write $f(x) \sim g(x)$ if $\frac{f (x)}{g (x)} \longrightarrow 1$ as $x \longrightarrow \infty$, and we write $f (x) = o \big( g (x) \big)$ if $\frac{f (x)}{g (x)} \longrightarrow 0$ as $x \longrightarrow \infty$. We write $f (x) \ll g (x)$ or $f (x) = O \big(g (x) \big)$ if there is some constant $c$ such that for all $x$ in the domain of $f$ we have $\lvert f (x) \rvert \leq c \lvert g (x) \rvert$. It may be the case that $f$, and perhaps $g$, are dependent on some parameter $k$. For example, we will often have functions where the order of our finite field, $q$, appears as a constant. If the implied constant above depends on the parameter $k$, then we write $f (x) \ll_k g (x)$ or $f (x) = O_k \big(g (x) \big)$. Otherwise, it is to be understood that the implied constant is independent of $k$. If $\lvert f (x) \rvert \leq c \lvert g (x) \rvert$ for all $x$ greater than some constant $d = d(k)$, then we write $f (x) \ll g (x)$ as $x \overset{k}{\longrightarrow} \infty$, or $f (x) = O \big(g (x) \big)$ as $x \overset{k}{\longrightarrow} \infty$. \\

Let $a \in \mathbb{C}$ and $b \in \mathbb{C} \backslash \{ 0 \}$, and let $f$ be an integrable complex function. The integral $\int_{t=a}^{a + b \infty} f(t) \mathrm{d} t$ is defined to be over the straight line starting at $a$ and in the direction of $b$. That is, $\int_{t=a}^{a + b \infty} f(t) \mathrm{d} t = \int_{s=0}^{\infty} f \big( a + \frac{b}{\lvert b \rvert} s \big) \mathrm{d} s$. If $a=0$ then we will simply write $\int_{t=0}^{b \infty} f(t) \mathrm{d} t$, and if $b= \pm 1$ then we will write $\int_{t=a}^{a \pm \infty} f(t) \mathrm{d} t$. \\

Now we state some standard definitions and results. The prime polynomial theorem tells us that
\begin{align} \label{prime poly theorem, precise, statement}
\lvert \mathcal{P}_n \rvert
= \frac{1}{n} \sum_{d \mid n} \mu (d) q^{\frac{n}{d}} ,
\end{align}
which implies
\begin{align*}
\lvert \mathcal{P}_n \rvert
= \frac{q^n}{n} + O \Big( \frac{q^{\frac{n}{2}}}{n} \Big) .
\end{align*}

The Riemann zeta-function on $\mathcal{A}$ is defined, for $\Re (s) > 1$, by
\begin{align*}
\zeta_{\mathcal{A}} (s) 
:= \sum_{A \in \mathcal{M}} \frac{1}{\lvert A \rvert^s}
= \frac{1}{1 - q^{1-s}}.
\end{align*}
We can see that the right side provides a meromorphic continuation for $\zeta_{\mathcal{A}} (s)$ to $\mathbb{C}$. Dirichlet characters are defined similarly as in the classical case:

\begin{definition}[Dirichlet Characters]
A Dirichlet character on $\mathcal{A}$ with modulus $R \in \mathcal{M}$ is a function $\chi : \mathcal{A} \longrightarrow \mathbb{C}^*$ satisfying the following properties. For all $A,B \in \mathcal{A}$:
\begin{enumerate}
\item $\chi (AB) = \chi (A) \chi (B)$;
\item \label{Dirichlet character definition, equal mod R point} If $A \equiv B (\modulus R)$, then $\chi (A) = \chi (B)$;
\item $\chi (A) = 0$ if and only if $(A,R) \neq 1$.
\end{enumerate}
\end{definition}

As usual, $\chi_0$ represents the trivial character: $\chi_0 (A) = 1$ if $(A,R)=1$ and is zero elsewhere. A character $\chi$ is even if $\chi (a) = 1$ for all $a \in {\mathbb{F}_q}^*$, and otherwise it is odd. Now, suppose $S \mid R$. We say that $S$ is an induced modulus of $\chi$ if there exists a character $\chi_1$ of modulus $S$ such that
\begin{align*}
\chi (A)
=\begin{cases}
\chi_1 (A) &\text{ if $(A,R)=1$} \\
0 &\text{ otherwise} .
\end{cases}
\end{align*}
$\chi$ is said to be primitive if there is no induced modulus of strictly smaller degree than $R$. Otherwise, $\chi$ is said to be non-primitive. $\phi^* (R)$ denotes the number of primitive characters of modulus $R$. We denote a sum over all characters $\chi$ of modulus $R$ by $\sum_{\chi \modulus R}$, and a sum over all primitive characters $\chi$ of modulus $R$ by $\sumstar_{\chi \modulus R}$. The following two results are standard, and proofs can be found in Section 3 of \cite{AndradeYiasemides2021_4thPowMeanDLFuncField_Final}.

\begin{lemma} \label{Primitive chracter sum, mobius inversion}
Let $R \in \mathcal{M}$ and let  $A,B \in \mathcal{A}$. Then,
\begin{align*}
\sumstar_{\chi \modulus R} \chi (A) \tilde{\chi} (B) =
\begin{cases}
\sum_{\substack{EF = R \\ F \mid (A-B)}} \mu (E) \phi (F) &\text{ if $(AB,R)=1$} , \\
0 &\text{ otherwise} ;
\end{cases} 
\end{align*}
and 
\begin{align*}
\sumstar_{\substack{\chi \modulus R \\ \chi \text{ even} }} \chi (A) \conj\chi (B) =
\begin{cases}
\frac{1}{q-1} \sum_{a \in {\mathbb{F}_q}^*} \sum_{\substack{EF = R \\ F \mid (A-aB) }} \mu (E) \phi (F) &\text{ if $(AB,R)=1$} , \\
0 &\text{ otherwise} .
\end{cases} 
\end{align*}
\end{lemma}

\begin{corollary} \label{Number of primitive characters}
For all $R \in \mathcal{M}$ we have that
\begin{align*}
\phi^* (R) = \sum_{EF=R} \mu (E) \phi (F) .
\end{align*}
\end{corollary}

\begin{definition}[Dirichlet $L$-function]
Let $\chi$ be a Dirichlet character. The associated L-function, $L(s,\chi )$, is defined for $\Re (s) > 1$ by
\begin{align*}
L(s,\chi )
:= \sum_{A \in \mathcal{M}} \frac{\chi (A)}{\lvert A \rvert^s} .
\end{align*}
\end{definition}

If $\chi_0$ is the trivial Dirichlet character of modulus $R$, then
\begin{align*}
L (s , \chi_0 )
= \sum_{\substack{A \in \mathcal{M} \\ (A,R) =1 }} \frac{1}{\lvert A \rvert^s}
= \prod_{\substack{P \in \mathcal{P} \\ P \nmid R }} \frac{1}{1 - \frac{1}{\lvert P \rvert^s} }
= \prod_{P \mid R} \bigg( 1 - \frac{1}{\lvert P \rvert^s} \bigg) \zeta_{\mathcal{A}} (s) 
= \prod_{P \mid R} \bigg( 1 - \frac{1}{\lvert P \rvert^s} \bigg) \frac{1}{1 - q^{1-s}}.
\end{align*}
We can see that the far right side provides a meromorphic continuation to $\mathbb{C}$ with simple poles at $1 + \frac{2m \pi i}{\log q}$ for $m \in \mathbb{Z}$. \\

If $\chi$ is a non-trivial character of modulus R, then we have some $B \in \mathcal{A}$ with $\degree B < \degree R$ and $(B,R)=1$ satisfying $\chi (B) \neq 1$. Thus, we have
\begin{align*}
\sum_{\substack{A \in \mathcal{A} \\ \degree A < \degree R }} \chi (A)
= \sum_{\substack{A \in \mathcal{A} \\ \degree A < \degree R }} \chi (AB)
= \chi (B) \sum_{\substack{A \in \mathcal{A} \\ \degree A < \degree R }} \chi (A) ,
\end{align*}
and so we must have
\begin{align*}
\sum_{\substack{A \in \mathcal{A} \\ \degree A < \degree R }} \chi (A)
= 0.
\end{align*}
This leads to
\begin{align*}
L (s , \chi )
= \sum_{A \in \mathcal{M}} \frac{\chi (A)}{\lvert A \rvert^s}
= \sum_{\substack{A \in \mathcal{M} \\ \degree A < \degree R }} \frac{\chi (A)}{\lvert A \rvert^s} .
\end{align*}
Thus, we have a finite polynomial in $q^{-s}$ which provides a holomorphic continuation to $\mathbb{C}$. The Riemann hypothesis for these $L$-functions has been proved in this setting, and so we have that all zeros lie on the critical line. Thus, we can order them and write the $n$-th zero as $\gamma_n = \frac{1}{2} + i \rho_n$ for some $\rho_n \in \mathbb{R}$. Clearly, they are vertically periodic with period $\frac{2 \pi}{\log q}$. \\

\begin{lemma} \label{Lemma, short sum for squared L-function}
Let $\chi$ a primitive character of modulus $R \neq 1$. Then, 
\begin{align*}
\Big\lvert L \Big( \frac{1}{2} , \chi \Big) \Big\rvert^2
= 2 \sum_{\substack{A,B \in \mathcal{M} \\ \degree AB < \degree R}} \frac{\chi (A) \conj\chi (B)}{\lvert AB \rvert^{\frac{1}{2}}}
+ c (\chi ) ,
\end{align*}
where, if $\chi$ is odd, we define
\begin{align*}
c (\chi )
:= - \sum_{\substack{A,B \in \mathcal{M} \\ \degree AB = \degree R -1}} \frac{\chi (A) \conj\chi (B)}{\lvert AB \rvert^{\frac{1}{2}}},
\end{align*}
and if $\chi$ is even we define 
\begin{align*}
c (\chi )
:= &-\frac{q}{\big( q^{\frac{1}{2}}-1 \big)^2} \sum_{\substack{A,B \in \mathcal{M} \\ \degree AB = \degree R -2}} \frac{\chi (A) \conj\chi (B)}{\lvert AB \rvert^{\frac{1}{2}}}
-\frac{2q^{\frac{1}{2}}}{ q^{\frac{1}{2}}-1} \sum_{\substack{A,B \in \mathcal{M} \\ \degree AB = \degree R -1}} \frac{\chi (A) \conj\chi (B)}{\lvert AB \rvert^{\frac{1}{2}}} \\
&+ \frac{1}{\big( q^{\frac{1}{2}}-1 \big)^2} \sum_{\substack{A,B \in \mathcal{M} \\ \degree AB = \degree R }} \frac{\chi (A) \conj\chi (B)}{\lvert AB \rvert^{\frac{1}{2}}} .
\end{align*}
\end{lemma}

\begin{proof}
See Lemmas 3.10 and 3.11 in \cite{AndradeYiasemides2021_4thPowMeanDLFuncField_Final}.
\end{proof}

\begin{lemma} \label{Sum over (A,R)=1, deg A <= x of 1/A}
Let $R \in \mathcal{M}$ and let $x$ be a positive integer. Then,
\begin{align*}
\sum_{\substack{ A \in \mathcal{M} \\ \degree A \leq x \\ (A , R) = 1}} \frac{1}{\lvert A \rvert}
= \begin{cases}
\frac{\phi (R)}{\lvert R \rvert} x + O \Big( \frac{\phi (R)}{\lvert R \rvert} \log \omega (R) \Big)  &\text{ if $x \geq \degree R$} \\
\frac{\phi (R)}{\lvert R \rvert} x + O \Big( \frac{\phi (R)}{\lvert R \rvert} \log \omega (R) \Big) + O \Big( \frac{2^{\omega (R)} x}{q^x} \Big) &\text{ if $x < \degree R$}
\end{cases} .
\end{align*}
\end{lemma}

\begin{proof}
See Lemma 4.12 in \cite{AndradeYiasemides2021_4thPowMeanDLFuncField_Final}. This result is slightly stronger, but the proof is identical.
\end{proof}

\begin{corollary} \label{Sum over (A,R)=1, deg A <= a deg R of 1/A}
If $a > 0$ and $x = a \degree R$, then,
\begin{align*}
\sum_{\substack{ A \in \mathcal{M} \\ \degree A \leq x \\ (A , R) = 1}} \frac{1}{\lvert A \rvert}
= \frac{\phi (R)}{\lvert R \rvert} x + O_a \Big( \frac{ \phi (R)}{\lvert R \rvert} \log \omega (R) \Big) .
\end{align*}
If $b > 2$ and $x = \log_q b^{\omega (R)}$, then
\begin{align*}
\sum_{\substack{ A \in \mathcal{M} \\ \degree A \leq x \\ (A , R) = 1}} \frac{1}{\lvert A \rvert}
= \frac{\phi (R)}{\lvert R \rvert} x + O_b \Big( \frac{ \phi (R)}{\lvert R \rvert} \log \omega (R) \Big) .
\end{align*}
\end{corollary}

\begin{proof}
First consider the case where $x = a \degree R$. If $q > e^{\frac{4 \log 2}{a}}$, then
\begin{align*}
\frac{2^{\omega (R)} x}{q^x}
\ll \frac{2^{\omega (R)} }{q^\frac{x}{2}}
\leq q^{\frac{\log 2}{\log q} \degree R - \frac{a}{2} \degree R}
< q^{- \frac{a}{4} \degree R}
\ll_a \frac{\phi (R)}{\lvert R \rvert} .
\end{align*}
If $q \leq e^{\frac{4 \log 2}{a}}$, then
\begin{align*}
\frac{2^{\omega (R)} x}{q^x}
\ll \frac{2^{\omega (R)} }{q^\frac{x}{2}}
= q^{O \Big( \frac{\degree R}{\log \degree R} \Big) - \frac{a}{2} \degree R}
\leq q^{- \frac{a}{4} \degree R}
\ll_a \frac{\phi (R)}{\lvert R \rvert} ,
\end{align*}
where the second relation holds for $\degree R > c_a$, where $c_a$ is some constant that is dependent on $a$, but independent of $q$. Finally, there are only a finite number of cases where $q \leq e^{\frac{4 \log 2}{a}}$ and $\degree R \leq c_a$, and so
\begin{align*}
\frac{2^{\omega (R)} x}{q^x}
\ll_a \frac{\phi (R)}{\lvert R \rvert}
\end{align*}
for these cases too. The proof follows from Lemma \ref{Sum over (A,R)=1, deg A <= x of 1/A}. \\

Now consider the case where $x = \log_q b^{\omega (R)}$. We have that
\begin{align*}
\frac{2^{\omega (R)} x}{q^x}
= &\frac{2^{\omega (R)} (\log_q b) \omega (R)}{b^{\omega (R)}}
\ll_b \frac{2^{\omega (R)}}{\Big( \frac{b+2}{2} \Big)^{\omega (R)}}
= \Big( \frac{4}{b+2} \Big)^{\omega (R)} \\
= &\prod_{P \mid R} \bigg( \frac{4}{b+2} \bigg)
\ll_b \prod_{P \mid R} \bigg( 1 - \frac{1}{\lvert P \rvert} \bigg)
\ll_b \frac{\phi (R)}{\lvert R \rvert} .
\end{align*}
Again, the proof follows from Lemma \ref{Sum over (A,R)=1, deg A <= x of 1/A}.
\end{proof}


\section{The Hybrid Euler-Hadamard Product Formula} \label{section, hybrid E-H formula}

Before proving Theorem \ref{theorem, L(s, chi) as hybrid E-H product, Intro version}, we prove several lemmas.

\begin{lemma}
For all Dirichlet characters $\chi$ and all $\Re (s) > 1$ we have
\begin{align*}
- \frac{L'}{L} (s, \chi )
= \sum_{A \in \mathcal{M}} \frac{ \chi (A) \Lambda (A)}{\lvert A \rvert^s} .
\end{align*}
\end{lemma}

\begin{proof}
Taking the logarithmic derivative of 
\begin{align*}
L (s , \chi )
= \prod_{P \in \mathcal{P}} \bigg( 1 - \frac{\chi (P)}{\lvert P \rvert^{s}} \bigg)^{-1}
\end{align*}
gives
\begin{align*}
\frac{L'}{L} (s , \chi )
= - \sum_{P \in \mathcal{P}} \frac{\chi (P) \log \lvert P \rvert}{\lvert P \rvert^{s}} \bigg( 1 - \frac{\chi (P)}{\lvert P \rvert^{s}} \bigg)^{-1}
= - \sum_{A \in \mathcal{M}} \frac{ \chi (A) \Lambda (A)}{\lvert A \rvert^s} .
\end{align*}
\end{proof}

\begin{lemma} \label{boundedness of L'/L on left half plane}
Let $\chi$ be a non-trivial character. As $\Re (s) \longrightarrow \infty$,
\begin{align*}
\frac{L'}{L} (-s, \chi ) = O_{\chi} (1) .
\end{align*}
\end{lemma}

\begin{proof}
As $\chi$ is non-trivial, there is some maximal integer $N \geq 0$ with $L_N ( \chi ) \neq 0$. Hence,
\begin{align*}
L (-s, \chi )
= \sum_{n=0}^{N} L_n (\chi ) q^{ns}
\gg_{\chi} q^{N \Re (s)} 
\end{align*}
and
\begin{align*}
L' (-s, \chi )
= - \log q \sum_{n=0}^{N} n L_n (\chi ) q^{ns}
\ll_{\chi} q^{N \Re (s)} .
\end{align*}
The proof follows.
\end{proof}

\begin{lemma} \label{lemma, tilde(u) (s) << 1/s bound}
Let $X$ be a positive integer, and let $u(x)$ be a positive $C^{\infty}$-function with support in $[e , e^{1+q^{-X}}]$. Let $\tilde{u} (s)$ be its Mellin transform. That is,
\begin{align*}
\tilde{u} (s)
= \int_{x=0}^{\infty} x^{s-1} u (x) \mathrm{d} x
\end{align*}
and
\begin{align*}
u (x)
= \frac{1}{2 \pi i} \int_{\Re (s) = c} x^{-s} \tilde{u} (s) \mathrm{d} s ,
\end{align*}
where $c$ can take any value in $\mathbb{R}$ (due to our restrictions on the support of $u$, we can see that $\tilde{u} (s)$ is well-defined for all $s \in \mathbb{C}$, and so, by the Mellin inversion theorem, $c$ can take any value in $\mathbb{R}$). Then,
\begin{align*}
\tilde{u} (s)
\ll \begin{cases}
\frac{1}{\lvert s \rvert + 1} \max_x \{ \lvert u' (x) \rvert \} e^{2 \Re (s) } &\text{ if $\Re (s) > 0$} \\
\frac{1}{\lvert s \rvert + 1} \max_x \{ \lvert u' (x) \rvert \} e^{\Re (s) } &\text{ if $\Re (s) \leq 0$} .
\end{cases}
\end{align*}
\end{lemma}

\begin{proof}
We have, by integration by parts, that
\begin{align*}
\tilde{u} (s)
= \int_{x=e}^{e^{1+q^{-X}} } x^{s-1} u (x) \mathrm{d} x
= -\frac{1}{s} \int_{x=e}^{e^{1+q^{-X}} } x^{s} u' (x) \mathrm{d} x.
\end{align*}
If $\lvert s \rvert > 1$, then it is not difficult to deduce that the above is
\begin{align*}
\ll &\begin{cases}
\frac{1}{\lvert s \rvert + 1} \max_x \{ \lvert u' (x) \rvert \} e^{2 \Re (s) } &\text{ if $\Re (s) > 0$} \\
\frac{1}{\lvert s \rvert + 1} \max_x \{ \lvert u' (x) \rvert \} e^{\Re (s) } &\text{ if $\Re (s) \leq 0$} .
\end{cases}
\end{align*}
If $\lvert s \rvert \leq 1$, then, by using the fact that $\int_{x=e}^{e^{1+q^{-X}} } u' (x) \mathrm{d} x = 0$, we obtain
\begin{align*}
\tilde{u} (s)
= &\int_{x=e}^{e^{1+q^{-X}} } \frac{1 - x^{s}}{s} u' (x) \mathrm{d} x
= - \int_{x=e}^{e^{1+q^{-X}} } \bigg( \int_{y=1}^{x} y^{s-1} \mathrm{d} y \bigg) u' (x) \mathrm{d} x \\
\ll &\int_{x=e}^{e^{1+q^{-X}} } \lvert u' (x) \rvert \mathrm{d} x
\ll \max_x \{ \lvert u' (x) \rvert \} ,
\end{align*}
from which the result follows.
\end{proof}

\begin{lemma} \label{L'/L (s.chi) in terms of zeros sum and von mangoldt sum}
Let $X$ be a positive integer, and let $u(x)$ be a positive $C^{\infty}$-function with support in $[e , e^{1 + q^{-X} }]$, and let $\tilde{u} (s)$ be its Mellin transform. Let
\begin{align*}
v(x)
= \int_{t=x}^{\infty} u(t) \mathrm{d} t
\end{align*}
and take $u$ to be normalised so that $v(0) = 1$. Note that its Mellin transform is
\begin{align*}
\tilde{v} (s)
= \frac{ \tilde{u} (s+1)}{s} .
\end{align*}
Let $\chi$ be a primitive Dirichlet character of modulus $R \in \mathcal{M} \backslash \{ 1 \}$. Then, for $s \in \mathbb{C}$ not being a zero of $L (s, \chi)$, we have
\begin{align} \label{L'/L (s.chi) in terms of zeros sum and von mangoldt sum, actual equation}
- \frac{L'}{L} (s, \chi )
= \sum_{\substack{A \in \mathcal{M} \\ \degree A \leq X }} \frac{ \chi (A) \Lambda (A)}{ \lvert A \rvert^s}
	+ \sum_{\rho_n} \frac{\tilde{u} \big( 1 + (\rho_n -s) (\log q) X \big)}{\rho_n - s} ,
\end{align}
where $\rho_n = \frac{1}{2} + i \gamma_n$ is the $n$-th zero of $L (s , \chi )$. Note that, by Lemma \ref{lemma, tilde(u) (s) << 1/s bound}, we can see that the sum over the zeros is absolutely convergent.
\end{lemma}

\begin{proof}
Let $c > \max \{ 0 , \big( 1- \Re(s) \big) (\log q) X \}$. By the Mellin inversion theorem, we have
\begin{align*}
\sum_{A \in \mathcal{M}} \frac{ \chi (A) \Lambda (A)}{ \lvert A \rvert^s} v \Big( e^{\frac{\degree A}{X}} \Big)
= &\frac{1}{2 \pi i} \sum_{A \in \mathcal{M}} \frac{ \chi (A) \Lambda (A)}{ \lvert A \rvert^s} \int_{\Re (w) = c} \frac{\tilde{u} (w+1)}{w} \lvert A \rvert^{- \frac{w}{(\log q) X}} \mathrm{d} w \\
= &\frac{1}{2 \pi i} \int_{\Re (w) = c} \frac{\tilde{u} (w+1)}{w} \sum_{A \in \mathcal{M}} \frac{ \chi (A) \Lambda (A)}{ \lvert A \rvert^{s+\frac{w}{(\log q) X}}} \mathrm{d} w \\
= &-\frac{1}{2 \pi i} \int_{\Re (w) = c} \frac{\tilde{u} (w+1)}{w} \frac{L'}{L} \Big( s + \frac{w}{(\log q) X}, \chi \Big) \mathrm{d} w .
\end{align*}
The interchange of integral and summation is justified by absolute convergence, which holds because $c > \big(1-\Re (s) \big) (\log q) X$ and by Lemma \ref{lemma, tilde(u) (s) << 1/s bound}. \\

We now shift the line of integration to $\Re (w) = - M$, for some \\
$M > \max \{ 0 , \Re (s) (\log q) X \}$, giving
\begin{align*}
\sum_{A \in \mathcal{M}} \frac{ \chi (A) \Lambda (A)}{ \lvert A \rvert^s} v \Big( e^{\frac{\degree A}{X}} \Big)
= &-\frac{L'}{L} ( s , \chi ) 
	- \sum_{\rho_n} \frac{\tilde{u} \big( 1 + (\rho_n -s) (\log q) X \big) }{\rho_n -s} \\
	&-\frac{1}{2 \pi i} \int_{\Re (w) = -M} \frac{\tilde{u} (w+1)}{w} \frac{L'}{L} \Big( s + \frac{w}{(\log q) X}, \chi \Big) \mathrm{d} w ,
\end{align*}
where the sum over the zeros counts multiplicities. This requires some justification. We make use of the contour that is the rectangle with vertices at 
\begin{align*}
c \pm i \Big( \big(d-\Im (s) \big) (\log q) X + 2 \pi n X \Big) , \\
-M \pm i \Big( \big(d-\Im (s) \big) (\log q) X + 2 \pi n X \Big) . 
\end{align*}
Here, $d > 0$ is such that $\frac{1}{2} + i d$ is not a pole of $ \frac{L'}{L} (s , \chi )$ (that is, not a zero of $L(s , \chi )$). It is clear that as $n \longrightarrow \infty$ we capture all the poles and the left edge tends to the integral over $\Re (w) = - M$. Due to the vertical periodicity of $\frac{L'}{L}$, and our choice of $d$, we can see that the top and bottom integrals are equal to $O_{c,M} (n^{-1})$, which vanishes as $n \longrightarrow \infty$. \\

By Lemmas \ref{boundedness of L'/L on left half plane} and \ref{lemma, tilde(u) (s) << 1/s bound}, if we let $M \longrightarrow \infty$ then we see that the integral over $\Re (w) = -M$ vanishes. \\

Finally, we note that
\begin{align*}
v \Big( e^{\frac{\degree A}{X}} \Big) 
= \begin{cases}
1 &\text{ if $\degree A \leq X$} \\
0 &\text{ if $\degree A \geq X (1+q^{-X})$} .
\end{cases}
\end{align*}
Also, since $X$ is a positive integer, there are no integers in the interval $\big( X , X (1+q^{-X}) \big) \allowbreak \subseteq {\big( X , X + \frac{1}{2} \big)}$, and so there are no $A \in \mathcal{A}$ that have degree in this interval. It follows that
\begin{align*}
\sum_{A \in \mathcal{M}} \frac{ \chi (A) \Lambda (A)}{ \lvert A \rvert^s} v \Big( e^{\frac{\log \lvert A \rvert}{(\log q) X}} \Big) 
= \sum_{\substack{A \in \mathcal{M} \\ \degree A \leq X }} \frac{ \chi (A) \Lambda (A)}{ \lvert A \rvert^s} .
\end{align*}
\end{proof}

\begin{lemma} \label{int of tilde(u) (1 + (gamma_n - s) log X)/(gamma_n - s) = - U ((s_0 - gamma_n) log X)}
Suppose $u(x)$ has support in $[e , e^{1 + q^{-X} }]$. For all $z \in \mathbb{C} \backslash \{ 0 \}$ with $\arg (z) \neq \pi$ we define
\begin{align*}
U (z) := \int_{x=0}^{\infty} u(x) E_1 (z \log x ) \mathrm{d} x .
\end{align*}
(Recall, for $y \in \mathbb{C} \backslash \{ 0 \}$ with $\arg (y) \neq \pi$, we define $E_1 (y) := \int_{w=y}^{y+ \infty} \frac{e^{-w}}{w} \mathrm{d} w$). Let $\chi$ be a primitive Dirichlet character of modulus $R \in \mathcal{M} \backslash \{ 1 \}$, and suppose $\rho$ is a zero of $L (s , \chi )$ and $s \in \mathbb{C} \backslash \{ \rho \}$ with $\arg (s -\rho ) \neq \pi$. Then,
\begin{align*}
\int_{s_0 =s}^{s + \infty} \frac{ \tilde{u} \big( 1 + (\rho - s_0 ) (\log q) X \big)}{\rho - s_0 } \mathrm{d} s_0
=  -U \Big( (s - \rho ) (\log q) X \Big) .
\end{align*}
\end{lemma}

\begin{proof}
We have
\begin{align*}
\int_{s_0 = s}^{s + \infty} \frac{ \tilde{u} \big( 1 + (\rho - s_0 ) (\log q) X \big)}{\rho - s_0 } \mathrm{d} s_0
= &\int_{s_0 = s}^{s + \infty } \frac{1}{\rho - s_0} \int_{x=0}^{\infty} x^{(\rho - s_0) (\log q) X} u(x) \mathrm{d} x \mathrm{d} s_0 \\
= &\int_{x=0}^{\infty} u(x) \int_{s_0 = s}^{s + \infty } \frac{e^{(\rho - s_0 ) (\log q) X \log x}}{\rho - s_0 }  \mathrm{d} s_0 \mathrm{d} x \\
= &-\int_{x=0}^{\infty} u(x) \int_{w= (s - \rho ) (\log q) X \log x}^{(s - \rho ) (\log q) X \log x + \infty} \frac{e^{-w}}{w}  \mathrm{d} w \mathrm{d} x \\
= &-\int_{x=0}^{\infty} u(x) E_1 \Big( (s - \rho ) (\log q) X \log x \Big) \mathrm{d} x \\
= & -U \Big( (s - \rho ) (\log q) X \Big).
\end{align*}
The interchange of integration is justified by absolute convergence, which holds for $X > 1$. 
\end{proof}

We can now proceed with the proof of Theorem \ref{theorem, L(s, chi) as hybrid E-H product, Intro version}.

\begin{proof}[Proof of Theorem \ref{theorem, L(s, chi) as hybrid E-H product, Intro version}]
Suppose $s \in \mathbb{C}$ is not a zero of $L(s,\chi )$ and $\arg (s -\rho ) \neq \pi$ for all zeros $\rho$ of $L (s , \chi)$. We recall that (\ref{L'/L (s.chi) in terms of zeros sum and von mangoldt sum, actual equation}) gives us
\begin{align*}
- \frac{L'}{L} (s_0 , \chi )
= \sum_{\substack{A \in \mathcal{M} \\ \degree A \leq X }} \frac{ \chi (A) \Lambda (A)}{ \lvert A \rvert^{s_0}}
	+ \sum_{\rho_n} \frac{\tilde{u} \big( 1 + (\rho_n -s_0 ) (\log q) X \big)}{\rho_n - s_0 } ,
\end{align*}
to which we apply the integral $\int_{s_0 = s}^{s + \infty} \mathrm{d} s_0$ to both sides to obtain
\begin{align} \label{L(s, chi) as hybrid E-H product, before exponentials}
\log L (s , \chi )
= \sum_{\substack{A \in \mathcal{M} \\ \degree A \leq X }} \frac{ \chi (A) \Lambda (A)}{ \lvert A \rvert^{s} \log \lvert A \rvert}
	- \sum_{\rho} U \Big( (s - \rho ) (\log q) X \Big) .
\end{align}
For the integral over the sum over zeros, we applied Lemma \ref{int of tilde(u) (1 + (gamma_n - s) log X)/(gamma_n - s) = - U ((s_0 - gamma_n) log X)}, after an interchange of summation and integration that is justified by Lemma \ref{lemma, tilde(u) (s) << 1/s bound}. We now take exponentials of both sides of (\ref{L(s, chi) as hybrid E-H product, before exponentials}) to obtain
\begin{align*}
L (s , \chi )
= &\exp \bigg( \sum_{\substack{A \in \mathcal{M} \\ \degree A \leq X }} \frac{ \chi (A) \Lambda (A)}{ \lvert A \rvert^{s} \log \lvert A \rvert} \bigg)
	\exp \bigg( - \sum_{\rho} U \Big( (s - \rho ) (\log q) X \Big) \bigg) \\
= &P_{X} (s , \chi) Z_{X} (s , \chi) .
\end{align*}

Now suppose we have $s \in \mathbb{C}$, not being a zero of $L(s,\chi )$, but with $\arg (s -\rho ) = \pi$ for some zero $\rho$ of $L (s , \chi )$. We can see that $\lim_{s_0 \longrightarrow s} L (s_0 , \chi ) = L (s , \chi )$ and $\lim_{s_0 \longrightarrow s} P_{X} (s_0 , \chi)  = P_{X} (s , \chi) \neq 0 $. The latter is non-zero as $P_{X} (s , \chi) $ is the exponential of a polynomial. From this, we can deduce that $\lim_{s_0 \longrightarrow s} Z_{X} (s_0 , \chi)  = L (s , \chi ) \big( P_{X} (s , \chi) \big)^{-1} \in \mathbb{C}$. Similarly, if $s$ is a zero of $L(s,\chi )$, then we can see that $\lim_{s_0 \longrightarrow s} Z_{X} (s_0 , \chi)  = L (s , \chi ) \big( P_{X} (s , \chi) \big)^{-1} = 0$. This completes the proof.
\end{proof}


\section{Moments of the Partial Euler Product} \label{section, Euler product moments}
We require the following two lemmas before proving Theorem \ref{theorem, Euler product 2k-th moment, Intro version}.

\begin{lemma} \label{lemma, P_X (1/2 , chi) in terms of P_X* (1/2 , chi)}
For all $\Re (s) > 0$ and primitive characters $\chi$ we define
\begin{align}\label{P* definition}
P_X^{*} (s, \chi )
:=\prod_{\degree P \leq X} \bigg( 1 - \frac{\chi (P)}{\lvert P \rvert^s } \bigg)^{-1}
	\prod_{\frac{X}{2} < \degree P \leq X} \bigg( 1 + \frac{\chi (P)^2}{2 \lvert P \rvert^{2s} } \bigg)^{-1} ,
\end{align}
and for positive integers $k$ and $A \in \mathcal{S}_{\mathcal{M}} (X)$ we define $\alpha_k (A)$ by 
\begin{align*}
P_X^{*} (s, \chi )^{k} 
= \sum_{A \in \mathcal{S}_{\mathcal{M}} (X) } \frac{ \alpha_k (A) \chi (A)}{\lvert A \rvert^{s} } .
\end{align*}
Then, for positive integers $k$, we have
\begin{align}
\begin{split} \label{P_X (1/2 , chi) in terms of P_X* (1/2 , chi) and sum over a_k (A), statement}
P_X \Big( \frac{1}{2} , \chi \Big)^k
= \Big( 1 + O_k \big( X^{-1} \big) \Big) P_X^{*} \Big( \frac{1}{2} , \chi \Big)^{k} 
= \Big( 1 + O_k \big( X^{-1} \big) \Big) \sum_{A \in \mathcal{S}_{\mathcal{M}} (X) } \frac{ \alpha_k (A) \chi (A)}{\lvert A \rvert^{\frac{1}{2}} } .
\end{split}
\end{align}
We also have that
\begin{align}
\begin{split} \label{P_X (1/2 , chi) in terms of P_X* (1/2 , chi), a_k (A) def}
\alpha_k (A)  = &d_k (A) \hspace{3em} \text{ if $A \in \mathcal{S}_{\mathcal{M}} \Big( \frac{X}{2} \Big)$ or $A$ is prime} \\
0 \leq \alpha_k (A)  \leq &d_k (A) \hspace{3em} \text{ if $A \not\in \mathcal{S}_{\mathcal{M}} \Big( \frac{X}{2} \Big)$ and $A$ is not prime} .
\end{split}
\end{align}
\end{lemma}

\begin{proof}
First we note that
\begin{align*}
P_X \Big( \frac{1}{2} , \chi \Big)
= \exp \bigg( \sum_{\substack{A \in \mathcal{M} \\ \degree A \leq X }} \frac{ \chi (A) \Lambda (A)}{ \lvert A \rvert^{\frac{1}{2}} \log \lvert A \rvert} \bigg)
= \exp \bigg( \sum_{\degree P \leq X} \sum_{j=1}^{N_P} \frac{\chi (P)^j}{ j \lvert P \rvert^{\frac{j}{2}} } \bigg) ,
\end{align*}
where 
\begin{align*}
N_P := \Big\lfloor \frac{X}{\degree P} \Big\rfloor.
\end{align*}
Also, by using the Taylor series for $\log$, we have
\begin{align*}
P_X^{*} \Big( \frac{1}{2} , \chi \Big)
= \exp \bigg( \sum_{\degree P \leq X} \sum_{j=1}^{\infty} \frac{\chi (P)^j}{ j \lvert P \rvert^{\frac{j}{2}} }
	+ \sum_{\frac{X}{2} < \degree P \leq X} \sum_{j=1}^{\infty} \frac{(-1)^j \chi (P)^{2j}}{ j 2^j \lvert P \rvert^{j} } \bigg) .
\end{align*}
Hence,
\begin{align*}
P_X \Big( \frac{1}{2} , \chi \Big)
	P_X^{*} \Big( \frac{1}{2} , \chi \Big)^{-1} 
= \exp \bigg( - \sum_{\degree P \leq X} \sum_{j=N_p + 1}^{\infty} \frac{\chi (P)^j}{ j \lvert P \rvert^{\frac{j}{2} }}
	- \sum_{\frac{X}{2} < \degree P \leq X} \sum_{j=1}^{\infty} \frac{(-1)^j \chi (P)^{2j}}{ j 2^j \lvert P \rvert^{j} } \bigg) .
\end{align*}
We now show that the terms inside the exponential are equal to $O \big( X^{-1} \big)$, from which we easily deduce
\begin{align*}
P_X \Big( \frac{1}{2} , \chi \Big)^k
= \Big( 1 + O_k \big( X^{-1} \big) \Big) P_X^{*} \Big( \frac{1}{2} , \chi \Big)^{k} .
\end{align*}
To this end, using the prime polynomial theorem for the last line below, we have
\begin{align}
\begin{split}\label{P^k P*^-k exponential terms bound}
& \sum_{\degree P \leq X} \sum_{j=N_P + 1}^{\infty} \frac{\chi (P)^j}{ j \lvert P \rvert^{\frac{j}{2} }}
	+  \sum_{\frac{X}{2} < \degree P \leq X} \sum_{j=1}^{\infty} \frac{(-1)^j \chi (P)^{2j}}{ j 2^j \lvert P \rvert^{j}} \\
= & \sum_{\degree P \leq \frac{X}{2}} \sum_{j=N_P + 1}^{\infty} \frac{\chi (P)^j}{ j \lvert P \rvert^{\frac{j}{2} }}
	+ \sum_{\frac{X}{2} < \degree P \leq X} \sum_{j=2}^{\infty} \frac{\chi (P)^j}{ j \lvert P \rvert^{\frac{j}{2} }}
	+ \sum_{\frac{X}{2} < \degree P \leq X} \sum_{j=1}^{\infty} \frac{(-1)^j \chi (P)^{2j}}{ j 2^j \lvert P \rvert^{j}} \\
= & \sum_{\degree P \leq \frac{X}{2}} \sum_{j=N_P + 1}^{\infty} \frac{\chi (P)^j}{ j \lvert P \rvert^{\frac{j}{2} }}
	+ \sum_{\frac{X}{2} < \degree P \leq X} \sum_{j=3}^{\infty} \frac{\chi (P)^j}{ j \lvert P \rvert^{\frac{j}{2} }}
	+ \sum_{\frac{X}{2} < \degree P \leq X} \sum_{j=2}^{\infty} \frac{(-1)^j \chi (P)^{2j}}{ j 2^j \lvert P \rvert^{j}}  \\
\ll &\sum_{\degree P \leq \frac{X}{2}} \lvert P \rvert^{-\frac{N_P +1}{2} }
	+    \sum_{\frac{X}{2} < \degree P \leq X} \lvert P \rvert^{-\frac{3}{2}} 
\ll q^{-\frac{X}{2}} \sum_{\degree P \leq \frac{X}{2}} 1
	+    \sum_{\frac{X}{2} < n \leq X} \frac{q^{-\frac{n}{2}} }{n}
\ll \frac{ 1 }{X} .
\end{split}
\end{align}

We now proceed to prove (\ref{P_X (1/2 , chi) in terms of P_X* (1/2 , chi), a_k (A) def}). The first case is clear, so assume that $A \not\in \mathcal{S}_{\mathcal{M}} \Big( \frac{X}{2} \Big)$ and $A$ is not prime. We note that
\begin{align*}
&\bigg( 1 - \frac{\chi (P)}{\lvert P \rvert^{\frac{1}{2}} } \bigg)^{-1} \bigg( 1 + \frac{\chi (P)^2}{2 \lvert P \rvert } \bigg)^{-1} \\
= &\bigg( 1 + \frac{\chi (P)}{\lvert P \rvert^{\frac{1}{2}} } + \frac{\chi (P)^2}{\vert P \rvert } + \ldots \bigg)
	\bigg( 1 - \frac{\chi (P)^2}{2 \lvert P \rvert } + \frac{\chi (P)^4}{2^2 \lvert P \rvert^2 } - \ldots \bigg) \\
= &\sum_{r=0}^{\infty} \Bigg( \sum_{\substack{r_1 , r_2 \geq 0 \\ r_1 + 2r_2 = r }} \bigg( - \frac{1}{2} \bigg)^{r_2} \Bigg) \frac{ \chi (P)^r}{\lvert P \rvert^{\frac{r}{2}}}
= \sum_{r=0}^{\infty}  \frac{2}{3} \bigg( 1 - \Big( - \frac{1}{2} \Big)^{\lfloor \frac{r}{2} \rfloor + 1} \bigg) \frac{ \chi (P)^r}{\lvert P \rvert^{\frac{r}{2}}} .
\end{align*}
Since
\begin{align*}
0 
\leq \frac{2}{3} \bigg( 1 - \Big( - \frac{1}{2} \Big)^{\lfloor \frac{r}{2} \rfloor + 1} \bigg)
 \leq 1
\end{align*}
for all $r \geq 0$, the result follows.
\end{proof}

\begin{lemma}[Mertens' Third Theorem in {$\mathbb{F}_q [T]$}] \label{lemma, Mertens 3rd Theorem in FF}
We have
\begin{align*}
\prod_{\degree P \leq n} \bigg( 1 - \frac{1}{\lvert P \rvert} \bigg)^{-1}
\sim e^{\gamma} n .
\end{align*}
\end{lemma}

\begin{proof}
The proof is very similar to that of Theorem 3 in \cite{Rosen1999_GeneralizMertensTheorem}.
\end{proof}

We can now prove Theorem \ref{theorem, Euler product 2k-th moment, Intro version}.

\begin{proof} [Proof of Theorem \ref{theorem, Euler product 2k-th moment, Intro version}]
Throughout this proof, any asymptotic relations are to be taken as $X , \degree R \overset{q,k}{\longrightarrow} \infty$ with $X \leq \log_q \degree R$. By Lemma \ref{lemma, P_X (1/2 , chi) in terms of P_X* (1/2 , chi)} it suffices to prove that
\begin{align*}
\frac{1}{\phi^* (R)} \sumstar_{\chi \modulus R} \Bigg\lvert \sum_{A \in \mathcal{S}_{\mathcal{M}} (X)} \frac{ \alpha_k (A) \chi (A)}{\lvert A \rvert^{\frac{1}{2}} } \Bigg\rvert^2
\sim a(k)  \prod_{\substack{ \degree P \leq X \\ P \mid R }} \Bigg( \sum_{m=0}^{\infty} \frac{ d_k \big( P^m \big)^2 }{\lvert P \rvert^m }  \Bigg)^{-1}  \Big( e^{\gamma} X \Big)^{k^2} .
\end{align*}
We will truncate our Dirichlet series. This will allow us to bound the lower order terms later. We have
\begin{align} \label{P_X* (1/2,chi) Dir Ser truncation, statement}
\sum_{A \in \mathcal{S}_{\mathcal{M}} (X) } \frac{ \alpha_k (A) \chi (A)}{\lvert A \rvert^{\frac{1}{2}} }
= \sum_{\substack{A \in \mathcal{S}_{\mathcal{M}} (X) \\ \degree A \leq \frac{1}{4} \degree R }} \frac{ \alpha_k (A) \chi (A)}{\lvert A \rvert^{\frac{1}{2}} }
	+ O \Big( \lvert R \rvert^{-\frac{1}{17}} \Big) .
\end{align}
This makes use of the following:
\begin{align}
\begin{split} \label{2k-th moment of Euler product, Dirichlet series truncation bound}
&\sum_{\substack{ A \in \mathcal{S}_{\mathcal{M}} (X) \\ \degree A > \frac{1}{4} \degree R }} \frac{ \alpha_k (A) \chi (A)}{\lvert A \rvert^{\frac{1}{2}} }
\leq \lvert R \rvert^{-\frac{1}{16}} \sum_{A \in \mathcal{S}_{\mathcal{M}} (X) } \frac{ d_k (A) }{\lvert A \rvert^{\frac{1}{4}} }
= \lvert R \rvert^{-\frac{1}{16}} \prod_{\degree P \leq X } \bigg( 1 - \frac{1}{\lvert P \rvert^{\frac{1}{4} }} \bigg)^{-k} \\
= &\lvert R \rvert^{-\frac{1}{16}} \exp \Bigg( \sum_{\degree P \leq X } -k \log \bigg( 1 - \frac{1}{\lvert P \rvert^{\frac{1}{4} }} \bigg) \Bigg)
= \lvert R \rvert^{-\frac{1}{16}} \exp \Bigg( k O \bigg( \sum_{\degree P \leq X } \frac{1}{\lvert P \rvert^{\frac{1}{4} }} \bigg) \Bigg) \\
= &\lvert R \rvert^{-\frac{1}{16}} \exp \Bigg( k O \bigg( \frac{q^{ \frac{3}{4} X} }{X} \bigg) \Bigg)
= \lvert R \rvert^{-\frac{1}{16}} \exp \Bigg( k O \bigg( \frac{ \degree R }{\log_q \degree R} \bigg) \Bigg)
= O \Big( \lvert R \rvert^{-\frac{1}{17}} \Big) .
\end{split}
\end{align}
By the Cauchy-Schwarz inequality, it suffices to prove that
\begin{align*}
\frac{1}{\phi^* (R)} \sumstar_{\chi \modulus R} \Bigg\lvert \sum_{\substack{A \in \mathcal{S}_{\mathcal{M}} (X) \\ \degree A \leq \frac{1}{4} \degree R }} \hspace{-0.75em} \frac{ \alpha_k (A) \chi (A)}{\lvert A \rvert^{\frac{1}{2}} } \Bigg\rvert^2
\sim a(k)  \prod_{\substack{ \degree P \leq X \\ P \mid R }} \Bigg( \sum_{m=0}^{\infty} \frac{ d_k \big( P^m \big)^2 }{\lvert P \rvert^m }  \Bigg)^{-1}  \Big( e^{\gamma} X \Big)^{k^2} .
\end{align*}
Now, we have that
\begin{align}
\begin{split} \label{2k-th moment of Euler product, expression with diag and off diag terms}
&\frac{1}{\phi^* (R)} \sumstar_{\chi \modulus R} \Bigg\lvert \sum_{\substack{ A \in \mathcal{S}_{\mathcal{M}} (X) \\ \degree A \leq \frac{1}{4} \degree R }} \frac{ \alpha_k (A) \chi (A)}{\lvert A \rvert^{\frac{1}{2}} } \Bigg\rvert^2 \\
= &\frac{1}{\phi^* (R)} \sum_{\substack{ A,B \in \mathcal{S}_{\mathcal{M}} (X) \\ \degree A, \degree B \leq \frac{1}{4} \degree R \\ (AB,R) = 1 }} \frac{ \alpha_k (A) \alpha_k (B) }{\lvert AB \rvert^{\frac{1}{2}} } \sum_{\substack{ EF=R \\ F \mid (A-B) }} \mu (E) \phi (F) \\
= &\frac{1}{\phi^* (R)} \sum_{EF=R} \mu (E) \phi (F) \sum_{\substack{ A,B \in \mathcal{S}_{\mathcal{M}} (X) \\ \degree A, \degree B \leq \frac{1}{4} \degree R \\ (AB,R) = 1 \\ A \equiv B (\modulus F)}} \frac{ \alpha_k (A) \alpha_k (B) }{\lvert AB \rvert^{\frac{1}{2}} } \\
= &\sum_{\substack{ A \in \mathcal{S}_{\mathcal{M}} (X) \\ \degree A \leq \frac{1}{4} \degree R \\ (A,R) = 1 }} \frac{ \alpha_k (A)^2 }{\lvert A \rvert }
	+ \frac{1}{\phi^* (R)} \sum_{EF=R} \mu (E) \phi (F) \sum_{\substack{ A,B \in \mathcal{S}_{\mathcal{M}} (X) \\ \degree A, \degree B \leq \frac{1}{4} \degree R \\ (AB,R) = 1 \\ A \equiv B (\modulus F) \\ A \neq B }} \frac{ \alpha_k (A) \alpha_k (B) }{\lvert AB \rvert^{\frac{1}{2}} } .
\end{split}
\end{align}
We first consider the second term on the far right side: The off-diagonal terms. We note that the inner sum is zero if $\degree F > \frac{1}{4} \degree R$, and we also make use of (\ref{P_X (1/2 , chi) in terms of P_X* (1/2 , chi), a_k (A) def}), to obtain
\begin{align*}
&\frac{1}{\phi^* (R)} \sum_{EF=R} \mu (E) \phi (F) \sum_{\substack{ A,B \in \mathcal{S}_{\mathcal{M}} (X) \\ \degree A, \degree B \leq \frac{1}{4} \degree R \\ (AB,R) = 1 \\ A \equiv B (\modulus F) \\ A \neq B }} \frac{ \alpha_k (A) \alpha_k (B) }{\lvert AB \rvert^{\frac{1}{2}} } \\
\ll &\frac{1}{\phi^* (R)} \sum_{\substack{EF=R \\ \degree F \leq \frac{1}{4} \degree R }} \phi (F) \sum_{A,B \in \mathcal{S}_{\mathcal{M}} (X) } \frac{ d_k (A) d_k (B) }{\lvert AB \rvert^{\frac{1}{2}} } \\
\leq &\frac{1}{\phi^* (R)} \prod_{\degree P \leq X} \Big( 1 - \lvert P \rvert^{-\frac{1}{2}} \Big)^{-2k} \sum_{\substack{EF=R \\ \degree F \leq \frac{1}{4} \degree R }} \phi (F) \\
\leq &\frac{1}{\phi^* (R)} \prod_{\degree P \leq X} \Big( 1 - \lvert P \rvert^{-\frac{1}{2}} \Big)^{-2k} \sum_{\substack{F \in \mathcal{M} \\ \degree F \leq \frac{1}{4} \degree R }} \lvert R \rvert^{\frac{1}{4}} \\
\leq &\frac{\lvert R \rvert^{\frac{1}{2}}}{\phi^* (R)} \exp \bigg( O \Big( 2k \frac{q^{\frac{X}{2}}}{X} \Big) \bigg)
= o(1) .
\end{align*}
The second-to-last relation makes use of a similar result to (\ref{2k-th moment of Euler product, Dirichlet series truncation bound}) and the last relation follows from the fact that $X , \degree R \longrightarrow \infty$ with $X \leq \log_q \degree R$. Now we consider the first term on the far right side of (\ref{2k-th moment of Euler product, expression with diag and off diag terms}): The diagonal terms. We required a truncated sum only for the off-diagonal terms, and so we extend our sum using similar means as in (\ref{2k-th moment of Euler product, Dirichlet series truncation bound}):
\begin{align*}
\sum_{\substack{ A \in \mathcal{S}_{\mathcal{M}} (X) \\ \degree A \leq \frac{1}{4} \degree R \\ (A,R) = 1 }} \frac{ \alpha_k (A)^2 }{\lvert A \rvert }
= \sum_{\substack{A \in \mathcal{S}_{\mathcal{M}} (X) \\ (A,R) = 1 }} \frac{ \alpha_k (A)^2 }{\lvert A \rvert }
	+ O \Big( \lvert R \rvert^{-\frac{1}{9}} \Big) .
\end{align*}
Now, using (\ref{P_X (1/2 , chi) in terms of P_X* (1/2 , chi), a_k (A) def}) for the first relation below (and part of the second relation), we have that
\begin{align}
\begin{split} \label{2k-th moment of Euler product, sum over A X smooth coprime to R of a_k (A)^2 /A}
&\sum_{\substack{A \in \mathcal{S}_{\mathcal{M}} (X) \\ (A,R) = 1 }} \frac{ \alpha_k (A)^2 }{\lvert A \rvert }
= \prod_{\substack{ \degree P \leq X \\ P \nmid R }} \Bigg( \sum_{m=0}^{\infty} \frac{ \alpha_k \big( P^m \big)^2 }{\lvert P \rvert^m }  \Bigg)  \\
= &\prod_{\substack{ \degree P \leq X \\ P \nmid R }} \Bigg( \sum_{m=0}^{\infty} \frac{ d_k \big( P^m \big)^2 }{\lvert P \rvert^m }  \Bigg) 
	\prod_{\substack{ \frac{X}{2} < \degree P \leq X \\ P \nmid R }} \Bigg( \frac{ 1 + \frac{d_k (P)^2}{\lvert P \rvert} + \sum_{m=2}^{\infty} \frac{ \alpha_k \big( P^m \big)^2 }{\lvert P \rvert^m }}{1 + \frac{d_k (P)^2}{\lvert P \rvert} + \sum_{m=2}^{\infty} \frac{ d_k \big( P^m \big)^2 }{\lvert P \rvert^m } } \Bigg) \\
= &\prod_{\substack{ \degree P \leq X \\ P \mid R }} \Bigg( \sum_{m=0}^{\infty} \frac{ d_k \big( P^m \big)^2 }{\lvert P \rvert^m }  \Bigg)^{-1}
	\prod_{\degree P \leq X} \Bigg( \bigg( 1 - \frac{1}{\lvert P \rvert} \bigg)^{k^2} \sum_{m=0}^{\infty} \frac{ d_k \big( P^m \big)^2 }{\lvert P \rvert^m } \Bigg) \\
	&\cdot \prod_{\degree P \leq X} \Bigg( 1 - \frac{1}{\lvert P \rvert} \Bigg)^{-k^2} 
	\prod_{\frac{X}{2} < \degree P \leq X} \Bigg( 1 + O_k \Big( \frac{1}{\lvert P \rvert^2} \Big) \Bigg) \\
= &\big( 1 + o (1) \big) a(k) \prod_{\substack{ \degree P \leq X \\ P \mid R }} \Bigg( \sum_{m=0}^{\infty} \frac{ d_k \big( P^m \big)^2 }{\lvert P \rvert^m }  \Bigg)^{-1} \Big( e^{\gamma} X \Big)^{k^2} .
\end{split}
\end{align}
For the last equality, we used Lemma \ref{lemma, Mertens 3rd Theorem in FF}. The proof follows.
\end{proof}


\section{Moments of the Hadamard Product} \label{section, Hadamard moments conjecture}

In this section we provide support for the Conjecture \ref{Hadamard moment conjecture, Intro version}. We require the following lemma.

\begin{lemma}
For real $y > 0$ define
\begin{align*}
\Ci (y)
:= - \int_{t=y}^{\infty} \frac{\cos (t)}{t} \mathrm{d} t ,
\end{align*}
and let $x$ be real and non-zero. Then,
\begin{align*}
\Re E_1 ( i x)
= - \Ci ( \lvert x \rvert ) .
\end{align*}
\end{lemma}

\begin{proof}
If $x > 0$, then
\begin{align*}
\Re E_1 ( i x)
= \Re \int_{w= i x}^{ix + \infty} \frac{e^{-w}}{w} \mathrm{d} w
= \Re \int_{w= i x}^{i \infty} \frac{e^{-w}}{w} \mathrm{d} w
= \Re \int_{t=x}^{\infty} \frac{e^{-it}}{t} \mathrm{d} t
= - \Ci ( \lvert x \rvert ) ,
\end{align*}
where the second relation follows from a contour shift. Similarly, if $x < 0$, then
\begin{align*}
\Re E_1 ( i x)
= \Re \int_{w= i x}^{ix + \infty} \frac{e^{-w}}{w} \mathrm{d} w
= \Re \int_{w= i x}^{- i \infty} \frac{e^{-w}}{w} \mathrm{d} w
= \Re \int_{t= \lvert x \rvert}^{\infty} \frac{e^{it}}{t} \mathrm{d} t
= - \Ci ( \lvert x \rvert ) .
\end{align*}
\end{proof}

Now, writing $\gamma_n (\chi)$ for the imaginary part of the $n$-th zero of $L (s , \chi )$, we can see that
\begin{align}
\begin{split} \label{statement, conjecture Hadamard moments section, Z^2k in terms of Ci and gamma_n}
&\frac{1}{\phi^* (R)} \sumstar_{\chi \modulus R} \Big\lvert Z_{X} \Big( \frac{1}{2} , \chi \Big) \Big\rvert^{2k} \\
= &\frac{1}{\phi^* (R)} \sumstar_{\chi \modulus R} \exp \Bigg( - 2k \Re \sum_{\gamma_n (\chi )} U \bigg( - i \gamma_n (\chi ) (\log q) X \bigg) \Bigg) \\
= &\frac{1}{\phi^* (R)} \sumstar_{\chi \modulus R} \exp \Bigg( - 2k \Re \sum_{\gamma_n (\chi )} \int_{x=0}^{\infty} u(x) E_1 \big( - i \gamma_n (\chi ) (\log q) X \log x \big) \mathrm{d} x \Bigg) \\
= &\frac{1}{\phi^* (R)} \sumstar_{\chi \modulus R} \exp \Bigg( 2k \sum_{\gamma_n (\chi )} \int_{x=0}^{\infty} u(x) \Ci \big( \lvert \gamma_n (\chi ) \rvert (\log q) X \log x \big) \mathrm{d} x \Bigg) .
\end{split}
\end{align}
We note that the terms in the exponential tend to zero as $\lvert \gamma_n (\chi ) \rvert$ tends to infinity, and so the above is primarily concerned with the zeros close to $\frac{1}{2}$. As described in Section \ref{section, introduction}, there is a relationship between the zeros of Dirichlet $L$-functions near $\frac{1}{2}$ and the eigenphases of random unitary matrices near $0$: The proportion of Dirichlet $L$-functions of modulus $R$ that have $j$-th zero (that is, its imaginary part) in some interval $[a,b]$ appears to be the same as the proportion of unitary $N (R) \times N (R)$ matrices that have $j$-th eigenphase in $[a,b]$ (at least, this is the case in an appropriate limit). Naturally, one asks what value $N (R)$ should take in terms of $R$. We note that the mean spacing between zeros of Dirichlet $L$-functions of modulus $R$ is $\frac{2 \pi}{\log q \degree R}$, while the mean spacing between eigenphases of unitary $N \times N$ matrices is $\frac{2 \pi}{N}$. Therefore, we take $N (R) = \lfloor \log q \degree R \rfloor$. So, we replace the imaginary parts of the zeros with eigenphases of $N (R) \times N (R)$ unitary matrices, and instead of averaging over primitive characters we average over unitary matrices. That is, we conjecture
\begin{align}
\begin{split} \label{statement, Hada prod 2k-th mom conj via GHK, replacing zeros with e vals}
&\frac{1}{\phi^* (R)} \sumstar_{\chi \modulus R} \Big\lvert Z_{X} \Big( \frac{1}{2} , \chi \Big) \Big\rvert^{2k} \\
= &\frac{1}{\phi^* (R)} \sumstar_{\chi \modulus R} \exp \Bigg( 2k \sum_{\gamma_n (\chi )} \int_{x=0}^{\infty} u(x) \Ci \big( \lvert \gamma_n (\chi ) \rvert (\log q) X \log x \big) \mathrm{d} x \Bigg) \\
\sim &\int_{A \in U \big( N (R) \big)} \exp \Bigg( 2k \sum_{\theta_n (A)} \int_{x=0}^{\infty} u(x) \Ci \big( \lvert \theta_n (A) \rvert (\log q) X \log x \big) \mathrm{d} x \Bigg) \mathrm{d} A 
\end{split}
\end{align}
as $\degree R \longrightarrow \infty$, where the integral is with respect to the Haar measure, and $\theta_n (A)$ is the $n$-th eigenphase of $A$. The eigenphases are periodic with period $2 \pi$, and these periodicised eigenphases are included in the sum. An asymptotic evaluation of the right side can be made identically as in Section 4 of \cite{HybridEulerHadProdRZF_GonekHughesKeating}; but we simply replace their $\log X$ with our $(\log q) X$, and we replace their $N = \lfloor \log T \rfloor$ with our $N (R) = \lfloor \log q \degree R \rfloor$. This leads us to the conjecture that
\begin{align*} 
\frac{1}{\phi^* (R)} \sumstar_{\chi \modulus R} \Big\lvert Z_{X} \Big( \frac{1}{2} , \chi \Big) \Big\rvert^{2k}
\sim \frac{G^2 (k+1)}{G (2k+1)} \bigg( \frac{\degree R}{e^{\gamma} X} \bigg)^{k^2} ,
\end{align*}
as $\degree R \longrightarrow \infty$. We note that in \cite{HybridEulerHadProdRZF_GonekHughesKeating}, their $u(x)$ has a slightly different support than the support of our $u(x)$. However, this does not affect the result.

\begin{remark}
We will provide further justification for one of the steps above, which is not given in \cite{HybridEulerHadProdRZF_GonekHughesKeating}. In the middle line of (\ref{statement, Hada prod 2k-th mom conj via GHK, replacing zeros with e vals}) we have a sum over all $\gamma_n (\chi)$. This includes zeros that are far away from $\frac{1}{2}$. We mentioned previously that their contribution is small, but a closer inspection reveals that we cannot dismiss them so easily, and so we must justify replacing them with the eigenphases of our unitary matrices. For the zeros close to $\frac{1}{2}$ (that is, for $\gamma_n (\chi)$ close to $0$) we have already provided this justification. For the zeros further away, one can argue that the zeros of a typical Dirichlet $L$-function are equidistributed in some manner, and that the eigenphases of a typical unitary matrix are also equidistributed in some manner. Thus, we could replace the former with the latter. This is based on the idea that if you sum a function over a set of equidistributed points on some interval $I$, then the result is roughly equal to the integral over $I$ of that function multiplied by the reciprocal of the mean spacing of the points. Recall that the mean spacing of our eigenphases is equal to that of our zeros. Naturally, one asks why we do not use the same justification for the zeros close to $\frac{1}{2}$. The answer is that the function $\Ci (x)$ has a discontinuity at $x=0$, and so we require a stronger justification for the zeros near $\frac{1}{2}$ (that is, the $\gamma_n (\chi)$ close to $0$). Finally, we remark that we do not provide any rigorous support for the claims on equidistribution above.
\end{remark}


\section{The Second Hadamard Moment} \label{section, second twisted moment}

Before proving Theorem \ref{theorem, second twisted moment, Intro version}, we prove several lemmas. First, by (\ref{P_X (1/2 , chi) in terms of P_X* (1/2 , chi) and sum over a_k (A), statement}) we have
\begin{align*}
P_X \Big( \frac{1}{2} , \chi \Big)
= &\Big( 1 + O \big( X^{-1} \big) \Big) P_X^{*} \Big( \frac{1}{2} , \chi \Big) .
\end{align*}
Rearranging and using (\ref{P* definition}) gives 
\begin{align}
\begin{split} \label{P_X (1/2 , chi)^(-1) in terms of P_X* (1/2 , chi) and sum over a_(-1) (A), statement}
P_X \Big( \frac{1}{2} , \chi \Big)^{-1}
= &\Big( 1 + O \big( X^{-1} \big) \Big) P_X^{*} \Big( \frac{1}{2} , \chi \Big)^{-1} \\
= &\Big( 1 + O \big( X^{-1} \big) \Big)
	\prod_{\degree P \leq X} \bigg( 1 - \frac{\chi (P)}{\lvert P \rvert^s } \bigg)^{-1}
	\prod_{\frac{X}{2} < \degree P \leq X} \bigg( 1 + \frac{\chi (P)^2}{2 \lvert P \rvert^{2s} } \bigg)^{-1} \\
= &\Big( 1 + O \big( X^{-1} \big) \Big) \sum_{A \in \mathcal{S}_{\mathcal{M}} (X) } \frac{ \alpha_{-1} (A) \chi (A)}{\lvert A \rvert^{\frac{1}{2}} } ,
\end{split}
\end{align}
where $\alpha_{-1}$ is defined multiplicatively by
\begin{align*}
\alpha_{-1} (P)
&:= \begin{cases}
-1 &\text{ if $\degree P \leq X$} \\
0 &\text{ if $\degree P > X$} ;
\end{cases} \\
\alpha_{-1} (P^2)
&:= \begin{cases}
0 &\text{ if $\degree P \leq \frac{X}{2}$} \\
\frac{1}{2} &\text{ if $\frac{X}{2} < \degree P \leq X$} \\
0 &\text{ if $\degree P > X$} ;
\end{cases} \\
\alpha_{-1} (P^3)
&:= \begin{cases}
0 &\text{ if $\degree P \leq \frac{X}{2}$} \\
-\frac{1}{2} &\text{ if $\frac{X}{2} < \degree P \leq X$} \\
0 &\text{ if $\degree P > X$} ;
\end{cases} \\
\alpha_{-1} (P^m)
&:= 0 \text{ for $m \geq 4$.}
\end{align*}

\begin{lemma} \label{Sum of alpha_(-1) HS and HT /HST absolute}
For all $R \in \mathcal{M}$, we have that
\begin{align*}
\sum_{\substack{HST \in \mathcal{S}_{\mathcal{M}} (X) \\ (S,T)=1 \\ (HST,R)=1 \\ \degree HS , \degree HT \leq \frac{1}{10} \degree R }} \frac{\lvert \alpha_{-1} (HS) \alpha_{-1} (HT) \rvert}{\lvert HST \rvert}
\ll X^3
\end{align*}
as $X \longrightarrow \infty$.
\end{lemma}

\begin{proof}
Using Lemma \ref{lemma, Mertens 3rd Theorem in FF}, we have that
\begin{align*}
\sum_{\substack{HST \in \mathcal{S}_{\mathcal{M}}(X) \\ (S,T)=1 \\ (HST,R)=1 \\ \degree HS , \degree HT \leq \frac{1}{10} \degree R }} \hspace{-3em} \frac{\vert \alpha_{-1} (HS) \alpha_{-1} (HT) \rvert}{\lvert HST \rvert}
\ll \bigg( \sum_{H \in \mathcal{S}_{\mathcal{M}}(X)} \frac{1}{\lvert H \rvert} \bigg)^3 
= \hspace{-0.6em} \prod_{\degree P \leq X } \Big( 1 - \lvert P \rvert^{-1} \Big)^{-3} \hspace{-0.6em}
\ll X^3 
\end{align*}
as $X \longrightarrow \infty$.
\end{proof}

\begin{lemma} \label{Sum of alpha_(-1) HS and HT deg ST /HST }
For all $R \in \mathcal{M}$, we have that
\begin{align*}
\sum_{\substack{HST \in \mathcal{S}_{\mathcal{M}}(X) \\ (S,T)=1 \\ (HST,R)=1 \\ \degree HS , \degree HT \leq \frac{1}{10} \degree R }} \frac{\alpha_{-1} (HS) \alpha_{-1} (HT)}{\lvert HST \rvert} \degree ST
\ll X^4
\end{align*}
as $X \longrightarrow \infty$.
\end{lemma}

\begin{proof}
We have that
\begin{align*}
\sum_{\substack{HST \in \mathcal{S}_{\mathcal{M}} (X) \\ (S,T)=1 \\ (HST,R)=1 \\ \degree HS , \degree HT \leq \frac{1}{10} \degree R }} \frac{\alpha_{-1} (HS) \alpha_{-1} (HT)}{\lvert HST \rvert} \degree ST
\ll \sum_{H \in \mathcal{S}_{\mathcal{M}} (X)} \frac{1}{\lvert H \rvert} \sum_{ S,T \in \mathcal{S}_{\mathcal{M}} (X)} \frac{\degree ST }{\lvert S T \rvert} .
\end{align*}
Consider
\begin{align*}
f(s) := \sum_{S,T \in \mathcal{S}_{\mathcal{M}} (X)} \frac{1}{\lvert S T \rvert^s}
= \bigg( \sum_{T \in \mathcal{S}_{\mathcal{M}} (X)} \frac{1}{\lvert T \rvert^s} \bigg)^2  
= \prod_{\degree P \leq X } \Big( 1 - \lvert P \rvert^{-s} \Big)^{-2} .
\end{align*}
Taking the derivative of the above and then evaluating at $s=1$, we obtain
\begin{align*}
\sum_{ S,T \in \mathcal{S}_{\mathcal{M}} (X)} \frac{\degree ST }{\lvert S T \rvert}
= 2 \prod_{\degree P \leq X } \Big( 1 - \lvert P \rvert^{-1} \Big)^{-2} \sum_{\degree P \leq X } \frac{\degree P}{\lvert P \rvert -1}
\ll X^3 
\end{align*}
as $X \longrightarrow \infty$, where we have made use of Lemma \ref{lemma, Mertens 3rd Theorem in FF} and the prime polynomial theorem. This, along with the fact that
\begin{align*}
\sum_{H \in \mathcal{S}_{\mathcal{M}} (X)} \frac{1}{\lvert H \rvert}
= \prod_{\degree P \leq X } \Big( 1 - \lvert P \rvert^{-1} \Big)^{-1}
\ll X 
\end{align*}
as $X \longrightarrow \infty$, proves the lemma.
\end{proof}

\begin{lemma} \label{Sum of alpha_(-1) HS and HT /HST exact}
Let $V \in \mathcal{M}$. $V$ may or may not depend on $R$. As $X,\degree R \overset{q}{\longrightarrow} \infty$ with $X \leq \log_q \degree R$, we have
\begin{align*}
&\sum_{\substack{HST \in \mathcal{S}_{\mathcal{M}} (X) \\ (S,T)=1 \\ (HST,V)=1 \\ \degree HS , \degree HT \leq \frac{1}{10} \degree R }} \frac{\alpha_{-1} (HS) \alpha_{-1} (HT)}{\lvert HST \rvert} \\
= &\bigg( 1 + O \Big( q^{-\frac{X}{2}} \Big) \bigg) \prod_{\substack{\degree P \leq X \\ P \nmid V}} \bigg( 1 - \frac{1}{\lvert P \rvert} \bigg)
	+ O \bigg( \frac{1}{\lvert R \rvert^{\frac{1}{21}}} \bigg) 
\sim \prod_{\substack{\degree P \leq X \\ P \nmid V}} \bigg( 1 - \frac{1}{\lvert P \rvert} \bigg) .
\end{align*}
\end{lemma}

\begin{proof}
The second relation in the Lemma follows easily from Lemma \ref{lemma, Mertens 3rd Theorem in FF}. We will prove the first. In this proof, all asymptotic relations are to be taken as $X,\degree R \overset{q}{\longrightarrow} \infty$ with $X \leq \log_q \degree R$. \\

Similar to (\ref{P_X* (1/2,chi) Dir Ser truncation, statement}), we can remove the conditions $\degree HS , \degree HT \leq \frac{1}{10} \degree R$ from the sum and this only adds an $O \big( \lvert R \rvert^{-\frac{1}{21}})$ term . Now, writing $C=HS$ and $D=HT$, we have
\begin{align*}
&\sum_{\substack{HST \in \mathcal{S}_{\mathcal{M}} (X) \\ (S,T)=1 \\ (HST,V)=1 }} \frac{\alpha_{-1} (HS) \alpha_{-1} (HT)}{\lvert HST \rvert} 
= \sum_{\substack{CD \in \mathcal{S}_{\mathcal{M}} (X) \\ (CD,V)=1 }} \frac{\alpha_{-1} (C) \alpha_{-1} (D)}{\lvert CD \rvert} \lvert (C,D) \rvert \\
= &\sum_{\substack{CD \in \mathcal{S}_{\mathcal{M}} (X) \\ (CD,V)=1 }} \frac{\alpha_{-1} (C) \alpha_{-1} (D)}{\lvert CD \rvert} \sum_{G \mid (C,D)} \phi (G) 
= \sum_{\substack{G \in \mathcal{S}_{\mathcal{M}} (X) \\ (G,V)=1 }} \frac{\phi (G)}{\lvert G \rvert^2} \Bigg( \sum_{\substack{C \in \mathcal{S}_{\mathcal{M}} (X) \\ (C,V)=1 }} \frac{\alpha_{-1} (CG)}{\lvert C \rvert} \Bigg)^2 .
\end{align*}

Before continuing, let us make a definition: For all $A \in \mathcal{M}$ and all $P \in \mathcal{P}$, let $e_P (A)$ be the largest integer such that $P^{e_P (A)} \mid A$. Continuing, we note that we can restrict the sums to polynomials that are fourth power free. Indeed, $\alpha_{-1} (P^m) = 0$ for all $P \in \mathcal{P}$ and all $m \geq 4$. Note that if $P \mid G$ then we must have that $0 \leq e_P (C) \leq 3 - e_P (G)$, while if $P \nmid G$ then $0 \leq e_P (C) \leq 3$. So, we have
\begin{align*}
&\sum_{\substack{C \in \mathcal{S}_{\mathcal{M}} (X) \\ (C,V)=1 }} \hspace{-1em} \frac{\alpha_{-1} (CG)}{\lvert C \rvert} 
= \prod_{P \mid G} \Bigg( \sum_{j=0}^{3-e_P (G)} \frac{\alpha_{-1} (P^{j + e_P (G)})}{\lvert P \rvert^j} \Bigg)
	\prod_{\substack{\degree P \leq X \\ P \nmid G \\ P \nmid V}} \Bigg( \sum_{j=0}^{3} \frac{\alpha_{-1} (P^{j})}{\lvert P \rvert^j} \Bigg) \\
= &\prod_{\substack{\degree P \leq X \\ P \nmid V}} \Bigg( \sum_{j=0}^{3} \frac{\alpha_{-1} (P^{j})}{\lvert P \rvert^j} \Bigg)
	\prod_{P \mid G} \Bigg( \sum_{j=0}^{3-e_P (G)} \frac{\alpha_{-1} (P^{j + e_P (G)})}{\lvert P \rvert^j} \Bigg)
	\prod_{P \mid G} \Bigg( \sum_{j=0}^{3} \frac{\alpha_{-1} (P^{j})}{\lvert P \rvert^j} \Bigg)^{-1} .
\end{align*}
So,
\begin{align*}
&\sum_{\substack{G \in \mathcal{S}_{\mathcal{M}} (X) \\ (G,V)=1 }} \frac{\phi (G)}{\lvert G \rvert^2} \Bigg( \sum_{\substack{C \in \mathcal{S}_{\mathcal{M}} (X) \\ (C,V)=1 }} \frac{\alpha_{-1} (CG)}{\lvert C \rvert} \Bigg)^2 \\
= &\prod_{\substack{\degree P \leq X \\ P \nmid V}} \hspace{-0.3em} \Bigg( \sum_{j=0}^{3} \frac{\alpha_{-1} (P^{j})}{\lvert P \rvert^j} \Bigg)^2
	\hspace{-0.8em} \prod_{\substack{\degree P \leq X \\ P \nmid V}} \hspace{-0.3em} \Bigg( \sum_{i=0}^{3} \frac{\phi (P^i)}{\lvert P \rvert^{2i}} 
		\Bigg( \sum_{j=0}^{3-i} \frac{\alpha_{-1} (P^{j + i})}{\lvert P \rvert^j} \Bigg)^2
		\Bigg( \sum_{j=0}^{3} \frac{\alpha_{-1} (P^{j})}{\lvert P \rvert^j} \Bigg)^{-2} \Bigg) \\
= &\prod_{\substack{\degree P \leq X \\ P \nmid V}} \Bigg( \sum_{i=0}^{3} \frac{\phi (P^i)}{\lvert P \rvert^{2i}} 
		\Bigg( \sum_{j=0}^{3-i} \frac{\alpha_{-1} (P^{j + i})}{\lvert P \rvert^j} \Bigg)^2 \Bigg) \\
= & \prod_{\substack{\degree P \leq X \\ P \nmid V}} \Bigg( \sum_{i=0}^{3} \sum_{j=0}^{3-i} \sum_{k=0}^{3-i} \frac{\phi (P^i ) \alpha_{-1} (P^{j+i}) \alpha_{-1} (P^{k+i})}{\lvert P^{2i+j+k} \rvert} \Bigg) \\
= &\prod_{\substack{\degree P \leq X \\ P \nmid V}} \bigg( 1 - \frac{1}{\lvert P \rvert} \bigg) 
	\prod_{\substack{\frac{X}{2} \leq \degree P \leq X \\ P \nmid V}} \bigg( 1 + O \Big( \frac{1}{\lvert P^2 \rvert} \Big) \bigg) \\
= &\bigg( 1 + O \Big( q^{-\frac{X}{2}} \Big) \bigg) \prod_{\substack{\degree P \leq X \\ P \nmid V}} \bigg( 1 - \frac{1}{\lvert P \rvert} \bigg) .
\end{align*}
The result follows.
\end{proof}

\begin{lemma} \label{Sum of deg AB = Z_1, AC equiv BD mod E of AB^(1/2)}
Let $R \in \mathcal{M}$. Suppose $Z_1 \leq \degree R$ and $F \mid R$. Further, suppose $C,D \in \mathcal{S}_{\mathcal{M}} (X)$ with $\degree C , \degree D \leq \frac{1}{10} \degree R$. Then, we have
\begin{align*}
\sum_{\substack{A,B \in \mathcal{M} \\ \degree AB = Z_1 \\ AC \equiv BD (\modulus F) \\ AC \neq BD \\ (AB,R)=1}} \frac{1}{\lvert AB \rvert^{\frac{1}{2}}}
\ll \frac{q^{\frac{Z_1 }{2}} (Z_1 + 1) \lvert CD \rvert}{\lvert F \rvert} .
\end{align*}
\end{lemma}

\begin{proof}
Consider the case where $\degree AC > \degree BD$, and suppose that $\degree A =i$. We have that $AC = LF + BD$ for some $L \in \mathcal{M}$ with $\degree L = \degree AC - \degree F = i + \degree C - \degree F$, and $\degree B = Z_1 - \degree A = Z_1 - i$. Hence,
\begin{align*}
\sum_{\substack{A,B \in \mathcal{M} \\ \degree AB = Z_1 \\ AC \equiv BD (\modulus F) \\ (AB,R)=1 \\ \degree AC > \degree BD}} \frac{1}{\lvert AB \rvert^{\frac{1}{2}}} 
\leq &q^{-\frac{Z_1 }{2}} \sum_{i=0}^{Z_1} \sum_{\substack{L \in \mathcal{M} \\ \degree L = i + \degree C - \degree F}} \; \sum_{\substack{B \in \mathcal{M} \\ \degree B = Z_1 - i}} 1 \\
= &q^{\frac{Z_1 }{2}} \sum_{i=0}^{Z_1} \sum_{\substack{L \in \mathcal{M} \\ \degree L = i + \degree C - \degree F}} q^{-i}
= \frac{q^{\frac{Z_1 }{2}} \lvert C \rvert}{\lvert F \rvert} \sum_{i=0}^{Z_1} 1
= \frac{q^{\frac{Z_1 }{2}} (Z_1 + 1) \lvert C \rvert}{\lvert F \rvert} .
\end{align*}
Similarly, when $\degree BD > \degree AC$ we have
\begin{align*}
\sum_{\substack{A,B \in \mathcal{M} \\ \degree AB = Z_1 \\ AC \equiv BD (\modulus F) \\ (AB,R)=1 \\ \degree AC > \degree BD}} \frac{1}{\lvert AB \rvert^{\frac{1}{2}}}
\leq \frac{q^{\frac{Z_1 }{2}} (Z_1 + 1) \lvert D \rvert}{\lvert F \rvert} .
\end{align*}
Suppose now that $\degree AC = \degree BD = i$. Then, $2i = \degree ABCD = Z_1 + \degree CD$. We have $\degree B = i - \degree D = \frac{Z_1 + \degree C - \degree D}{2}$, and $AC = LF + BD$ for some $L \in \mathcal{A}$ with $\degree L < i - \degree F = \frac{Z_1 + \degree CD}{2} - \degree F$. Hence,
\begin{align*}
\sum_{\substack{A,B \in \mathcal{M} \\ \degree AB = Z_1 \\ AC \equiv BD (\modulus F) \\ (AB,R)=1 \\ \degree AC = \degree BD}} \frac{1}{\lvert AB \rvert^{\frac{1}{2}}}
\leq &q^{-\frac{Z_1 }{2}} \sum_{\substack{B \in \mathcal{M} \\ \degree B = \frac{Z_1 + \degree C - \degree D}{2}}} \sum_{\substack{L \in \mathcal{A} \\ \degree L < \frac{Z_1 + \degree CD}{2} - \degree F}} 1 \\
= &\frac{\lvert CD \rvert^{\frac{1}{2}}}{\lvert F \rvert} \sum_{\substack{B \in \mathcal{M} \\ \degree B = \frac{Z_1 + \degree C - \degree D}{2}}} 1
= \frac{q^{\frac{Z_1 }{2}} \lvert C \rvert}{\lvert F \rvert} .
\end{align*}
The result follows.
\end{proof}

We can now prove Theorem \ref{theorem, second twisted moment, Intro version}.

\begin{proof}[Proof of Theorem \ref{theorem, second twisted moment, Intro version}]
Throughout the proof, all asymptotic relations will be taken as $X , \degree R \overset{q}{\longrightarrow} \infty$ with $X \leq \log_q \degree R$. Now, by (\ref{P_X (1/2 , chi)^(-1) in terms of P_X* (1/2 , chi) and sum over a_(-1) (A), statement}), we have
\begin{align} \label{statement, second twisted moment proof, replace P with P*}
\frac{1}{\phi^* (R)} \sumstar_{\chi \modulus R} \Big\lvert L \Big( \frac{1}{2} , \chi \Big) P_X \Big( \frac{1}{2} , \chi \Big)^{-1} \Big\rvert^2
\sim \frac{1}{\phi^* (R)} \sumstar_{\chi \modulus R} \Big\lvert L \Big( \frac{1}{2} , \chi \Big)  P_X^* \Big( \frac{1}{2} , \chi \Big)^{-1} \Big\rvert^2 .
\end{align}
Similar to (\ref{P_X* (1/2,chi) Dir Ser truncation, statement}), we truncate our sum:
\begin{align*}
P_X^* \Big( \frac{1}{2} , \chi \Big)^{-1}
= \sum_{\substack{ C \in \mathcal{S}_{\mathcal{M}} (X) \\ \degree C \leq \frac{1}{10} \degree R }} \frac{\alpha_{-1} (C) \chi (C) }{\lvert C \rvert^{\frac{1}{2}} }
	+ O \Big( \lvert R \rvert^{-\frac{1}{50}} \Big) .
\end{align*}
Using this, the Cauchy-Schwarz inequality, and (\ref{statement, 2nd moment DLF in FF}), it suffices to prove that
\begin{align}
\begin{split} \label{statement, 2nd twisted moment,C,D truncted suffices to prove}
&\frac{1}{\phi^* (R)} \sumstar_{\chi \modulus R} \Big\lvert L \Big( \frac{1}{2} , \chi \Big) \Big\rvert^2 \sum_{\substack{ C,D \in \mathcal{S}_{\mathcal{M}} (X) \\ \degree C , \degree D \leq \frac{1}{10} \degree R }} \frac{\alpha_{-1} (C) \alpha_{-1} (D) \chi (C) \conj\chi (D)}{\lvert C D \rvert^{\frac{1}{2}} } \\
\sim &\frac{\degree R}{e^{\gamma} X} \prod_{\substack{\degree P > X \\ P \mid R}} \bigg( 1 - \frac{1}{\lvert P \rvert} \bigg)  .
\end{split}
\end{align}
Now, by Lemma \ref{Lemma, short sum for squared L-function}, we have
\begin{align*}
&\frac{1}{\phi^* (R)} \sumstar_{\chi \modulus R} \Big\lvert L \Big( \frac{1}{2} , \chi \Big) \Big\rvert^2 \sum_{\substack{ C,D \in \mathcal{S}_{\mathcal{M}} (X) \\ \degree C , \degree D \leq \frac{1}{10} \degree R }} \frac{\alpha_{-1} (C) \alpha_{-1} (D) \chi (C) \conj\chi (D)}{\lvert C D \rvert^{\frac{1}{2}} } \\
= &\frac{1}{\phi^* (R)} \sumstar_{\chi \modulus R} \Big( a(\chi )+ c(\chi ) \Big)  \sum_{\substack{ C,D \in \mathcal{S}_{\mathcal{M}} (X) \\ \degree C , \degree D \leq \frac{1}{10} \degree R }} \frac{\alpha_{-1} (C) \alpha_{-1} (D) \chi (C) \conj\chi (D)}{\lvert C D \rvert^{\frac{1}{2}} } ,
\end{align*}
where
\begin{align*}
a (\chi) := &2\sum_{\substack{A,B \in \mathcal{M} \\ \degree AB < \degree R }} \frac{\chi (A) \conj\chi (B)}{\lvert AB \rvert^{\frac{1}{2}}}
\end{align*}
and $c(\chi )$ is defined in Lemma \ref{Lemma, short sum for squared L-function}. \\

We first consider the case with $a (\chi )$. We have 
\begin{align}
\begin{split} \label{second twisted moment, a(chi) sum as diag and off diag terms}
&\frac{1}{\phi^* (R)} \sumstar_{\chi \modulus R} a(\chi ) \sum_{\substack{ C,D \in \mathcal{S}_{\mathcal{M}} (X) \\ \degree C , \degree D \leq \frac{1}{10} \degree R }} \frac{\alpha_{-1} (C) \alpha_{-1} (D) \chi (C) \conj\chi (D)}{\lvert C D \rvert^{\frac{1}{2}} } \\
= &\frac{2}{\phi^* (R)} \sumstar_{\chi \modulus R} \sum_{\substack{A,B \in \mathcal{M} \\ C,D \in \mathcal{S}_{\mathcal{M}} (X) \\ \degree AB < \degree R \\ \degree C , \degree D \leq \frac{1}{10}\degree R }} \frac{ \alpha_{-1} (C) \alpha_{-1} (D) \chi (AC) \conj\chi (BD)}{\lvert ABCD \rvert^{\frac{1}{2}}} \\
= &\frac{2}{\phi^* (R)} \sum_{EF=R} \mu (E) \phi (F)  \sum_{\substack{A,B \in \mathcal{M} \\ C,D \in \mathcal{S}_{\mathcal{M}} (X) \\ \degree AB < \degree R \\ \degree C , \degree D \leq \frac{1}{10}\degree R \\ (ABCD,R)=1 \\ AC \equiv BD (\modulus F) }} \frac{\alpha_{-1} (C) \alpha_{-1} (D)}{\lvert ABCD \rvert^{\frac{1}{2}}} \\
= &2\sum_{\substack{A,B \in \mathcal{M} \\ C,D \in \mathcal{S}_{\mathcal{M}} (X) \\ \degree AB < \degree R \\ \degree C , \degree D \leq \frac{1}{10}\degree R \\ (ABCD,R)=1 \\ AC = BD }} \frac{\alpha_{-1} (C) \alpha_{-1} (D)}{\lvert ABCD \rvert^{\frac{1}{2}}} \\
	&+ \frac{2}{\phi^* (R)} \sum_{EF=R} \mu (E) \phi (F)  \sum_{\substack{A,B \in \mathcal{M} \\ C,D \in \mathcal{S}_{\mathcal{M}} (X) \\ \degree AB < \degree R \\ \degree C , \degree D \leq \frac{1}{10}\degree R \\ (ABCD,R)=1 \\ AC \equiv BD (\modulus F) \\ AC \neq BD }} \frac{\alpha_{-1} (C) \alpha_{-1} (D)}{\lvert ABCD \rvert^{\frac{1}{2}}} .
\end{split}
\end{align}

For the first term on the far right side, the diagonal terms, we can write $A=GS$, $B=GT$, $C=HT$, $D=HS$ where $G,H,S,T \in \mathcal{M}$ and $(S,T)=1$, giving
\begin{align}
\begin{split} \label{statement, EH chapter, 2nd mom hada prod section, ABCD main sum as GHST sum, G sum separate}
2\sum_{\substack{A,B,C,D \in \mathcal{M} \\ C,D \in \mathcal{S}_{\mathcal{M}} (X) \\ \degree AB < \degree R \\ \degree C , \degree D \leq \frac{1}{10}\degree R \\ (ABCD,R)=1 \\ AC = BD }} \frac{\alpha_{-1} (C) \alpha_{-1} (D)}{\lvert ABCD \rvert^{\frac{1}{2}}} 
= &2\sum_{\substack{G \in \mathcal{M} \\ H,S,T \in \mathcal{S}_{\mathcal{M}} (X) \\ \degree G^2 ST < \degree R \\ \degree HS , \degree HT \leq \frac{1}{10}\degree R \\ (GHST,R)=1 \\ (S,T)=1 }} \frac{\alpha_{-1} (HT) \alpha_{-1} (HS)}{\lvert GHST \rvert} \\
= &2\sum_{\substack{H,S,T \in \mathcal{S}_{\mathcal{M}} (X) \\ \degree HS , \degree HT \leq \frac{1}{10}\degree R \\ (HST,R)=1 \\ (S,T)=1 }} \frac{\alpha_{-1} (HS) \alpha_{-1} (HT)}{\lvert HST \rvert} \hspace{-2em} \sum_{\substack{G \in \mathcal{M} \\ \degree G \leq \frac{\degree R - \degree ST}{2} \\ (G,R)=1 }} \frac{1}{\lvert G \rvert} .
\end{split}
\end{align}
By Corollary \ref{Sum over (A,R)=1, deg A <= a deg R of 1/A} and Lemmas \ref{Sum of alpha_(-1) HS and HT /HST absolute}, \ref{Sum of alpha_(-1) HS and HT deg ST /HST }, and \ref{Sum of alpha_(-1) HS and HT /HST exact} we obtain the asymptotic relation below. The final equality uses Lemma \ref{lemma, Mertens 3rd Theorem in FF}.
\begin{align}
\begin{split} \label{statement, 2nd hada mom proof, sum of alpha_-1 (C,D)/ABCD^0.5 asymp}
2\sum_{\substack{A,B \in \mathcal{M} \\ C,D \in \mathcal{S}_{\mathcal{M}} (X) \\ \degree AB < \degree R \\ \degree C , \degree D \leq \frac{1}{10}\degree R \\ (ABCD,R)=1 \\ AC = BD }} \frac{\alpha_{-1} (C) \alpha_{-1} (D)}{\lvert ABCD \rvert^{\frac{1}{2}}} 
\sim &\frac{\phi (R)}{\lvert R \rvert} \degree R \prod_{\substack{\degree P \leq X \\ P \nmid R}} \bigg( 1 - \frac{1}{\lvert P \rvert} \bigg) \\
\sim & \frac{\degree R}{e^{\gamma} X} \prod_{\substack{\degree P > X \\ P \mid R}} \bigg( 1 - \frac{1}{\lvert P \rvert} \bigg)  .
\end{split}
\end{align}

For the second term on the far right side of (\ref{second twisted moment, a(chi) sum as diag and off diag terms}), the off-diagonal terms, we use Lemma \ref{Sum of deg AB = Z_1, AC equiv BD mod E of AB^(1/2)} to obtain
\begin{align}
\begin{split} \label{statement, 2nd hada mom proof, off diag final eval}
&\frac{2}{\phi^* (R)} \sum_{EF=R} \mu (E) \phi (F)  \sum_{\substack{A,B \in \mathcal{M} \\ C,D \in \mathcal{S}_{\mathcal{M}} (X) \\ \degree AB < \degree R \\ \degree C , \degree D \leq \frac{1}{10}\degree R \\ (ABCD,R)=1 \\ AC \equiv BD (\modulus F) \\ AC \neq BD }} \frac{\alpha_{-1} (C) \alpha_{-1} (D)}{\lvert ABCD \rvert^{\frac{1}{2}}} \\
= &\frac{2}{\phi^* (R)} \sum_{\substack{C,D \in \mathcal{S}_{\mathcal{M}} (X) \\ \degree C , \degree D \leq \frac{1}{10}\degree R \\  (CD,R)=1 }} \frac{\alpha_{-1} (C) \alpha_{-1} (D)}{\lvert CD \rvert^{\frac{1}{2}}} \sum_{EF=R} \mu (E) \phi (F)  \sum_{\substack{A,B \in \mathcal{M} \\ \degree AB < \degree R \\ (AB,R)=1 \\  AC \equiv BD (\modulus F) \\ AC \neq BD }} \frac{1}{\lvert AB \rvert^{\frac{1}{2}}} \\
\ll &\frac{\lvert R \rvert^{\frac{1}{2}} \degree R }{\phi^* (R)} \sum_{\substack{C,D \in \mathcal{S}_{\mathcal{M}}(X) \\ \degree C , \degree D \leq \frac{1}{10}\degree R }} \lvert CD \rvert^{\frac{1}{2}} \sum_{EF=R} \lvert \mu (E) \rvert \frac{\phi (F)}{\lvert F \rvert} \\
\ll &\frac{\lvert R \rvert^{\frac{4}{5}} 2^{\omega (R)} \degree R }{\phi^* (R)}
= o(1) .
\end{split}
\end{align}

Finally, consider the case with $c (\chi)$. We recall that if $\chi$ is odd then it consists of one sum, whereas, if $\chi$ is even it consists of three sums. We will show that one of the sums for the even $\chi$ is of lower order. The other sums for the even $\chi$, and the odd $\chi$, are similar. We then see that the total contribution of the case with $c (\chi)$ is of lower order. We have
\begin{align}
\begin{split} \label{statement, 2nd hada mom proof, other terms final eval}
&\frac{1}{\phi^* (R)} \sumstar_{\substack{\chi \modulus R \\ \chi \text{ even} }} \sum_{\substack{A,B \in \mathcal{M} \\ \degree AB = \degree R}} \frac{\chi (A) \conj\chi (B)}{\lvert AB \rvert^{\frac{1}{2}} } \sum_{\substack{ C,D \in \mathcal{S}_{\mathcal{M}} (X) \\ \degree C , \degree D \leq \frac{1}{10} \degree R }} \frac{\alpha_{-1} (C) \alpha_{-1} (D) \chi (C) \conj\chi (D)}{\lvert C D \rvert^{\frac{1}{2}} } \\
\leq &\frac{1}{\phi^* (R)} \sumstar_{\chi \modulus R} \sum_{\substack{A,B,C,D \in \mathcal{M} \\ C,D \in \mathcal{S}_{\mathcal{M}} (X) \\ \degree AB = \degree R \\ \degree C , \degree D \leq \frac{1}{10}\degree R }} \frac{ \alpha_{-1} (C) \alpha_{-1} (D) \chi (AC) \conj\chi (BD)}{\lvert ABCD \rvert^{\frac{1}{2}}} 
\ll X^3 ,
\end{split}
\end{align}
where the last relation follows by similar means as the case with $a (\chi )$.
\end{proof}


\section{Preliminary Results for the Fourth Hadamard Moment}

In this section we develop the preliminary results that are required for the proof of Theorem \ref{Theorem, fourth moment of Hadamard Product, Intro version}. We begin with two results that will simplify the problem.

\begin{lemma} \label{Lemma P_X expressed as product P_x^**}
For $X \geq 12$, we have that
\begin{align*}
P_X \Big( \frac{1}{2}, \chi \Big)^{-2}
= \big(1 + O(X^{-1}) \big) {P_X}^{**} \Big( \frac{1}{2}, \chi \Big) ,
\end{align*}
where
\begin{align*}
P_X^{**} \Big( \frac{1}{2}, \chi \Big)
:= \sum_{A \in \mathcal{S}_{\mathcal{M}} (X)} \frac{\beta (A) \chi (A)}{\lvert A \rvert^{\frac{1}{2}}}
\end{align*}
and $\beta$ is defined multiplicatively by
\begin{align}
\begin{split} \label{statement def, EH chapter, prelim results 4th Hada prod section, beta (A) def}
\beta (P)
&:= \begin{cases}
-2 &\text{ if $\degree P \leq X$} \\
0 &\text{ if $\degree P > X$}
\end{cases} \\
\beta (P^2)
&:= \begin{cases}
1 &\text{ if $\degree P \leq \frac{X}{2}$} \\
2 &\text{ if $\frac{X}{2} < \degree P \leq X$} \\
0 &\text{ if $\degree P > X$}
\end{cases} \\
\beta (P^k) 
&:= 0 \text{ for $k \geq 3$} .
\end{split}
\end{align}
\end{lemma}

\begin{proof}
By Lemma \ref{lemma, P_X (1/2 , chi) in terms of P_X* (1/2 , chi)} we have
\begin{align*}
P_X \Big( \frac{1}{2}, \chi \Big)^{-2}
= \big( 1 + O (X^{-1}) \big) \prod_{\degree P \leq X} \bigg( 1 - \frac{\chi (P)}{\vert P \rvert^{\frac{1}{2}} } \bigg)^{2}
	\prod_{\frac{X}{2} < \degree P \leq X} \bigg( 1 + \frac{\chi (P)^2}{2 \vert P \rvert } \bigg)^{2} .
\end{align*}
By writing $P_X^{**} \Big( \frac{1}{2}, \chi \Big)$ as an Euler product, we see that
\begin{align*}
&\prod_{\degree P \leq X} \bigg( 1 - \frac{\chi (P)}{\vert P \rvert^{\frac{1}{2}} } \bigg)^{2}
	\prod_{\frac{X}{2} < \degree P \leq X} \bigg( 1 + \frac{\chi (P)^2}{2 \vert P \rvert } \bigg)^{2} \\
= &P_X^{**} \Big( \frac{1}{2}, \chi \Big) 
	\prod_{\frac{X}{2} < \degree P \leq X} \Bigg( 1
		+ \frac{ - \frac{2 \chi (P)^3 }{\lvert P \rvert^{\frac{3}{2}}} + \frac{5 \chi (P)^4 }{ 4\lvert P \rvert^{2}} - \frac{\chi (P)^5 }{2 \lvert P \rvert^{\frac{5}{2}}} + \frac{\chi (P)^6 }{4 \lvert P \rvert^{6}} }{1 - \frac{2 \chi (P)}{\lvert P \rvert^{\frac{1}{2}}} + \frac{2 \chi (P)^2 }{\lvert P \rvert }} \Bigg) \\
= &P_X^{**} \Big( \frac{1}{2}, \chi \Big)
	\prod_{\frac{X}{2} < \degree P \leq X} \bigg( 1 + O \Big( \lvert P \rvert^{-\frac{3}{2}} \Big) \bigg) \\
= &P_X^{**} \Big( \frac{1}{2}, \chi \Big)
	\exp \Bigg( O \bigg( \sum_{\frac{X}{2} < \degree P \leq X} \lvert P \rvert^{-\frac{3}{2}} \bigg) \Bigg) \\
= & \Big( 1 + O \big( X^{-1} q^{-\frac{X}{4}} \big) \Big) P_X^{**} \Big( \frac{1}{2}, \chi \Big) .
\end{align*}
The result follows. The requirement that $X \geq 12$ is so that the factor $\Big( 1 - \frac{2 \chi (P)}{\lvert P \rvert^{\frac{1}{2}}} + \frac{2 \chi (P)^2 }{\lvert P \rvert } \Big)^{-1}$ in the second line is guaranteed to be non-zero.
\end{proof}

\begin{lemma} \label{Lemma P_X expressed as sum hat P_x^**}
We define
\begin{align*}
\widehat{P_x^{**}} \Big( \frac{1}{2}, \chi \Big)
:= \sum_{\substack{A \in \mathcal{S}_{\mathcal{M}} (X) \\ \degree A \leq \frac{1}{8} \log_q \degree R}} \frac{\beta (A) \chi (A)}{\lvert A \rvert^{\frac{1}{2}}} .
\end{align*}
Then, as $X , \degree R \overset{q}{\longrightarrow} \infty$ with $X \leq \log_q \degree R$,
\begin{align*}
P_X^{**} \Big( \frac{1}{2}, \chi \Big)
= \widehat{P_x^{**}} \Big( \frac{1}{2}, \chi \Big) + O \Big( (\degree R)^{-\frac{1}{33}} \Big) .
\end{align*}
\end{lemma}

\begin{proof}
We have, as $X , \degree R \overset{q}{\longrightarrow} \infty$ with $X \leq \log_q \degree R$,
\begin{align*}
&\sum_{\substack{A \in \mathcal{S}_{\mathcal{M}} (X) \\ \degree A > \frac{1}{8} \log_q \degree R}} \frac{\beta (A) \chi (A)}{\lvert A \rvert^{\frac{1}{2}}}
\ll \frac{1}{(\degree R)^{\frac{1}{32}}} \sum_{A \in \mathcal{S}_{\mathcal{M}} (X)} \frac{\lvert \beta (A) \rvert}{\lvert A \rvert^{\frac{1}{4}}} \\
= &(\degree R)^{-\frac{1}{32}} \prod_{\degree P \leq X} \Big( 1 + 2 \lvert P \rvert^{-\frac{1}{4}} + 2 \lvert P \rvert^{-\frac{1}{2}} \Big) \\
= &(\degree R)^{-\frac{1}{32}} \exp \Bigg( O \bigg( \sum_{\degree P \leq X}\lvert P \rvert^{-\frac{1}{4}} \bigg) \Bigg) \\
= &(\degree R)^{-\frac{1}{32}} \exp \Bigg( O \bigg( \frac{q^{\frac{3}{4} X}}{X}\bigg) \Bigg) 
\leq (\degree R)^{-\frac{1}{33}} .
\end{align*}
\end{proof}

We now prove several results that will be used to obtain the main asymptotic term in Theorem \ref{Theorem, fourth moment of Hadamard Product, Intro version}.

\begin{lemma} \label{A_1 A_2 A_3 = B_1 B_2 B_3 rewrite}
Suppose $A_1 , A_2 , A_3 , B_1 , B_2 , B_3 \in \mathcal{M}$ satisfy $A_1 A_2 A_3 = B_1 B_2 B_3$. Then, there are $G_1 , G_2 , G_3 , V_{1,2} , V_{1,3} , V_{2,1} , V_{2,3} , V_{3,1} , V_{3,2} \in \mathcal{M}$, satisfying $(V_{i,j} , V_{k,l}) = 1$ when both $i\neq k$ and $j \neq l$ hold, such that
\begin{align*}
A_1 = G_1 V_{1,2} V_{1,3} \quad &B_1 = G_1 V_{2,1} V_{3,1} \\
A_2 = G_2 V_{2,1} V_{2,3} \quad &B_2 = G_2 V_{1,2} V_{3,2} \\
A_3 = G_3 V_{3,1} V_{3,2} \quad &B_3 = G_3 V_{1,3} V_{2,3} . \\
\end{align*}
Furthermore, this is a bijective correspondence. To clarify, $G_i$ is the highest common divisor of $A_i$ and $B_i$; and in $V_{i,j}$ the subscript $i$ indicates that $V_{i,j}$ divides $A_i$ and the subscript $j$ indicates that $V_{i,j}$ divides $B_j$.
\end{lemma}

\begin{proof}
Let us write $A_i = G_i S_i$ and $B_i = G_i T_i$, where 
\begin{align} 
\begin{split}\label{S_i and T_i coprime}
G_i = (A_i , B_i) \\
(S_i , T_i)=1 .
\end{split}
\end{align}
 Since $A_1 A_2 A_3 = B_1 B_2 B_3$, we must have that
 \begin{align} \label{S_1 S_2 S_3 = T_1 T_2 T_3}
 S_1 S_2 S_3 = T_1 T_2 T_3.
 \end{align}
 First we note that, due to (\ref{S_1 S_2 S_3 = T_1 T_2 T_3}) and the coprimality relations in (\ref{S_i and T_i coprime}), we have that $S_i \mid T_j T_k$ and $T_i \mid S_j S_k$ for $i,j,k$ distinct. \\
 
Second, again due to to (\ref{S_1 S_2 S_3 = T_1 T_2 T_3}) and (\ref{S_i and T_i coprime}), we must have that $(S_1 , S_2 , S_3) , (T_1 , T_2 , T_3) =1$. \\

Third, for $i \neq j$, we define $S_{i,j} := (S_i , S_j)$ and $T_{i,j} := (T_i , T_j)$. Again due to to (\ref{S_1 S_2 S_3 = T_1 T_2 T_3}) and (\ref{S_i and T_i coprime}), we have $(S_{i,j})^2 \mid T_k$ and $(T_{i,j})^2 \mid S_k$ for $i,j,k$ distinct. Furthermore, $(S_{i_1 , j_1} , S_{i_2 , j_2}) =1$ and $(T_{i_1 , j_1} , T_{i_2 , j_2}) =1$ for all $\{ i_1 , j_1 \} \neq \{ i_2 , j_2 \}$, and $(S_{i_1 , j_1} , T_{i_2 , j_2}) =1$ for all $i_1 , j_1 , i_2 , j_2$.\\

From these three points we can deduce that
\begin{align*}
S_1 = S_{1,2} S_{1,3} (T_{2,3})^2 {S_1}' \quad \quad &T_1 = T_{1,2} T_{1,3} (S_{2,3})^2 {T_1}' \\
S_2 = S_{1,2} S_{2,3} (T_{1,3})^2 {S_2}' \quad \quad &T_2 = T_{1,2} T_{2,3} (S_{1,3})^2 {T_2}' \\
S_3 = S_{1,3} S_{2,3} (T_{1,2})^2 {S_3}' \quad \quad &T_3 = T_{1,3} T_{2,3} (S_{1,2})^2 {T_3}' 
\end{align*}
for some ${S_i}'$ and ${T_i}'$ satisfying $({S_i}' , {T_i}' )=1$ for all $i$ and $({S_i}' , {S_j}' ) , ({T_i}' , {T_j}' ) =1$ for $i \neq j$. By (\ref{S_1 S_2 S_3 = T_1 T_2 T_3}) we have that ${S_1}' {S_2}' {S_3}' = {T_1}' {T_2}' {T_3}'$. From these points we can deduce that
\begin{align*}
{S_1}' = U_{1,2} U_{1,3} \quad \quad &{T_1}' = U_{2,1} U_{3,1} \\
{S_2}' = U_{2,1} U_{2,3} \quad \quad &{T_2}' = U_{1,2} U_{3,2} \\
{S_3}' = U_{3,1} U_{3,2} \quad \quad &{T_3}' = U_{1,3} U_{2,3}
\end{align*}
where the $U_{i,j}$ are pairwise coprime. Also, for $i,j,k$ distinct, because $U_{i,j} \mid T_j$ and $(S_j , T_j ) =1$, we have that $(U_{i,j} , S_j ) = 1$, and hence $(U_{i,j} , S_{j,k}) , (U_{i,j} , S_{j,i}) = 1$. Similarly, for $i,j,k$ distinct, we have $(U_{i,j} , T_{i,k}) , (U_{i,j} , T_{i,j}) = 1$. \\

So, by defining
\begin{align*}
V_{1,2} = S_{1,3} T_{2,3} U_{1,2} \quad \quad V_{2,1} = S_{2,3} T_{1,3} U_{2,1} \quad \quad V_{3,1} = S_{2,3} T_{1,2} U_{3,1} \\
V_{1,3} = S_{1,2} T_{2,3} U_{1,3} \quad \quad V_{2,3} = S_{1,2} T_{1,3} U_{2,3} \quad \quad V_{3,2} = S_{1,3} T_{1,2} U_{3,2}
\end{align*}
we complete the proof for the existence claim. \\

Uniqueness follows from the following observation: If we have $G_i$ and $V_{i,j}$ satisfying the conditions in the Lemma, then we can deduce 
\begin{align*}
&G_i = (A_i , B_i ) \quad \text{ for all $i$, and} \\
&V_{i,j} = \Big( V_{i,j} V_{k,j} , \frac{ V_{i,j} V_{k,j} V_{j,i} V_{k,i} }{V_{k,i} V_{k,j} } \Big) 
	= \Big( \hat{B}_j , \frac{ \hat{B}_i \hat{B}_j }{\hat{A}_k } \Big) \quad \text{ for $i,j,k$ distinct,}
\end{align*}
where we define $\hat{B}_i , \hat{A}_i$ by $B_i = G_i \hat{B}_i = (A_i , B_i ) \hat{B}_i$ and $A_i = G_i \hat{A}_i = (A_i , B_i ) \hat{A}_i$ for all $i$. Since the far right side of each line above is expressed entirely in terms of $A_1 , A_2 , A_3 , B_1 , B_2 , B_3$, we must have uniqueness.
\end{proof}

\begin{lemma} \label{V = V_(1,2) V_(2,1)}
Suppose $V_{1,3} , V_{2,3} , V_{3,1} , V_{3,2} \in \mathcal{M}$, and $(V_{1,3} , V_{3,1} V_{3,2}) = 1$ and $(V_{2,3} , V_{3,1} V_{3,2}) = 1$. Then,
\begin{align*}
&\Big\{ (V_{1,2} , V_{2,1}) \in \mathcal{M}^2 : (V_{1,2} , V_{2,3} V_{3,1}) = 1 , (V_{2,1} , V_{1,3} V_{3,2}) = 1 , (V_{1,2} , V_{2,1}) = 1 \Big\} \\
&= \bigcup_{\substack{V \in \mathcal{M} \\ \big(V , (V_{1,3} V_{3,1} , V_{2,3}  V_{3,2}) \big) = 1}} \Big\{ ( V_{1,2} , V_{2,1} ) \in \mathcal{M}^2 :  \\
	&\hspace{6em} V_{1,2} V_{2,1} = V , (V_{1,2} , V_{2,3} V_{3,1}) = 1 , (V_{2,1} , V_{1,3} V_{3,2}) = 1 , (V_{1,2} , V_{2,1}) = 1 \Big\} ,
\end{align*}
and for each such $V$ we have
\begin{align*}
&\# \Big\{ (V_{1,2} , V_{2,1}) \in \mathcal{M}^2 : V_{1,2} V_{2,1} = V , (V_{1,2} , V_{2,3} V_{3,1}) = 1 , (V_{2,1} , V_{1,3} V_{3,2}) = 1 , (V_{1,2} , V_{2,1}) = 1 \Big\} \\
= & 2^{\omega(V) - \omega \Big( \big(V , V_{1,3} V_{2,3} V_{3,1} V_{3,2} \big) \Big)} .
\end{align*}
\end{lemma}

\begin{proof}
For the first claim we note that $(V_{1,2} , V_{2,3} V_{3,1}) = 1$ and $(V_{2,1} , V_{1,3} V_{3,2}) = 1$ imply that
\begin{align*}
\Big(V , (V_{1,3} , V_{2,3}) \cdot (V_{3,1} , V_{3,2}) \Big) = 1,
\end{align*}
and, due to the given coprimality relations of $V_{1,3}$,$V_{2,3}$,$V_{3,1}$, and $V_{3,2}$ given in Lemma \ref{A_1 A_2 A_3 = B_1 B_2 B_3 rewrite}, we have
\begin{align*}
(V_{1,3} , V_{2,3}) \cdot (V_{3,1} , V_{3,2})
= (V_{1,3} V_{3,1} , V_{2,3}  V_{3,2}) .
\end{align*}
The first claim follows.\\

We now look at the second claim. For $A,B \in \mathcal{M}$, we define $A_B$ to be the maximal divisor of $A$ that is coprime to $B$, and we define $A^B$ by $A = A_B A^B$. We then have that
\begin{align*}
V
= V_{V_{1,3} V_{2,3} V_{3,1} V_{3,2}} V^{V_{1,3} V_{2,3} V_{3,1} V_{3,2}}
= V_{V_{1,3} V_{2,3} V_{3,1} V_{3,2}} V^{V_{1,3}} V^{V_{2,3}} V^{V_{3,1}} V^{V_{3,2}} ,
\end{align*}
where the last equality follows from $\big(V , (V_{1,3} V_{3,1} , V_{2,3}  V_{3,2}) \big) = 1$ and the fact that $(V_{1,3} , V_{3,1}) = 1$ and $(V_{2,3} , V_{3,2}) = 1$. Now, $V = V_{1,2} V_{2,1}$ and by the coprimality relations we must have that $V^{V_{1,3}} V^{V_{3,2}} \mid V_{1,2}$ and $V^{V_{2,3}} V^{V_{3,1}} \mid V_{2,1}$. So, we see that
\begin{align*}
&\# \Big\{ (V_{1,2} , V_{2,1}) \in \mathcal{M}^2 : V_{1,2} V_{2,1} = V , (V_{1,2} , V_{2,3} V_{3,1}) = 1 , (V_{2,1} , V_{1,3} V_{3,2}) = 1 , (V_{1,2} , V_{2,1}) = 1 \Big\} \\
= &\# \Big\{ (V_{1,2} , V_{2,1}) \in \mathcal{M}^2 : V_{1,2} V_{2,1} = V_{V_{1,3} V_{2,3} V_{3,1} V_{3,2}} V^{V_{1,3}} V^{V_{2,3}} V^{V_{3,1}} V^{V_{3,2}} , \\
	&\hspace{2em}  V^{V_{1,3}} V^{V_{3,2}} \mid V_{1,2} \; , \; V^{V_{2,3}} V^{V_{3,1}} \mid V_{2,1} \; , \; (V_{1,2} , V_{2,1}) = 1 \Big\} \\
= &2^{\omega \big( V_{V_{1,3} V_{2,3} V_{3,1} V_{3,2}} \big) } = 2^{\omega(V) - \omega \Big( \big(V , V_{1,3} V_{2,3} V_{3,1} V_{3,2} \big) \Big)} .
\end{align*}
\end{proof}

We now need to give a definition for the primorials in $\mathbb{F}_q [T]$.

\begin{definition}[Primorial Polynomials] \label{definition, primorial polynomials}
Let $(S_i)_{i \in \mathbb{Z}_{> 0}}$ be a fixed ordering of $\mathcal{P}$ such that $\degree S_i \leq \degree S_{i+1}$ for all $i \geq 1$ (the order of the primes of a given degree is not of importance here). For all positive integers $n$ we define
\begin{align*}
R_n
:= \prod_{i=1}^{n} S_i .
\end{align*}
We will refer to $R_n$ as the $n$-th primorial. For each positive integer $n$ we have unique non-negative integers $m_n$ and $r_n$ such that
\begin{align} \label{R_n decomposition}
R_n
= \bigg( \prod_{\degree P \leq m_n } P \bigg) \bigg( \prod_{i=1}^{r_n } Q_i \bigg) ,
\end{align}
where the $Q_i$ are distinct primes of degree $m_n +1$. This definition of primorial is not standard.
\end{definition}

\begin{lemma} \label{Primorial bound}
For all positive integers $n$ we have that
\begin{align*}
\log_q \log_q \lvert R_n \rvert
= m_n + O(1) .
\end{align*}
From this we can deduce that
\begin{align*}
m_n \ll \log_q \log_q \lvert R_n \rvert
\end{align*}
for $n$ satisfying $m_n \geq 1$. In particular, the implied constant is independent of $q$.
\end{lemma}

\begin{proof}
For the first claim, by (\ref{R_n decomposition}) and (\ref{prime poly theorem, precise, statement}), we see that
\begin{align*}
\log_q \lvert R_n \rvert 
= \degree R_n 
\leq \sum_{i=1}^{m_n +1} \bigg( q^i + O \Big( q^{\frac{i}{2}} \Big) \bigg)
\ll  q^{m_n +1}
\end{align*}
and
\begin{align*}
\log_q \lvert R_n \rvert 
= \degree R_n 
\geq \sum_{i=1}^{m_n} \bigg( q^i + O \Big( q^{\frac{i}{2}} \Big) \bigg)
\gg  q^{m_n} .
\end{align*}
By taking logarithms of both equations above, we deduce that 
\begin{align*}
\log_q \log_q \lvert R_n \rvert
= m_n + O (1) .
\end{align*}

For the second claim, if $m_n \geq 1$ then $\log_q \log_q \lvert R_n \rvert \geq 1$, and so by the first claim we have
\begin{align*} 
\frac{m_n}{\log_q \log_q \lvert R_n \rvert }
\ll 1 + \frac{1}{\log_q \log_q \lvert R_n \rvert }
\ll 1 .
\end{align*}
\end{proof}

\begin{lemma} \label{Prod of P div R of 1-P^(-1) derivatives bounds}
For all $R \in \mathcal{M}$ with $\degree R \geq 1$, non-negative integers $k$, and $s \in \mathbb{C}$ with $\Re (s) > -1$ we define
\begin{align*}
f_{R,k} (s) := &\prod_{P \mid R} \Big( 1 - \lvert P \rvert^{-s-1} \Big)^k , \\
h_{R,k} (s) := &\prod_{P \mid R} \Big( 1 + \lvert P \rvert^{-s-1} \Big)^{-k} .
\end{align*}
Then, for all non-negative integers $j$ and all integers $r$ we have
\begin{align*}
f_{R,k}^{(j)} \Big( \frac{2 r \pi i}{\log q} \Big) &\ll_j k^j \big( \log_q \degree R + O(1) \big)^{j}  \prod_{P \mid R} \Big( 1 - \lvert P \rvert^{-1} \Big)^k , \\
h_{R,k}^{(j)} \Big( \frac{2 r \pi i}{\log q} \Big) &\ll_j k^j \big( \log_q \degree R + O(1) \big)^{j}  \prod_{P \mid R} \Big( 1 + \lvert P \rvert^{-1} \Big)^{-k} .
\end{align*}
Generally, we could incorporate the $O(1)$ terms into the relation $\ll_j$, but for the case $\degree R = 1$, where we would have $\log_q \degree R = 0$, the $O(1)$ terms are required.
\end{lemma}

\begin{proof}
We will prove only the claim for $f_{R,k} (s)$ and $r=0$. The proofs for all $r$ and $h_{R,k} (s)$ are almost identical. First, we note that
\begin{align} \label{Prod of P div R of 1-P^(-1) first derivative}
f_{R,k} '(s)
= k \, g_{R} (s) f_{R,k} (s) ,
\end{align}
where
\begin{align*}
g_{R} (s) 
:= \sum_{P \mid R} \frac{\log \lvert P \rvert}{\lvert P \rvert^{s+1} - 1} .
\end{align*}

We note further that, for integers $j \geq 1$,
\begin{align} \label{Prod of P div R of 1-P^(-1) all derivatives}
f_{R,k}^{(j)} (s) = G_{R,k,j} (s) f_{R,k} (s) ,
\end{align}
where $G_{R,k,j} (s)$ is a sum of terms of the form
\begin{align}\label{Prod of P div R of 1-P^(-1) derivative terms}
k^m g_{R}^{(j_1)} (s) \; g_{R}^{(j_2)} (s) \; \ldots \; g_{R}^{(j_m)} (s) ,
\end{align}
where $1 \leq m \leq j$ and $\sum_{r=1}^{m} (j_r + 1) = j$. The number of such terms and their coefficients are dependent only on $j$. \\

Now, for all $R \in \mathcal{M}$, and non-negative integers $l$, it is not difficult to deduce that
\begin{align} \label{Prod of P div R of 1-P^(-1) , g derivatives bounds}
g_{R}^{(l)} (0)
\ll_{l} \sum_{P \mid R} \frac{ \big( \log \lvert P \rvert \big)^{l+1} }{ \lvert P \rvert - 1} .
\end{align}
The function $\frac{\big( \log x \big)^{l+1}}{x-1}$ is decreasing at large enough $x$, and the limit as $x \longrightarrow \infty$ is $0$. Therefore, there exists a constant $c_l > 0$ such that for all $A,B \in \mathcal{A}$ with $1 \leq \degree A \leq \degree B$ we have that
\begin{align*}
c_l \frac{\big( \log \lvert A \rvert \big)^{l+1}}{\lvert A \rvert -1}
\geq \frac{\big( \log \lvert B \rvert \big)^{l+1}}{\lvert B \rvert -1} .
\end{align*}
Hence, taking $n=\omega (R)$ and using Definition \ref{definition, primorial polynomials}, Lemma \ref{Primorial bound}, and the prime polynomial theorem, we see that
\begin{align}
\begin{split} \label{sum P div R of (log P)^(i+1) / P-1 bound}
&\sum_{P \mid R} \frac{ \big( \log \lvert P \rvert \big)^{l+1} }{\lvert P \rvert - 1}
\ll_l \sum_{P \mid R_n } \frac{ \big( \log \lvert P \rvert \big)^{l+1} }{\lvert P \rvert - 1}
\ll \sum_{r=1}^{m_n +1 } \frac{q^r}{r} \frac{r^{l+1}}{q^r - 1}
\ll \sum_{r=1}^{m_n +1 } r^l 
\ll (m_n + 1 )^{l+1} \\
\ll &\big( \log_q \log_q \lvert R_n \rvert + O(1) \big)^{l+1}
\ll \big( \log_q \degree R_n + O(1) \big)^{l+1} 
\ll \big( \log_q \degree R + O(1) \big)^{l+1} .
\end{split}
\end{align}

The result follows by (\ref{Prod of P div R of 1-P^(-1) all derivatives}), (\ref{Prod of P div R of 1-P^(-1) derivative terms}), (\ref{Prod of P div R of 1-P^(-1) , g derivatives bounds}), and (\ref{sum P div R of (log P)^(i+1) / P-1 bound}) .
\end{proof}

\begin{lemma}[Perron's Formula] \label{integral over Re s = c of (y/n)^s / s^3}
Let $c$ be a positive real number, and let $k \geq 2$ be an integer. Then,
\begin{align*}
\int_{c- i \infty}^{c+ i \infty} \frac{y^s}{s^k} \mathrm{d} s
= \begin{cases}
0 &\text{ if $0 \leq y < 1$} ; \\
\frac{2\pi i}{(k-1)!} (\log y)^{k-1} &\text{ if $y \geq 1$} .
\end{cases} 
\end{align*}
If $k=1$, then
\begin{align*}
\int_{c- i \infty}^{c+ i \infty} \frac{y^s}{s} \mathrm{d} s
= \begin{cases}
0 &\text{ if $0 \leq y < 1$} ; \\
\pi i &\text{ if $y = 1$} ; \\
2\pi i &\text{ if $y > 1$} .
\end{cases} 
\end{align*}
\end{lemma}

\begin{proof}
See \cite[4.1.6, Page 282]{PropAnalNumTheo_Murty}
\end{proof}

\begin{lemma} \label{Proposition Sum of 2^(omega(V) - omega((V,W)) (z - deg V)^k / |V|}
Let $R,M \in \mathcal{M}$ with $\degree M \leq \degree R$, $k$ be a non-negative integer, and $z$ be an integer-valued function of $R$ such that $z \sim \degree R$ as $\degree R \longrightarrow \infty$. We have that
\begin{align*}
\begin{split}
&\sum_{\substack{N \in \mathcal{M} \\ \degree N \leq z \\ (N,R) = 1}} \frac{2^{\omega (N) - \omega \big( (N,M) \big)}}{\lvert N \rvert} ( z - \degree N )^k \\
= &\frac{(1-q^{-1})}{(k+2) (k+1)} 
	\prod_{P \mid MR} \bigg( \frac{1 - \lvert P \rvert^{-1}}{1 + \lvert P \rvert^{-1}} \bigg) 
	\prod_{\substack{P \mid M \\ P \nmid R}} \bigg( \frac{1}{1 - \lvert P \rvert^{-1}} \bigg)
	\bigg( z^{k+2} + O_k \big( z^{k+1} \log \degree R \big) \bigg)
\end{split}
\end{align*}
as $\degree R \longrightarrow \infty$.
\end{lemma}

\begin{proof}
\textbf{Step 1:} Let us define the function $F$, for $\Re s > 1$, by
\begin{align*}
F(s)
= \sum_{\substack{N \in \mathcal{M} \\ (N,R) = 1}} \frac{2^{\omega (N) - \omega \big( (N,M) \big)}}{\lvert N \rvert^s} .
\end{align*}
We can see that
\begin{align*}
F(s)
&= \prod_{P \nmid MR} \bigg( 1 + \frac{2}{\lvert P \rvert^s} + \frac{2}{\lvert P \rvert^{2s}} + \ldots \bigg) 
	\prod_{\substack{P \mid M \\ P \nmid R}} \bigg( 1 + \frac{1}{\lvert P \rvert^s} + \frac{1}{\lvert P \rvert^{2s}} + \ldots \bigg) \\
= &\prod_{P \nmid MR} \bigg( \frac{2}{1 - \lvert P \rvert^{-s}} - 1 \bigg) 
	\prod_{\substack{P \mid M \\ P \nmid R}} \bigg( \frac{1}{1 - \lvert P \rvert^{-s}} \bigg)  \\
= &\prod_{P \in \mathcal{P}} \bigg( \frac{1 + \lvert P \rvert^{-s}}{1 - \lvert P \rvert^{-s}} \bigg) 
	\prod_{P \mid MR} \bigg( \frac{1 - \lvert P \rvert^{-s}}{1 + \lvert P \rvert^{-s}} \bigg) 
	\prod_{\substack{P \mid M \\ P \nmid R}} \bigg( \frac{1}{1 - \lvert P \rvert^{-s}} \bigg)  \\
= &\frac{\zeta_{\mathcal{A}} (s)^2}{\zeta_{\mathcal{A}} (2s)}
	\prod_{P \mid MR} \bigg( \frac{1 - \lvert P \rvert^{-s}}{1 + \lvert P \rvert^{-s}} \bigg) 
	\prod_{\substack{P \mid M \\ P \nmid R}} \bigg( \frac{1}{1 - \lvert P \rvert^{-s}} \bigg) .
\end{align*}

Now, let $c$ be a positive real number, and define
\begin{align*}
y := \begin{cases}
q^{z + \frac{1}{2}} &\text{ if $k = 0$} \\
q^z &\text{ if $k \neq 0$} .
\end{cases}
\end{align*}
On the one hand, we have that
\begin{align}
\begin{split} \label{Inside Int F(1+s) y^s / s^(k+1)}
\frac{1}{2 \pi i} \int_{c - i \infty}^{c+ i \infty} F(1+s) \frac{{y}^s}{s^{k+1}} \mathrm{d} s 
= &\frac{1}{2 \pi i} \sum_{\substack{N \in \mathcal{M} \\ (N,R) = 1}} \frac{2^{\omega (N) - \omega \big( (N,M) \big)}}{\lvert N \rvert} \int_{c - i \infty}^{c+ i \infty} \frac{{y}^s}{\lvert N \rvert ^s s^{k+1}} \mathrm{d} s \\
= &\frac{(\log q)^k}{k!}\sum_{\substack{N \in \mathcal{M} \\ \degree N \leq z \\ (N,R) = 1}} \frac{2^{\omega (N) - \omega \big( (N,M) \big)}}{\lvert N \rvert} \big( z - \degree N \big)^k .
\end{split}
\end{align}
For $k \geq 1$, the interchange of integral and summation is justified by absolute convergence, and the second equality follows by Lemma \ref{integral over Re s = c of (y/n)^s / s^3}. For $k=0$, the above holds by Lemma \ref{lemma, sum over deg N < z of 2^(omega(N) - omega (N,M))/N, eval via contours} below. We remark that we take $y=q^{z+\frac{1}{2}}$ when $k=0$ so that $\Big( \frac{y}{\lvert N \rvert} , k+1 \Big) \neq (1,1)$, which would be a special case of Lemma \ref{integral over Re s = c of (y/n)^s / s^3} that would be tedious to address. \\

On the other hand, for all positive integers $n$ define the following curves: 
\begin{align*}
l_1 (n) := &\bigg[ c - \frac{\big( 2n + \frac{1}{2} \big) \pi i}{\log q} , c + \frac{\big( 2n + \frac{1}{2} \big) \pi i}{\log q} \bigg] ; \\
l_2 (n) := &\bigg[ c + \frac{\big( 2n + \frac{1}{2} \big) \pi i}{\log q} , -\frac{1}{4} + \frac{\big( 2n + \frac{1}{2} \big) \pi i}{\log q} \bigg] ; \\
l_3 (n) := &\bigg[ -\frac{1}{4} + \frac{\big( 2n + \frac{1}{2} \big) \pi i}{\log q} , -\frac{1}{4} - \frac{\big( 2n + \frac{1}{2} \big) \pi i}{\log q} \bigg] ; \\
l_4 (n) := &\bigg[ -\frac{1}{4} - \frac{\big( 2n + \frac{1}{2} \big) \pi i}{\log q} , c - \frac{\big( 2n + \frac{1}{2} \big) \pi i}{\log q} \bigg] ; \\
L (n) := &l_1 (n) \cup l_2 (n) \cup l_3 (n) \cup l_4 (n) .
\end{align*}
Then, we have that
\begin{align}
\begin{split}\label{Outside Int F(1+s) y^s / s^(k+1)}
\frac{1}{2\pi i} \int_{c - i \infty}^{c+ i \infty} F(1+s) \frac{{y}^s}{s^{k+1}} \mathrm{d} s 
= &\frac{1}{2\pi i} \lim_{n \rightarrow \infty} \Bigg( \int_{L (n)} F(1+s) \frac{{y}^s}{s^{k+1}} \mathrm{d} s - \int_{l_2 (n)} F(1+s) \frac{{y}^s}{s^{k+1}} \mathrm{d} s \\
& \hspace{4em} - \int_{l_3 (n)} F(1+s) \frac{{y}^s}{s^{k+1}} \mathrm{d} s - \int_{l_4 (n)} F(1+s) \frac{{y}^s}{s^{k+1}} \mathrm{d} s \bigg) . 
\end{split}
\end{align}

\textbf{Step 2:} For the first integral in (\ref{Outside Int F(1+s) y^s / s^(k+1)}) we note that $F(1+s) \frac{{y}^s}{s^{k+1}}$ has a pole at $s=0$ of order $k+3$ and double poles at $s = \frac{2m \pi i}{\log q}$ for $m = \pm 1 , \pm 2 , \ldots , \pm n$. By applying the residue theorem we see that
\begin{align} 
\begin{split} \label{Proposition Sum of 2^(omega(V) - omega((V,W)) (deg V)^k / |V| , residues}
\lim_{n \rightarrow \infty} \frac{1}{2\pi i} \int_{L(n)} F(1+s) \frac{{y}^s}{s^{k+1}} \mathrm{d} s 
= \residue_{s=0} F(s+1)\frac{{y}^s}{s^{k+1}} + \sum_{\substack{m \in \mathbb{Z} \\ m \neq 0}} \residue_{s=\frac{2m \pi i}{\log q}} F(1+s)\frac{{y}^s}{s^{k+1}} . 
\end{split}
\end{align}

\textbf{Step 2.1:} For the first residue term we have
\begin{align} 
\begin{split}\label{Proposition Sum of 2^(omega(V) - omega((V,W)) (deg V)^k / |V| , residue at 0}
&\residue_{s=0} F(s+1)\frac{{y}^s}{s^{k+1}} \\
= &\frac{1}{(k+2)!} \lim_{s \longrightarrow 0} \frac{\mathrm{d}^{k+2}}{\mathrm{d}s^{k+2}} \Bigg( \zeta_{\mathcal{A}} (s +1)^2 s^2 
	\frac{1}{\zeta_{\mathcal{A}} (2s+2)} 
	\prod_{P \mid MR} \bigg( \frac{1 - \lvert P \rvert^{-s-1}}{1 + \lvert P \rvert^{-s-1}} \bigg) 
	\prod_{\substack{P \mid M \\ P \nmid R}} \bigg( \frac{1}{1 - \lvert P \rvert^{-s-1}} \bigg)
	{y}^s \Bigg) .
	\end{split}
\end{align}

If we apply the product rule for differentiation, then one of the terms will be
\begin{align*} 
&\frac{1}{(k+2)!} \lim_{s \longrightarrow 0} 
	\Bigg( \zeta_{\mathcal{A}} (s +1)^2 s^2 
	\frac{1}{\zeta_{\mathcal{A}} (2s+2)} 
	\prod_{P \mid MR} \bigg( \frac{1 - \lvert P \rvert^{-s-1}}{1 + \lvert P \rvert^{-s-1}} \bigg) 
	\prod_{\substack{P \mid M \\ P \nmid R}} \bigg( \frac{1}{1 - \lvert P \rvert^{-s-1}} \bigg)
	\frac{\mathrm{d}^{k+2}}{\mathrm{d}s^{k+2}} {y}^s \Bigg) \\
= & \frac{(1-q^{-1}) (\log q )^k}{(k+2)!} 
	\prod_{P \mid MR} \bigg( \frac{1 - \lvert P \rvert^{-1}}{1 + \lvert P \rvert^{-1}} \bigg) 
	\prod_{\substack{P \mid M \\ P \nmid R}} \bigg( \frac{1}{1 - \lvert P \rvert^{-1}} \bigg)
	\Big( z + O(1) \Big)^{k+2} .
\end{align*}
The $O(1)$ term is to account for the case where $y=q^{z+ \frac{1}{2}}$ (when $k=0$). \\

Now we look at the remaining terms that arise from the product rule. By using the fact that $\zeta_{\mathcal{A}} (1+s) = \frac{1}{1-q^{-s}}$, the Taylor series for $q^{-s}$, and the chain rule, we have, for non-negative integers $i$, that
\begin{align} \label{Proposition Sum of 2^(omega(V) - omega((V,W)) (deg V)^k / |V| , zeta(s+1) s derivatives at 0}
\lim_{s \rightarrow 0} \frac{1}{( \log q)^{i-1}} \frac{\mathrm{d}^i}{\mathrm{d} s^i} \zeta (s+1) s
= O_i (1) .
\end{align}
Similarly, for non-negative integers $i$,
\begin{align} \label{Proposition Sum of 2^(omega(V) - omega((V,W)) (deg V)^k / |V| , zeta(2s+2)^-1 derivatives at 0}
\frac{1}{( \log q)^{i}} \lim_{s \rightarrow 0} \frac{\mathrm{d}^i}{\mathrm{d} s^i}\zeta (2s+2)^{-1}
= \frac{1}{( \log q)^{i}} \lim_{s \rightarrow 0} \frac{\mathrm{d}^i}{\mathrm{d} s^i} \Big( 1 - q^{-1-2s} \Big)
= O_i (1) .
\end{align}
By (\ref{Proposition Sum of 2^(omega(V) - omega((V,W)) (deg V)^k / |V| , zeta(s+1) s derivatives at 0}), (\ref{Proposition Sum of 2^(omega(V) - omega((V,W)) (deg V)^k / |V| , zeta(2s+2)^-1 derivatives at 0}), and Lemma \ref{Prod of P div R of 1-P^(-1) derivatives bounds} and the fact that $\degree M \leq \degree R$, we see that the remaining terms are
\begin{align*}
\ll_k \frac{( \log q )^k}{(k+2)!}
	\prod_{P \mid MR} \bigg( \frac{1 - \lvert P \rvert^{-1}}{1 + \lvert P \rvert^{-1}} \bigg) 
	\prod_{\substack{P \mid M \\ P \nmid R}} \bigg( \frac{1}{1 - \lvert P \rvert^{-1}} \bigg)
	z^{k+1} \log \degree R .
\end{align*}

Hence,
\begin{align} 
\begin{split} \label{Proposition Sum of 2^(omega(V) - omega((V,W)) (deg V)^k / |V| , residue at 0 calculated}
&\residue_{s=0} F(s+1)\frac{{y}^s}{s^{k+1}} \\
= &\frac{(1-q^{-1}) (\log q )^k}{(k+2)!} 
	\hspace{-0.5em} \prod_{P \mid MR} \hspace{-0.5em} \bigg( \frac{1 - \lvert P \rvert^{-1}}{1 + \lvert P \rvert^{-1}} \bigg) 
	\prod_{\substack{P \mid M \\ P \nmid R}} \bigg( \frac{1}{1 - \lvert P \rvert^{-1}} \bigg)
	\bigg( z^{k+2} + O_k \big( z^{k+1} \log \degree R \big) \bigg)
\end{split}
\end{align}
as $\degree R \longrightarrow \infty$. \\

\textbf{Step 2.2:} Now we look at the remaining residue terms in (\ref{Proposition Sum of 2^(omega(V) - omega((V,W)) (deg V)^k / |V| , residues}). By similar (but simpler) means as above we can show that
\begin{align*}
\residue_{s=\frac{2m \pi i}{\log q}} F(1+s)\frac{{y}^s}{s^{k+1}}
= O_k \Bigg( \frac{ (\log q )^k}{m^{k+1}}
	\prod_{P \mid MR} \bigg( \frac{1 - \lvert P \rvert^{-1}}{1 + \lvert P \rvert^{-1}} \bigg) 
	\prod_{\substack{P \mid M \\ P \nmid R}} \bigg( \frac{1}{1 - \lvert P \rvert^{-1}} \bigg)
	z \Bigg) 
\end{align*}
as $\degree R \longrightarrow \infty$, and so, for $k \geq 1$,
\begin{align} \label{Proposition Sum of 2^(omega(V) - omega((V,W)) (deg V)^k / |V| , other residues calculated}
\sum_{\substack{m \in \mathbb{Z} \\ m \neq 0}} \residue_{s=\frac{2m \pi i}{\log q}} F(1+s)\frac{{y_R}^s}{s^{k+1}}
= O_k \Bigg( (\log q )^k
	\prod_{P \mid MR} \bigg( \frac{1 - \lvert P \rvert^{-1}}{1 + \lvert P \rvert^{-1}} \bigg) 
	\prod_{\substack{P \mid M \\ P \nmid R}} \bigg( \frac{1}{1 - \lvert P \rvert^{-1}} \bigg)
	z \Bigg) 
\end{align}
as $\degree R \longrightarrow \infty$. When $k=0$ we look at things more precisely and see that the term $\frac{1}{m}$ cancels with the term with $\frac{1}{-m}$, and so (\ref{Proposition Sum of 2^(omega(V) - omega((V,W)) (deg V)^k / |V| , other residues calculated}) holds for $k=0$ as well. \\

\textbf{Step 2.3:} By (\ref{Proposition Sum of 2^(omega(V) - omega((V,W)) (deg V)^k / |V| , residues}), (\ref{Proposition Sum of 2^(omega(V) - omega((V,W)) (deg V)^k / |V| , residue at 0 calculated}) and (\ref{Proposition Sum of 2^(omega(V) - omega((V,W)) (deg V)^k / |V| , other residues calculated}), we see that
\begin{align}
\begin{split} \label{Proposition Sum of 2^(omega(V) - omega((V,W)) (deg V)^k / |V| , L(n) integral}
&\lim_{n \rightarrow \infty} \frac{1}{2 \pi i} \int_{L(n)} F(1+s) \frac{{y_R}^s}{s^3} \mathrm{d} s \\
= &\frac{(1-q^{-1}) (\log q )^k}{(k+2)!} 
	\hspace{-0.5em} \prod_{P \mid MR} \hspace{-0.5em} \bigg( \frac{1 - \lvert P \rvert^{-1}}{1 + \lvert P \rvert^{-1}} \bigg) 
	\prod_{\substack{P \mid M \\ P \nmid R}} \bigg( \frac{1}{1 - \lvert P \rvert^{-1}} \bigg)
	\bigg( z^{k+2} + O_k \big( z^{k+1} \log \degree R \big) \bigg)
\end{split}
\end{align}
as $\degree R \longrightarrow \infty$. \\

\textbf{Step 3:} We now look at the integrals over $l_2 (n)$ and $l_4 (n)$. For all positive integers $n$ and all $s \in l_2 (n) , l_4 (n)$ we have that $F(s+1) {y}^s = O_{q,R,c} (1)$. One can now easily deduce for $i = 2,4$ that
\begin{align} \label{Proposition Sum of 2^(omega(V) - omega((V,W)) (deg V)^k / |V| , l_2 (n) and l_4 (n) integrals}
\lim_{n \rightarrow \infty} \Big\lvert \frac{1}{2 \pi i} \int_{l_i (n)} F(1+s) \frac{{y}^s}{s^{k+1}} \mathrm{d} s \Big\rvert = 0 .
\end{align}

\textbf{Step 4:} We now look at the integral over $l_3 (n)$. For all positive integers $n$ and all $s \in l_3 (n)$ we have that
\begin{align*}
\frac{\zeta_{\mathcal{A}} (s+1)^2}{\zeta_{\mathcal{A}} (2s+2)}
= O(1) 
\end{align*}
and
\begin{align*}
&\bigg\lvert  \prod_{P \mid MR} \bigg( \frac{1 - \lvert P \rvert^{-s-1}}{1 + \lvert P \rvert^{-s-1}} \bigg)
	\prod_{\substack{P \mid M \\ P \nmid R}} \bigg( \frac{1}{1 - \lvert P \rvert^{-s-1}} \bigg)
	{y}^{s} \bigg\lvert \\
\ll & \prod_{P \mid R} \bigg( \frac{1 + \lvert P \rvert^{-\frac{3}{4}}}{1 - \lvert P \rvert^{-\frac{3}{4}}} \bigg) 
	\prod_{P \mid M} \bigg( \frac{1}{1 - \lvert P \rvert^{-\frac{3}{4}}} \bigg) 
	\lvert R \rvert^{-\frac{1}{12}} \lvert R \rvert^{-\frac{1}{12}} \lvert M \rvert^{-\frac{1}{12}} q^{o(\degree R)} \\
\ll & \prod_{P \mid R} \bigg( \lvert P \rvert^{-\frac{1}{12}} \frac{1 + \lvert P \rvert^{-\frac{3}{4}}}{1 - \lvert P \rvert^{-\frac{3}{4}}} \bigg) 
	\prod_{P \mid M} \bigg( \lvert P \rvert^{-\frac{1}{12}} \frac{1}{1 - \lvert P \rvert^{-\frac{3}{4}}} \bigg) 
	q^{o(\degree R) - \frac{1}{12} \degree R} \\
\ll & O(1) 
\end{align*}
as $\degree R \longrightarrow \infty$. We now easily deduce that, for $k \geq 1$,
\begin{align} \label{Proposition Sum of 2^(omega(V) - omega((V,W)) (deg V)^k / |V| , l_3 (n) integral}
\lim_{n \rightarrow \infty} \Big\lvert \frac{1}{2\pi i} \int_{l_3 (n)} F(1+s) \frac{{y}^s}{s^{k+1}} \mathrm{d} s \Big\rvert
= O(1) 
\end{align}
as $\degree R \longrightarrow \infty$. For the case $k=0$ we must be more careful. Using the fact that $F(1+s)$ has vertical periodicity with period $\frac{2 \pi i}{\log q}$, and the fact that $y=q^{z+ \frac{1}{2}}$ where $z$ is an integer, we have that
\begin{align*}
\int_{-\frac{1}{4}}^{-\frac{1}{4} + i \infty} F(1+s) \frac{{y}^s}{s} \mathrm{d} s 
= &\sum_{m=0}^{\infty}
	\int_{-\frac{1}{4} + \frac{4 m \pi i}{\log q} }^{-\frac{1}{4} + \frac{(4m+2) \pi i}{\log q} } F(1+s) \frac{{y}^s}{s} \mathrm{d} s
	+ \int_{-\frac{1}{4} + \frac{(4m+2) \pi i}{\log q} }^{-\frac{1}{4} + \frac{(4m+4) \pi i}{\log q} } F(1+s) \frac{{y}^s}{s} \mathrm{d} s \\
= &\sum_{m=0}^{\infty}
	\int_{-\frac{1}{4} }^{-\frac{1}{4} + \frac{2 \pi i}{\log q} } F(1+s) \frac{{y}^s}{s + \frac{4 m \pi i}{\log q} } \mathrm{d} s
	- \int_{-\frac{1}{4} }^{-\frac{1}{4} + \frac{2 \pi i}{\log q} } F(1+s) \frac{{y}^s}{s + \frac{(4m+2) \pi i}{\log q} } \mathrm{d} s \\
= &\frac{2 \pi i}{\log q} \sum_{m=0}^{\infty}
	\int_{-\frac{1}{4} }^{-\frac{1}{4} + \frac{2 \pi i}{\log q} } F(1+s) \frac{{y}^s}{\Big( s + \frac{4 m \pi i}{\log q} \Big) \Big( s + \frac{(4m+2) \pi i}{\log q} \Big)} \mathrm{d} s \\
= &\frac{2 \pi i}{\log q} \sum_{m=0}^{\infty}
	\int_{-\frac{1}{4} + \frac{4 m \pi i}{\log q} }^{-\frac{1}{4} + \frac{(4 m +2) \pi i}{\log q} } F(1+s) \frac{{y}^s}{ s \Big( s + \frac{2 \pi i}{\log q} \Big)} \mathrm{d} s \\
\ll & \int_{-\frac{1}{4}}^{-\frac{1}{4} + i \infty} \frac{1}{ \lvert s \rvert \cdot \Big\lvert s + \frac{2 \pi i}{\log q} \Big\rvert} \mathrm{d} s
\ll 1 .
\end{align*}
A similar result can be obtained for the integral from $-\frac{1}{4}$ to $- \frac{1}{4} - i \infty$. Hence, we have that
\begin{align} \label{Proposition Sum of 2^(omega(V) - omega((V,W)) (deg V)^k / |V| , l_3 (n) integral, k=0 case}
\lim_{n \rightarrow \infty} \Big\lvert \frac{1}{2\pi i} \int_{l_3 (n)} F(1+s) \frac{{y}^s}{s} \mathrm{d} s \Big\rvert
= O(1) 
\end{align}
as $\degree R \longrightarrow \infty$. \\

\textbf{Step 5:} By (\ref{Inside Int F(1+s) y^s / s^(k+1)}), (\ref{Outside Int F(1+s) y^s / s^(k+1)}), (\ref{Proposition Sum of 2^(omega(V) - omega((V,W)) (deg V)^k / |V| , L(n) integral}), (\ref{Proposition Sum of 2^(omega(V) - omega((V,W)) (deg V)^k / |V| , l_2 (n) and l_4 (n) integrals}), (\ref{Proposition Sum of 2^(omega(V) - omega((V,W)) (deg V)^k / |V| , l_3 (n) integral}) and (\ref{Proposition Sum of 2^(omega(V) - omega((V,W)) (deg V)^k / |V| , l_3 (n) integral, k=0 case}), we deduce that
\begin{align*}
\begin{split}
&\sum_{\substack{N \in \mathcal{M} \\ \degree N \leq z \\ (N,R) = 1}} \frac{2^{\omega (N) - \omega \big( (N,M) \big)}}{\lvert N \rvert} ( z - \degree N )^k \\
= &\frac{(1-q^{-1})}{(k+2) (k+1)} 
	\prod_{P \mid MR} \bigg( \frac{1 - \lvert P \rvert^{-1}}{1 + \lvert P \rvert^{-1}} \bigg) 
	\prod_{\substack{P \mid M \\ P \nmid R}} \bigg( \frac{1}{1 - \lvert P \rvert^{-1}} \bigg)
	\bigg( z^{k+2} + O_k \big( z^{k+1} \log \degree R \big) \bigg)
\end{split}
\end{align*}
as $\degree R \longrightarrow \infty$.
\end{proof}

\begin{lemma} \label{lemma, sum over deg N < z of 2^(omega(N) - omega (N,M))/N, eval via contours}
Let $F(s)$, $z$, and $c$ be as in Lemma \ref{Proposition Sum of 2^(omega(V) - omega((V,W)) (z - deg V)^k / |V|}, and let $y = q^{z+ \frac{1}{2}}$. Then,
\begin{align*}
\frac{1}{2 \pi i} \int_{c- i \infty}^{c + i \infty} F (1+s) \frac{y^s}{s} \mathrm{d} s
= \sum_{\substack{N \in \mathcal{M} \\ \degree N \leq z \\ (N,R)=1 }} \frac{2^{\omega (N) - \omega \big( (N,M) \big)}}{\lvert N \rvert} .
\end{align*}
\end{lemma}

\begin{proof}
Let $w > z + \frac{1}{2}$ and define
\begin{align*}
F_w (s)
:= \sum_{\substack{N \in \mathcal{M} \\ \degree N > w \\ (N,R) = 1 }} \frac{2^{\omega (N) - \omega \big( (N,M) \big)}}{\lvert N \rvert^s} .
\end{align*}
Then,
\begin{align*}
\frac{1}{2 \pi i} \int_{c- i \infty}^{c + i \infty} F (1+s) \frac{y^s}{s} \mathrm{d} s 
= &\frac{1}{2 \pi i} \int_{c- i \infty}^{c + i \infty} \sum_{\substack{N \in \mathcal{M} \\ \degree N \leq w \\ (N,R)=1 }} \frac{2^{\omega (N) - \omega \big( (N,M) \big)}}{\lvert N \rvert^{1+s} } \frac{y^s}{s} \mathrm{d} s
	+ \frac{1}{2 \pi i} \int_{c- i \infty}^{c + i \infty}  F_w (s) \frac{y^s}{s} \mathrm{d} s \\
= &\sum_{\substack{N \in \mathcal{M} \\ \degree N \leq z \\ (N,R)=1 }} \frac{2^{\omega (N) - \omega \big( (N,M) \big)}}{\lvert N \rvert}
	+ \frac{1}{2 \pi i} \int_{c- i \infty}^{c + i \infty}  F_w (s) \frac{y^s}{s} \mathrm{d} s ,
\end{align*}
where we have used Lemma \ref{integral over Re s = c of (y/n)^s / s^3} for the last equality. We must show that the second term on the right side is zero. To this end, we note that
\begin{align*}
F_w (s)
\leq \sum_{\substack{N \in \mathcal{M} \\ \degree N > w }} \frac{1}{\lvert N \rvert^{\Re (s) -1} }
\ll q^{w (2- \Re (s))} ,
\end{align*}
and we define the contours
\begin{align*}
l_1 (n,m) := &[ c - n i , c + n i ] ; \\
l_2 (n,m) := &[ c + n i , m + n i ] ; \\
l_3 (n,m) := &[ m + n i , m - n i ] ; \\
l_4 (n,m) := &[ m - n i , c - n i ] ; \\
L (n,m) := &l_1 (n) \cup l_2 (n) \cup l_3 (n) \cup l_4 (n) .
\end{align*}
We then have that
\begin{align*}
\int_{l_3 (n,m)}  F_w (s) \frac{y^s}{s} \mathrm{d} s
\leq 2\frac{n}{m} q^{2w} \Big( \frac{y}{q^w} \Big)^m
\longrightarrow 0
\end{align*}
as $m \longrightarrow \infty$, since $q^w > q^{z + \frac{1}{2}} = y$. We also have that
\begin{align*}
\int_{c+ n i}^{\infty + ni}  F_w (s) \frac{y^s}{s} \mathrm{d} s
\leq \frac{q^{2w} }{n} \int_{t=c}^{\infty} \Big( \frac{y}{q^w} \Big)^t \mathrm{d} t
\ll O_{z,w,c} (n^{-1})
\longrightarrow 0
\end{align*}
as $n \longrightarrow \infty$, and, similarly,
\begin{align*}
\int_{c- n i}^{\infty - ni}  F_w (s) \frac{y^s}{s} \mathrm{d} s
\longrightarrow 0
\end{align*}
as $n \longrightarrow \infty$. Finally, we note that
\begin{align*}
\int_{L (n,m)}  F_w (s) \frac{y^s}{s} \mathrm{d} s
= 0
\end{align*}
for all positive $n,m$, by the residue theorem. Hence, we can see that
\begin{align*}
\int_{c- i \infty}^{c + i \infty}  F_w (s) \frac{y^s}{s} \mathrm{d} s
= 0.
\end{align*}
as required.
\end{proof}

We now give a Corollary to Lemma \ref{Proposition Sum of 2^(omega(V) - omega((V,W)) (z - deg V)^k / |V|}.

\begin{corollary}\label{Corollary Sum of 2^(omega(V) - omega((V,W)) (deg V)^k / |V|}
Let $R,M \in \mathcal{M}$ with $\degree M \leq \degree R$, $k$ be a non-negative integer, and $z$ be an integer-valued function of $R$ such that $z \sim \degree R$ as $\degree R \longrightarrow \infty$. We have that
\begin{align*}
\begin{split}
&\sum_{\substack{N \in \mathcal{M} \\ \degree N \leq z \\ (N,R) = 1}} \frac{2^{\omega (N) - \omega \big( (N,M) \big)}}{\lvert N \rvert} ( \degree N )^k \\
= &\frac{(1-q^{-1})}{(k+2)} 
	\prod_{P \mid MR} \bigg( \frac{1 - \lvert P \rvert^{-1}}{1 + \lvert P \rvert^{-1}} \bigg) 
	\prod_{\substack{P \mid M \\ P \nmid R}} \bigg( \frac{1}{1 - \lvert P \rvert^{-1}} \bigg)
	\bigg( z^{k+2} + O_k \big( z^{k+1} \log \degree R \big) \bigg)
\end{split}
\end{align*}
as $\degree R \longrightarrow \infty$.
\end{corollary}

\begin{proof}
By the binomial theorem we have
\begin{align*}
(\degree N )^k
= \sum_{i=0}^{k} \binom{k}{i} (-1)^i (z - \degree N)^i z^{k-i} ,
\end{align*}
and let us define
\begin{align*}
a(R)
:= (1-q^{-1})
	\prod_{P \mid MR} \bigg( \frac{1 - \lvert P \rvert^{-1}}{1 + \lvert P \rvert^{-1}} \bigg) 
	\prod_{\substack{P \mid M \\ P \nmid R}} \bigg( \frac{1}{1 - \lvert P \rvert^{-1}} \bigg) .
\end{align*}
Then, by Lemma \ref{Proposition Sum of 2^(omega(V) - omega((V,W)) (z - deg V)^k / |V|}, we have
\begin{align*}
\begin{split}
\sum_{\substack{N \in \mathcal{M} \\ \degree N \leq z \\ (N,R) = 1}} \frac{2^{\omega (N) - \omega \big( (N,M) \big)}}{\lvert N \rvert} ( \degree N )^k 
= &a(R) z^{k+2} \sum_{i=0}^{k} \binom{k}{i} \frac{1}{(i+2)(i+1)} (-1)^i  + O_k \big( a(R) z^{k+1} \log \degree R \big)\\
= &\frac{a(R) z^{k+2}}{(k+2) (k+1)} \sum_{i=2}^{k+2} \binom{k+2}{i} (-1)^i  + O_k \big( a(R) z^{k+1} \log \degree R \big)\\
= &\frac{a(R) z^{k+2}}{k+2}  + O_k \big( a(R) z^{k+1} \log \degree R \big) .
\end{split}
\end{align*}
\end{proof}

\begin{lemma} \label{Proposition, extending sums over smooth numbers}
Suppose $\nu$ is a multiplicative function on $\mathcal{A}$ and that there exists a non-negative integer $r$ such that $\nu (P^k ) = O(k^r )$ for all primes $P$ (the implied constant is independent of $P$). Furthermore, suppose there is an $\eta > 0$ such that $\nu (A) \ll_{\eta } \lvert A \rvert^{\eta }$ as $\degree A \overset{q}{\longrightarrow} \infty$. \\

Let $R \in \mathcal{M}$ be a variable, $a,b > 0$ be constants, and $X=X(R) , y= y (R)$ be non-negative, increasing, integer-valued functions such that $X \leq a \log_q \log \degree R$ and $y \geq b \log_q \degree R$ for large enough $\degree R$. \\

Let $c$ and $\epsilon$ be such that $c > \epsilon > \max \big\{ 0 , 1 - \frac{1}{a} \big\}$ and $c > \eta$, and let $\delta > 0$ be small. Finally, let $S \in \mathcal{M}$; $S$ may depend on $R$. We then have that
\begin{align*}
\sum_{\substack{A \in \mathcal{S}_{\mathcal{M}} (X) \\ \degree A \leq y \\ (A,S)=1}} \frac{ \nu (A)}{\lvert A \rvert^c} 
= \prod_{\substack{\degree P \leq X \\ (P,S)=1}} \bigg( 1 + \frac{ \nu (P)}{\lvert P \rvert^{c}} + \frac{ \nu (P^2)}{\lvert P \rvert^{2c} } + \ldots \bigg)
+ O_{q,a,b,c, r, \epsilon, \delta} \Big( (\degree R)^{-b (c - \epsilon ) (1 - \delta )} \Big)
\end{align*}
as $\degree R \longrightarrow \infty$.
\end{lemma}

\begin{proof}
Let $d \geq 2$. By similar means as in Lemma \ref{lemma, sum over deg N < z of 2^(omega(N) - omega (N,M))/N, eval via contours}, we have that
\begin{align*}
\frac{1}{2 \pi i} \int_{d- i \infty}^{d+ i \infty} \hspace{-0.75em}
	\sum_{\substack{A \in \mathcal{S}_{\mathcal{M}} (X) \\ (A,S)=1}} \frac{ \nu (A)}{\lvert A \rvert^{s+c} }
	\frac{q^{(y + \frac{1}{2})s} }{s} \mathrm{d} s
= \sum_{\substack{A \in \mathcal{S}_{\mathcal{M}} (X) \\ \degree A \leq y \\ (A,S)=1}} \frac{ \nu (A)}{\lvert A \rvert^{c} } .
\end{align*}
Now, let $n$ be a positive integer and let us define the following contours in $\mathbb{C}$.
\begin{align*}
l_1 (n) := &\bigg[ d - \frac{2n \pi i}{\log q} , d + \frac{2n \pi i}{\log q} \bigg] ;\\
l_2 (n) := &\bigg[ d + \frac{2n \pi i}{\log q} , -c+ \epsilon + \frac{2n \pi i}{\log q} \bigg] ;\\
l_3 (n) := &\bigg[ -c+ \epsilon + \frac{2n \pi i}{\log q} , -c+ \epsilon - \frac{2n \pi i}{\log q} \bigg] ;\\
l_4 (n) := &\bigg[ -c+ \epsilon - \frac{2n \pi i}{\log q} , d - \frac{2n \pi i}{\log q} \bigg] ;\\
L (n) := &l_1 (n) \cup l_2 (n) \cup l_3 (n) \cup l_4 (n) .
\end{align*}
We can see that
\begin{align*}
&\frac{1}{2 \pi i} \int_{d- i \infty}^{d+ i \infty}
	\sum_{\substack{A \in \mathcal{S}_{\mathcal{M}} (X) \\ (A,S)=1}} \frac{ \nu (A)}{\lvert A \rvert^{s+c} }
	\frac{q^{(y + \frac{1}{2})s} }{s} \mathrm{d} s \\
=& \frac{1}{2 \pi i} \lim_{n \longrightarrow \infty} \bigg(
\int_{L(n)}
	\sum_{\substack{A \in \mathcal{S}_{\mathcal{M}} (X) \\ (A,S)=1}} \frac{ \nu (A)}{\lvert A \rvert^{s+c} }
	\frac{q^{(y + \frac{1}{2})s} }{s} \mathrm{d} s
- \int_{l_2 (n)}
	\sum_{\substack{A \in \mathcal{S}_{\mathcal{M}} (X) \\ (A,S)=1}} \frac{ \nu (A)}{\lvert A \rvert^{s+c} }
	\frac{q^{(y + \frac{1}{2})s} }{s} \mathrm{d} s \\
& \quad \quad \quad - \int_{l_3 (n)}
	\sum_{\substack{A \in \mathcal{S}_{\mathcal{M}} (X) \\ (A,S)=1}} \frac{ \nu (A)}{\lvert A \rvert^{s+c} }
	\frac{q^{(y + \frac{1}{2})s} }{s} \mathrm{d} s
- \int_{l_4 (n)}
	\sum_{\substack{A \in \mathcal{S}_{\mathcal{M}} (X) \\ (A,S)=1}} \frac{ \nu (A)}{\lvert A \rvert^{s+c} }
	\frac{q^{(y + \frac{1}{2})s} }{s} \mathrm{d} s \bigg) .
\end{align*}
For the integral over $L(n)$ there is a simple pole at $s=0$. So, we have
\begin{align*}
\frac{1}{2 \pi i} \int_{L(n)}
	\sum_{\substack{A \in \mathcal{S}_{\mathcal{M}} (X) \\ (A,S)=1}} \frac{ \nu (A)}{\lvert A \rvert^{s+c} }
	\frac{q^{(y + \frac{1}{2})s} }{s} \mathrm{d} s 
= \sum_{\substack{A \in \mathcal{S}_{\mathcal{M}} (X) \\ (A,S)=1}} \frac{ \nu (A)}{\lvert A \rvert^{c}} 
= \prod_{\substack{\degree P \leq X \\ (P,S)=1}} \bigg( 1 + \frac{ \nu (P)}{\lvert P \rvert^{c}} + \frac{ \nu (P^2)}{\lvert P \rvert^{2c} } + \ldots \bigg) .
\end{align*}

We can see that for all $s \in l_2 (n)$ and all $ s \in l_4 (n)$ we have that $\sum_{\substack{A \in \mathcal{S}_{\mathcal{M}} (X) \\ (A,S)=1}} \frac{ \nu (A)}{\lvert A \rvert^{s+c} }$ and $q^{(y + \frac{1}{2})s}$ are uniformly bounded, independently of $n$. Hence, we can see that the integrals over $l_2 (n) , l_4 (n)$ tend to $0$ as $n \longrightarrow \infty$. \\

Now consider the integral over $l_3 (n)$. Suppose $\epsilon < 1$. Then, for all positive integers $n$ and all $s \in l_3 (n)$ we have that
\begin{align*}
&\sum_{\substack{A \in \mathcal{S}_{\mathcal{M}} (X) \\ (A,S)=1}} \frac{ \nu (A)}{\lvert A \rvert^{s+c} }
\ll \prod_{\substack{\degree P \leq X \\ (P,S)=1}} \bigg( 1 + \frac{ \lvert \nu (P) \rvert}{\lvert P \rvert^{\epsilon}} + \frac{ \lvert \nu (P^2) \rvert}{\lvert P \rvert^{2\epsilon} } + \ldots \bigg) \\
&\leq \prod_{\substack{\degree P \leq X \\ (P,S)=1}} \bigg( 1 + O_{r,\epsilon} \Big( \frac{1}{\lvert P \rvert^{\epsilon}} \big) \bigg) 
\leq \exp \bigg( O_{r , \epsilon } \Big( \sum_{\substack{\degree P \leq X \\ (P,S)=1}} \frac{1}{\lvert P \rvert^{\epsilon}} \Big) \bigg) \\
&\leq \exp \bigg( O_{r, \epsilon} \Big( \sum_{i=1}^{X} \frac{q^{i (1- \epsilon )}}{i} \Big) \bigg) 
\leq \exp \bigg( O_{r, \epsilon , a} \Big( \frac{ (\log \degree R )^{a (1-\epsilon )}}{\log_q \log \degree R} \Big) \bigg) 
\ll (\degree R)^{b(c - \epsilon ) \delta}
\end{align*}
as $\degree R \overset{a,b,c,q,r,\epsilon , \delta}{\longrightarrow} \infty$. Now suppose $\epsilon \geq 1$, then we can show that
\begin{align*}
\sum_{\substack{A \in \mathcal{S}_{\mathcal{M}} (X) \\ (A,S)=1}} \frac{ \nu (A)}{\lvert A \rvert^{s+c} }
\ll \exp \bigg( O_{r, \epsilon , a} \Big(  a \log_q \log \degree R \Big) \bigg)
\ll (\degree R)^{b(c - \epsilon ) \delta}
\end{align*}
as $\degree R \overset{a,b,c,q,r,\epsilon , \delta}{\longrightarrow} \infty$. We also have that
\begin{align*}
q^{(y + \frac{1}{2})s}
\ll (\degree R)^{-b(c - \epsilon )} ,
\end{align*}
from which we deduce that
\begin{align*}
\frac{1}{2 \pi i} \int_{l_3 (n)}
	\sum_{\substack{A \in \mathcal{S}_{\mathcal{M}} (X) \\ (A,S)=1}} \frac{ \nu (A)}{\lvert A \rvert^{s+c} }
	\frac{q^{(y + \frac{1}{2})s} }{s} \mathrm{d} s
\ll (\degree R)^{-b (c - \epsilon ) (1 - \delta )} 
\end{align*}
as $\degree R \overset{a,b,c,q,r,\epsilon , \delta}{\longrightarrow} \infty$. 
\end{proof}

We now prove a result that is required to bound the lower order terms in the proof of Theorem \ref{Theorem, fourth moment of Hadamard Product, Intro version}, but first we require two results from \cite{AndradeYiasemides2021_4thPowMeanDLFuncField_Final}:

\begin{theorem} \label{Brun-Titschmarsh theorem, divisor function case in function field, Intro version}
Suppose $\alpha , \beta$ are fixed and satisfy $0 < \alpha < \frac{1}{2}$ and $0 < \beta < \frac{1}{2}$. Let $X \in \mathcal{M}$ and $y$ be a positive integer satisfying $\beta \degree X < y \leq \degree X$. Also, let $A \in \mathcal{A}$ and $G \in \mathcal{M}$ satisfy $(A,G)=1$ and $\degree G < (1-\alpha ) y$. Then, we have that
\begin{align*}
\sum_{\substack{N \in \mathcal{M} \\ \degree (N-X) < y \\ N \equiv A (\modulus G)}} d(N)
\ll_{\alpha , \beta} \frac{q^y \degree X}{\phi (G)} .
\end{align*}
\end{theorem}

\begin{proof}
See Theorem 6.1 in \cite{AndradeYiasemides2021_4thPowMeanDLFuncField_Final}.
\end{proof}

\begin{lemma} \label{Double divisor sum}
Let $F,K \in \mathcal{M}$, $x \geq 0$, and $a \in \mathbb{F}_q^*$. Suppose also that $\frac{1}{2} x < \degree KF \leq \frac{3}{4} x$. Then,
\begin{align*}
\sum_{\substack{N \in \mathcal{M} \\ \degree N = x - \degree KF \\ (N,F) = 1}} d(N) d(KF+ aN)
\ll q^x x^2 \frac{1}{\lvert KF \rvert} \sum_{\substack{H \mid K \\ \degree H \leq \frac{x-\degree KF}{2}}} \frac{d(H)}{\lvert H \rvert} .
\end{align*}
\end{lemma}

\begin{proof}
See Lemma 7.7 in \cite{AndradeYiasemides2021_4thPowMeanDLFuncField_Final}.
\end{proof}

\begin{lemma} \label{Double divisor sum, number 2}
Let $F \in \mathcal{M}$, $K \in \mathcal{A} \backslash \{ 0 \}$, and $x \geq 0$ satisfy $\degree KF < x$. Then,
\begin{align*}
\sum_{\substack{N \in \mathcal{M} \\ \degree N = x \\ (N,F) = 1}} d(N) d(KF+N)
\ll q^x x^2  \sum_{\substack{H \mid K \\ \degree H \leq \frac{x}{2}}} \frac{d(H)}{\lvert H \rvert} .
\end{align*}
\end{lemma}

\begin{proof}
See Lemma 7.8 in \cite{AndradeYiasemides2021_4thPowMeanDLFuncField_Final}.
\end{proof}

\begin{lemma} \label{Lemma, off diagonal z_1 z_2 sum, EH version}
Let $F \in \mathcal{M}$, $A_3 , B_3 \in \mathcal{S}_{\mathcal{M}} (X)$ with $(A_3 B_3 , F)=1$, and $z_1 , z_2$ be non-negative integers. Also, we define
\begin{align*}
\widehat{\degree } (A)
:= \begin{cases}
1 &\text{ if $\degree A = 0$} \\
\degree A &\text{ if $\degree A \geq 1$}.
\end{cases}
\end{align*}
Then, for all $\epsilon > 0$ we have the following:
\begin{align*}
\sum_{\substack{A_1 , A_2 , B_1 , B_2 \in \mathcal{M} \\ \degree A_1 B_1 = z_1 \\ \degree A_2 B_2 = z_2 \\ (A_1 A_2 B_1 B_2 , F)=1 \\ A_1 A_2 A_3 \equiv B_1 B_2 B_3 (\modulus F) \\ A_1 A_2 A_3 \neq B_1 B_2 B_3}} 1
\ll_{\epsilon} \Big( q^{z_1} q^{z_2} \Big)^{1 + \epsilon} \lvert A_3 B_3 \rvert \frac{\widehat{\degree } (A_3 B_3)}{\lvert F \rvert}
\end{align*}
if $z_1 + z_2 + \degree A_3 B_3 \leq \frac{19}{10} \degree F$; and
\begin{align*}
\sum_{\substack{A_1 , A_2 , B_1 , B_2 \in \mathcal{M} \\ \degree A_1 B_1 = z_1 \\ \degree A_2 B_2 = z_2 \\ (A_1 A_2 B_1 B_2 , F)=1 \\ A_1 A_2 A_3 \equiv B_1 B_2 B_3 (\modulus F) \\ A_1 A_2 A_3 \neq B_1 B_2 B_3}} 1
\ll q^{z_1 + z_2}\lvert A_3 B_3 \rvert (z_1 + z_2 + \degree A_3 B_3 )^3 \frac{1}{\phi (F)}
\end{align*}
if $z_1 + z_2 + \degree A_3 B_3 > \frac{19}{10} \degree F$.
\end{lemma}

\begin{proof}
We can split the sum into the cases $\degree A_1 A_2 A_3 > \degree B_1 B_2 B_3$, $\degree A_1 A_2 A_3 < \degree B_1 B_2 B_3$, and $\degree A_1 A_2 A_3 = \degree B_1 B_2 B_3$ with $A_1 A_2 A_3 \neq B_1 B_2 B_3$. \\

When $\degree A_1 A_2 A_3 > \degree B_1 B_2 B_3$, we have that $A_1 A_2 A_3 = KF + B_1 B_2 B_3$ where $K \in \mathcal{M}$ and $\degree KF > \degree B_1 B_2 B_3$. Furthermore,
\begin{align*}
2 \degree KF
= &2 \degree A_1 A_2 A_3
> \degree A_1 A_2 A_3 + \degree B_1 B_2 B_3 \\
= &\degree A_1 B_1 + \degree A_2 B_2 + \degree A_3 B_3 
= z_1 + z_2 + \degree A_3 B_3 ,
\end{align*}
from which we deduce that 
\begin{align*}
a_0 := \frac{z_1 + z_2 + \degree A_3 B_3}{2} < \degree KF \leq z_1 + z_2 + \degree A_3 =: a_1 .
\end{align*}
Also,
\begin{align*}
\degree KF + \degree B_1 B_2
= \degree A_1 A_2 A_3 + \degree B_1 B_2 = z_1 + z_2 + \degree A_3 ,
\end{align*}
from which we deduce that
\begin{align*}
\degree B_1 B_2 = z_1 + z_2 + \degree A_3 - \degree KF. \\
\end{align*}

Similarly, if $\degree A_1 A_2 A_3 < \degree B_1 B_2 B_3$, we can show that
\begin{align*}
b_0 := \frac{z_1 + z_2 + \degree A_3 B_3}{2} < \degree KF \leq z_1 + z_2 + \degree B_3 =: b_1
\end{align*}
and
\begin{align*}
\degree A_1 A_2 = z_1 + z_2 + \degree B_3 - \degree KF. \\
\end{align*}

When $\degree A_1 A_2 A_3 = \degree B_1 B_2 B_3$, we must have that
\begin{align*}
&\degree A_1 A_2 = \frac{z_1 + z_2 + \degree B_3 - \degree A_3 }{2} , \\
&\degree B_1 B_2 = \frac{z_1 + z_2 + \degree A_3 - \degree B_3 }{2} .
\end{align*}
Also, we can write $A_1 A_2 A_3 = KF +B_1 B_2 B_3$, where $\degree KF < \degree B_1 B_2 B_3 = \frac{z_1 + z_2 + \degree A_3 B_3}{2}$ and $K \neq 0$ need not be monic. \\

So, writing $N = B_1 B_2$ when $\degree A_1 A_2 A_3 \geq \degree B_1 B_2 B_3$, and $N = A_1 A_2$ when $\degree A_1 A_2 A_3 < \degree B_1 B_2 B_3$, we have that
\begin{align}
\begin{split} \label{z_1 z_2 off diagonal divisor split, EH chapter}
&\sum_{\substack{A_1 , A_2 , B_1 , B_2 \in \mathcal{M} \\ \degree A_1 B_1 = z_1 \\ \degree A_2 B_2 = z_2 \\ (A_1 A_2 B_1 B_2 , F)=1 \\ A_1 A_2 A_3 \equiv B_1 B_2 B_3 (\modulus F) \\ A_1 A_2 A_3 \neq B_1 B_2 B_3}} 1 \\
\leq &\sum_{\substack{K \in \mathcal{M} \\ a_0 < \degree KF \leq a_1}} \;
	\sum_{\substack{N \in \mathcal{M} \\ \degree N = z_1 + z_2 + \degree A_3 - \degree KF \\ (N,F)=1}} d(N) d\Big( (K F + N B_3) {A_3}^{-1} \Big) \\
& + \sum_{\substack{K \in \mathcal{M} \\ b_0 < \degree KF \leq b_1}} \;
	\sum_{\substack{N \in \mathcal{M} \\ \degree N = z_1 + z_2 + \degree B_3 - \degree KF \\ (N,F)=1}} d(N) d\Big( (K F + N A_3) {B_3}^{-1} \Big) \\
&+ \sum_{\substack{K \in \mathcal{A} \backslash \{ 0 \} \\ \degree KF < a_0 }} \;
	\sum_{\substack{N \in \mathcal{M} \\ \degree N = \frac{z_1 + z_2 + \degree A_3 - \degree B_3}{2} \\ (N,F)=1}} d(N) d\Big( (K F + N B_3) {A_3}^{-1} \Big) .
\end{split}
\end{align}

We must remark that if $A_3 \mid (KF + NB_3 )$ then we define $(K F + N B_3) {A_3}^{-1}$ by $(K F + N B_3) {A_3}^{-1} \cdot A_3 = (K F + N B_3)$. If $A_3 \nmid (KF + NB_3 )$, then we ignore the term with $(K F + N B_3) {A_3}^{-1}$ in the sum; that is, we take the definition $d \Big( (K F + N B_3) {A_3}^{-1} \Big) := 0$. We do the same for $(K F + N A_3) {B_3}^{-1}$. \\

\textbf{Step 1:} Let us consider the case when $z_1 + z_2 + \degree A_3 B_3 \leq \frac{19}{10} \degree F$. By using well known bounds on the divisor function, we have that
\begin{align*}
&\sum_{\substack{K \in \mathcal{M} \\ a_0 < \degree KF \leq a_1}}
	\sum_{\substack{N \in \mathcal{M} \\ \degree N = z_1 + z_2 + \degree A_3 - \degree KF \\ (N,F)=1}} d(N) d\Big( (K F + N B_3) {A_3}^{-1} \Big) \\
\ll_{\epsilon} &\Big( q^{z_1} q^{z_2} \Big)^{\frac{\epsilon}{2}} \sum_{\substack{K \in \mathcal{M} \\ a_0 < \degree KF \leq a_1}}
	\sum_{\substack{N \in \mathcal{M} \\ \degree N = z_1 + z_2 + \degree A_3 - \degree KF \\ (N,F)=1}} 1 \\
\leq &\Big( q^{z_1} q^{z_2} \Big)^{1 + \frac{\epsilon}{2}} \lvert A_3 \rvert \sum_{\substack{K \in \mathcal{M} \\ a_0 < \degree KF \leq a_1}} \frac{1}{\lvert KF \rvert} \\
\ll &\Big( q^{z_1} q^{z_2} \Big)^{1 + \frac{\epsilon}{2}} \lvert A_3 \rvert \frac{z_1 + z_2 + \degree A_3}{\lvert F \rvert} 
\ll_{\epsilon} \Big( q^{z_1} q^{z_2} \Big)^{1 + \epsilon} \lvert A_3 \rvert \frac{\widehat{\degree } A_3}{\lvert F \rvert} .
\end{align*}

Similarly,
\begin{align*}
\sum_{\substack{K \in \mathcal{M} \\ b_0 < \degree KF \leq b_1}} \;
	\sum_{\substack{N \in \mathcal{M} \\ \degree N = z_1 + z_2 + \degree B_3 - \degree KF \\ (N,F)=1}}  \hspace{-1em} d(N) d\Big( (K F + N A_3) {B_3}^{-1} \Big) 
\ll_{\epsilon} \Big( q^{z_1} q^{z_2} \Big)^{1 + \epsilon} \lvert B_3 \rvert \frac{\widehat{\degree } B_3 }{\lvert F \rvert} .
\end{align*}

As for the sum
\begin{align*}
\sum_{\substack{K \in \mathcal{A} \backslash \{ 0 \} \\ \degree KF < a_0 }} \;
	\sum_{\substack{N \in \mathcal{M} \\ \degree N = \frac{z_1 + z_2 + \degree A_3 - \degree B_3}{2} \\ (N,F)=1}} d(N) d\Big( (K F + N B_3) {A_3}^{-1} \Big)  ,
\end{align*}
we note that it does not apply to this case where $z_1 + z_2 + \degree A_3 B_3 \leq \frac{19}{10} \degree F$ because this would imply $\degree KF \geq \degree F \geq \frac{20}{19} a_0 $, which does not overlap with range $\degree KF < a_0 $ in the sum. \\

Hence,
\begin{align*}
\sum_{\substack{A_1 , A_2 , B_1 , B_2 \in \mathcal{M} \\ \degree A_1 B_1 = z_1 \\ \degree A_2 B_2 = z_2 \\ (A_1 A_2 B_1 B_2 , F)=1 \\ A_1 A_2 A_3 \equiv B_1 B_2 B_3 (\modulus F) \\ A_1 A_2 A_3 \neq B_1 B_2 B_3}} 1
\ll_{\epsilon} \Big( q^{z_1} q^{z_2} \Big)^{1 + \epsilon} \lvert A_3 B_3 \rvert \frac{\widehat{\degree } (A_3 B_3) }{\lvert F \rvert} 
\end{align*}
for $z_1 + z_2 + \degree A_3 B_3 \leq \frac{19}{10} \degree F$. \\

\textbf{Step 2:} We now consider the case when $z_1 + z_2 + \degree A_3 B_3 > \frac{19}{10} \degree F$.\\

\textbf{Step 2.1:} We consider the subcase where $a_0 < \degree KF \leq \frac{3}{2} a_0$. This allows us to apply Lemma \ref{Double divisor sum} for the second relation below.
\begin{align*}
&\sum_{\substack{K \in \mathcal{M} \\ a_0 < \degree KF \leq \frac{3}{2} a_0}} \;
	\sum_{\substack{N \in \mathcal{M} \\ \degree N = z_1 + z_2 + \degree A_3 - \degree KF \\ (N,F)=1}} d(N) d\Big( (K F + N B_3) {A_3}^{-1} \Big) \\
\leq &\sum_{\substack{K \in \mathcal{M} \\ a_0 < \degree KF \leq \frac{3}{2} a_0}} \;
	\sum_{\substack{N \in \mathcal{M} \\ \degree N = 2 a_0 - \degree KF \\ (N,F)=1}} d(N) d\Big( K F + N \Big) \\
\ll &q^{z_1} q^{z_2} \lvert A_3 B_3 \rvert (z_1 + z_2 + \degree A_3 B_3 )^2 \frac{1}{\lvert F \rvert}
	\sum_{\substack{K \in \mathcal{M} \\ a_0 < \degree KF \leq \frac{3}{2} a_0}} \frac{1}{\lvert K \rvert}
	\sum_{\substack{H \mid K \\ \degree H \leq \frac{2 a_0 - \degree KF}{2}}} \frac{d(H)}{\lvert H \rvert} \\
\leq &q^{z_1} q^{z_2} \lvert A_3 B_3 \rvert (z_1 + z_2 + \degree A_3 B_3 )^2 \frac{1}{\lvert F \rvert}
	\sum_{\substack{K \in \mathcal{M} \\ \degree KF \leq 2 a_0}} \frac{1}{\lvert K \rvert}
	\sum_{H \mid K} \frac{d(H)}{\lvert H \rvert} \\
\leq &q^{z_1} q^{z_2} \lvert A_3 B_3 \rvert (z_1 + z_2 + \degree A_3 B_3 )^2 \frac{1}{\lvert F \rvert}
	\sum_{\substack{H \in \mathcal{M} \\ \degree H \leq 2 a_0 }} \frac{d(H)}{\lvert H \rvert}
	\sum_{\substack{K \in \mathcal{M} \\ \degree K \leq 2 a_0 \\ H \mid K}} \frac{1}{\lvert K \rvert} \\
\leq &q^{z_1} q^{z_2} \lvert A_3 B_3 \rvert (z_1 + z_2 + \degree A_3 B_3 )^3 \frac{1}{\lvert F \rvert}
	\sum_{\substack{H \in \mathcal{M} \\ \degree H \leq 2 a_0 }} \frac{d(H)}{\lvert H \rvert^2 } \\
\ll &q^{z_1} q^{z_2} \lvert A_3 B_3 \rvert (z_1 + z_2 + \degree A_3 B_3 )^3 \frac{1}{\lvert F \rvert} .
\end{align*}

Similarly,
\begin{align*}
&\sum_{\substack{K \in \mathcal{M} \\ b_0 < \degree KF \leq \frac{3}{2} b_0}} \;
	\sum_{\substack{N \in \mathcal{M} \\ \degree N = z_1 + z_2 + \degree B_3 - \degree KF \\ (N,F)=1}} d(N) d\Big( (K F + N A_3) {B_3}^{-1} \Big) \\
\ll &q^{z_1} q^{z_2} \lvert A_3 B_3 \rvert (z_1 + z_2 + \degree A_3 B_3 )^3 \frac{1}{\lvert F \rvert} .
\end{align*}

\textbf{Step 2.2:} Now we consider the subcase where $\frac{3}{2} a_0 < \degree KF \leq a_1$. We have that
\begin{align*}
&\sum_{\substack{K \in \mathcal{M} \\ \frac{3}{2} a_0 < \degree KF \leq a_1}} \;
	\sum_{\substack{N \in \mathcal{M} \\ \degree N = z_1 + z_2 + \degree A_3 - \degree KF \\ (N,F)=1}} d(N) d\Big( (K F + N B_3) {A_3}^{-1} \Big) \\
\leq &\sum_{\substack{K \in \mathcal{M} \\ \frac{3}{2} a_0 < \degree KF \leq a_1}} \;
	\sum_{\substack{N \in \mathcal{M} \\ \degree N = 2 a_0 - \degree KF \\ (N,F)=1}} d(N) d(K F+ N) \\
\leq &\sum_{\substack{N \in \mathcal{M} \\ \degree N < \frac{a_0 }{2} \\ (N,F)=1}}
	\; \sum_{\substack{K \in \mathcal{M} \\ \degree KF = 2 a_0 - \degree N}} d(N) d(KF+N) \\
\leq &\sum_{\substack{N \in \mathcal{M} \\ \degree N < \frac{a_0 }{2} \\ (N,F)=1}} d(N)
	\sum_{\substack{M \in \mathcal{M} \\ \degree (M - X_{(N)} ) < 2 a_0 - \degree N \\ M \equiv N (\modulus F)}} d(M) 
\end{align*}
where we define $X_{(N)} := T^{2 a_0 - \degree N}$ (The monic polynomial of degree $2 a_0 - \degree N$ with all non-leading coefficients equal to $0$). \\

We can now apply Theorem \ref{Brun-Titschmarsh theorem, divisor function case in function field, Intro version}. One may wish to note that
\begin{align*}
y := 2 a_0 - \degree N
\geq \frac{3}{4}(z_1 + z_2 + \degree A_3 B_3)
\geq \frac{3}{4} \frac{19}{10} \degree F
\end{align*}
and so
\begin{align*}
\degree F \leq \frac{40}{57} y = (1- \alpha) y
\end{align*}
where $0 < \alpha < \frac{1}{2}$, as required. Hence, we have that
\begin{align*}
&\sum_{\substack{K \in \mathcal{M} \\ \frac{3}{2} a_0 < \degree KF \leq a_1}} \;
	\sum_{\substack{N \in \mathcal{M} \\ \degree N = z_1 + z_2 + \degree A_3 - \degree KF \\ (N,F)=1}} d(N) d\Big( (K F + N B_3) {A_3}^{-1} \Big) \\
\leq & q^{z_1} q^{z_2} \lvert A_3 B_3 \rvert (z_1 + z_2 + \degree A_3 B_3 ) \frac{1}{\phi (F)}
	\sum_{\substack{N \in \mathcal{M} \\ \degree N < \frac{a_0 }{2} \\ (N,F)=1}} \frac{d(N)}{\lvert N \rvert} \\
\leq & q^{z_1} q^{z_2} \lvert A_3 B_3 \rvert (z_1 + z_2 + \degree A_3 B_3 )^3 \frac{1}{\phi (F)} . 
\end{align*}

Similarly, if $\frac{3}{2} b_0 < \degree KF \leq b_1$ then
\begin{align*}
&\sum_{\substack{K \in \mathcal{M} \\ \frac{3}{2} b_0 < \degree KF \leq b_1}} \;
	\sum_{\substack{N \in \mathcal{M} \\ \degree N = z_1 + z_2 + \degree B_3 - \degree KF \\ (N,F)=1}} d(N) d\Big( (K F + N A_3) {B_3}^{-1} \Big) \\
\leq & q^{z_1} q^{z_2} \lvert A_3 B_3 \rvert (z_1 + z_2 + \degree A_3 B_3 )^3 \frac{1}{\phi (F)} . 
\end{align*}

\textbf{Step 2.3:} We now look at the sum
\begin{align*}
\sum_{\substack{K \in \mathcal{A} \backslash \{ 0 \} \\ \degree KF < a_0 }} \;
	\sum_{\substack{N \in \mathcal{M} \\ \degree N = \frac{z_1 + z_2 + \degree A_3 - \degree B_3}{2} \\ (N,F)=1}} d(N) d\Big( (K F + N B_3) {A_3}^{-1} \Big) .
\end{align*}

By Lemma \ref{Double divisor sum, number 2} we have that
\begin{align*}
&\sum_{\substack{K \in \mathcal{A} \backslash \{ 0 \} \\ \degree KF < a_0 }} \;
	\sum_{\substack{N \in \mathcal{M} \\ \degree N = \frac{z_1 + z_2 + \degree A_3 - \degree B_3}{2} \\ (N,F)=1}} d(N) d\Big( (K F + N B_3) {A_3}^{-1} \Big) \\
\leq &\sum_{\substack{K \in \mathcal{A} \backslash \{ 0 \} \\ \degree KF < a_0 }} \;
	\sum_{\substack{N \in \mathcal{M} \\ \degree N = a_0 \\ (N,F)=1}} d(N) d(K F + N) \\
\ll &q^{\frac{z_1 + z_2 }{2}}\lvert A_3 B_3 \rvert^{\frac{1}{2}} (z_1 + z_2 + \degree A_3 B_3 )^2
	\sum_{\substack{K \in \mathcal{A} \backslash \{ 0 \} \\ \degree KF < a_0 }}
	\sum_{H \mid K} \frac{d(H)}{\lvert H \rvert} \\
\leq &q^{z_1 + z_2 -1}\lvert A_3 B_3 \rvert (z_1 + z_2 + \degree A_3 B_3 )^2 \frac{1}{\lvert F \rvert}
	\sum_{\substack{K \in \mathcal{A} \backslash \{ 0 \} \\ \degree KF < a_0 }} \frac{1}{\lvert K \rvert}
	\sum_{H \mid K} \frac{d(H)}{\lvert H \rvert} \\
\leq &q^{z_1 + z_2 }\lvert A_3 B_3 \rvert (z_1 + z_2 + \degree A_3 B_3 )^3 \frac{1}{\lvert F \rvert} ,
\end{align*}
where the second-to-last relation uses the fact that $a_0$ is an integer (since $\degree A_1 A_2 A_3 = \degree B_1 B_2 B_3$) and so $\degree KF < a_0$ implies $\degree KF \leq a_0 -1$, and the last relation uses a similar calculation as that in Step 2.1. \\

\textbf{Step 2.4:} We apply steps 2.1, 2.2, and 2.3 to (\ref{z_1 z_2 off diagonal divisor split, EH chapter}) and we see that
\begin{align*}
\sum_{\substack{A_1 , A_2 , B_1 , B_2 \in \mathcal{M} \\ \degree A_1 B_1 = z_1 \\ \degree A_2 B_2 = z_2 \\ (A_1 A_2 B_1 B_2 , F)=1 \\ A_1 A_2 A_3 \equiv B_1 B_2 B_3 (\modulus F) \\ A_1 A_2 A_3 \neq B_1 B_2 B_3}} 1
\ll q^{z_1 + z_2}\lvert A_3 B_3 \rvert (z_1 + z_2 + \degree A_3 B_3 )^3 \frac{1}{\phi (F)}
\end{align*}
for $z_1 + z_2 + \degree A_3 B_3 > \frac{19}{10} \degree F$.
\end{proof}


\section{The Fourth Hadamard Moment} \label{section, fourth twisted moment}

We can now prove Theorem \ref{Theorem, fourth moment of Hadamard Product, Intro version}.

\begin{proof}[Proof of Theorem \ref{Theorem, fourth moment of Hadamard Product, Intro version}]
In this proof, we assume all asymptotic relations are as $X , \degree R \overset{q}{\longrightarrow} \infty$ with
$X \leq \log_q \log \degree R$. Using Lemmas \ref{Lemma P_X expressed as product P_x^**} and \ref{Lemma P_X expressed as sum hat P_x^**}, we have
\begin{align*}
\frac{1}{\phi^* (R)} \sumstar_{\chi \modulus R} \Big\lvert L \Big( \frac{1}{2} , \chi \Big) P_X \Big( \frac{1}{2} , \chi \Big)^{-1} \Big\rvert^4 
\sim &\frac{1}{\phi^* (R)} \sumstar_{\chi \modulus R} \Big\lvert L \Big( \frac{1}{2} , \chi \Big) \Big\rvert^4 \Big\lvert P_X^{**} \Big( \frac{1}{2} , \chi \Big) \Big\rvert^2 \\
= &\frac{1}{\phi^* (R)} \sumstar_{\chi \modulus R} \Big\lvert L \Big( \frac{1}{2} , \chi \Big) \Big\rvert^4 \Big\lvert \widehat{P_X^{**}} \Big( \frac{1}{2} , \chi \Big) + O \Big( (\degree R )^{-\frac{1}{33}} \Big) \Big\rvert^2 .
\end{align*}
By the Cauchy-Schwarz inequality, (\ref{statement, 2nd moment DLF in FF}), and Lemma \ref{lemma, Mertens 3rd Theorem in FF}, it suffices to prove 
\begin{align*}
\frac{1}{\phi^* (R)} \sumstar_{\chi \modulus R} \Big\lvert L \Big( \frac{1}{2} , \chi \Big) \Big\rvert^4 \Big\lvert \widehat{P_X^{**}} \Big( \frac{1}{2} , \chi \Big) \Big\rvert^2 
\sim \frac{1}{12} (\degree R)^4 
	\prod_{\substack{\degree P > X \\ P \mid R}} \bigg( \frac{\big(1 - \lvert P \rvert^{-1} \big)^3}{1 + \lvert P \rvert^{-1}} \bigg)
	\prod_{\degree P \leq X } \big(1 - \lvert P \rvert^{-1} \big)^4 .
\end{align*}

By Lemma \ref{Lemma, short sum for squared L-function}, we have
\begin{align*}
\frac{1}{\phi^* (R)} \sumstar_{\chi \modulus R}
	\Big\lvert L \Big( \frac{1}{2} , \chi \Big) \Big\rvert^4
	\Big\lvert \widehat{P_X^{**}} \Big( \frac{1}{2} , \chi \Big) \Big\rvert^2 
= \frac{1}{\phi^* (R)} \sumstar_{\chi \modulus R}
	\Big( 2 a (\chi ) + 2 b (\chi ) + c (\chi ) \Big)^2
	\Big\lvert \widehat{P_X^{**}} \Big( \frac{1}{2} , \chi \Big) \Big\rvert^2 ,
\end{align*}
where $c (\chi )$ is as in Lemma \ref{Lemma, short sum for squared L-function} and
\begin{align*}
z_R := &\degree R - \log_q 2^{\omega (R)} ; \\
a(\chi ) := &\sum_{\substack{A,B \in \mathcal{M} \\ \degree AB \leq z_R }} \frac{\chi (A) \conj\chi (B)}{\lvert AB \rvert^{\frac{1}{2}}} ; \\
b(\chi ) := &\sum_{\substack{A,B \in \mathcal{M} \\ z_R < \degree AB < \degree R}} \frac{\chi (A) \conj\chi (B)}{\lvert AB \rvert^{\frac{1}{2}}} .
\end{align*}
Note that, by symmetry in $A,B$, the terms $a (\chi )$, $b (\chi )$, and $c (\chi )$ are equal to their conjugates and, therefore, they are real. Hence, by the Cauchy-Schwarz inequality, it suffices to obtain the asymptotic main term of
\begin{align} \label{Sum over prim chi of (deg AB < deg R term)^2 P_x**^2}
\frac{4}{\phi^* (R)} \sumstar_{\chi \modulus R}
	a (\chi )^2
	\Big\lvert \widehat{P_X^{**}} \Big( \frac{1}{2} , \chi \Big) \Big\rvert^2
\end{align}
and show that 
\begin{align*}
\frac{1}{\phi^* (R)} \sumstar_{\chi \modulus R} b(\chi )^2 \Big\lvert \widehat{P_X^{**}} \Big( \frac{1}{2} , \chi \Big) \Big\rvert^2
\quad \text{ and } \quad
\frac{1}{\phi^* (R)} \sumstar_{\chi \modulus R} c(\chi )^2 \Big\lvert \widehat{P_X^{**}} \Big( \frac{1}{2} , \chi \Big) \Big\rvert^2
\end{align*}
are of lower order. The reason we express the sum in terms of $a (\chi)$ and $b (\chi)$ is because the fact that $a (\chi)$ is truncated allows us to bound the lower order terms that it contributes. We cannot do this with $b (\chi)$ but, because $b (\chi)$ is a relatively short sum, we can apply others methods to bound it. \\

\textbf{Step 1; the asymptotic main term of $\frac{4}{\phi^* (R)} \sumstar_{\chi \modulus R} a(\chi )^2 \Big\lvert \widehat{P_X^{**}} \Big( \frac{1}{2} , \chi \Big) \Big\rvert^2$:} \\

By Lemma \ref{Primitive chracter sum, mobius inversion} and Corollary \ref{Number of primitive characters}, we have that
\begin{align}
\begin{split} \label{fourth twisted moment, a(chi) sum, diag and off diag split, statement}
&\frac{1}{\phi^* (R)} \sumstar_{\chi \modulus R} a(\chi )^2 \Big\lvert \widehat{P_X^{**}} \Big( \frac{1}{2} , \chi \Big) \Big\rvert^2 \\
= &\frac{1}{\phi^* (R)} \sumstar_{\chi \modulus R} \sum_{\substack{A_1 , A_2 , B_1 , B_2 \in \mathcal{M} \\ A_3 , B_3 \in \mathcal{S}_{\mathcal{M}} (X) \\ \degree A_1 B_1 , \degree A_2 B_2 \leq z_R \\ \degree A_3 , \degree B_3 \leq \frac{1}{8} \log_q \degree R}}
	\frac{\beta (A_3) \beta (B_3) \chi (A_1 A_2 A_3) \conj\chi (B_1 B_2 B_3)}{\lvert A_1 A_2 A_3 B_1 B_2 B_3 \rvert^{\frac{1}{2}}} \\
= & \sum_{\substack{A_1 , A_2 , B_1 , B_2 \in \mathcal{M} \\ A_3 , B_3 \in \mathcal{S}_{\mathcal{M}} (X) \\ \degree A_1 B_1 , \degree A_2 B_2 \leq z_R \\ \degree A_3 , \degree B_3 \leq \frac{1}{8} \log_q \degree R \\ (A_1 A_2 A_3 B_1 B_2 B_3 , R)=1 \\ A_1 A_2 A_3 = B_1 B_2 B_3 }}
	\frac{\beta (A_3) \beta (B_3) }{\lvert A_1 A_2 A_3 B_1 B_2 B_3 \rvert^{\frac{1}{2}}} \\
&+ \frac{1}{\phi^* (R)} \sum_{EF=R} \mu (E) \phi (F)
	\sum_{\substack{A_1 , A_2 , B_1 , B_2 \in \mathcal{M} \\ A_3 , B_3 \in \mathcal{S}_{\mathcal{M}} (X) \\ \degree A_1 B_1 , \degree A_2 B_2 \leq z_R \\ \degree A_3 , \degree B_3 \leq \frac{1}{8} \log_q \degree R \\ (A_1 A_2 A_3 B_1 B_2 B_3 , R)=1 \\ A_1 A_2 A_3 \equiv B_1 B_2 B_3 (\modulus F) \\  A_1 A_2 A_3 \neq B_1 B_2 B_3 }}
	\frac{\beta (A_3) \beta (B_3) }{\lvert A_1 A_2 A_3 B_1 B_2 B_3 \rvert^{\frac{1}{2}}} .
\end{split}
\end{align}

\textbf{Step 1.1:} We consider the first term on the far right side of (\ref{fourth twisted moment, a(chi) sum, diag and off diag split, statement}): the diagonal terms. By Lemma \ref{A_1 A_2 A_3 = B_1 B_2 B_3 rewrite} we have
\begin{align*}
& \sum_{\substack{A_1 , A_2 , B_1 , B_2 \in \mathcal{M} \\ A_3 , B_3 \in \mathcal{S}_{\mathcal{M}} (X) \\ \degree A_1 B_1 , \degree A_2 B_2 \leq z_R \\ \degree A_3 , \degree B_3 \leq \frac{1}{8} \log_q \degree R \\ (A_1 A_2 A_3 B_1 B_2 B_3 , R)=1 \\ A_1 A_2 A_3 = B_1 B_2 B_3 }}
	\frac{\beta (A_3) \beta (B_3) }{\lvert A_1 A_2 A_3 B_1 B_2 B_3 \rvert^{\frac{1}{2}}} \\
= &\sum_{\substack{G_1 , G_2 , V_{1,2} , V_{2,1} \in \mathcal{M} \\ G_3 , V_{1,3} , V_{2,3} , V_{3,1} , V_{3,2} \in \mathcal{S}_{\mathcal{M}} (X) \\ \degree (G_1)^2 V_{1,2} V_{1,3} V_{2,1} V_{3,1} \leq z_R \\ \degree (G_2)^2 V_{2,1} V_{2,3} V_{1,2} V_{3,2} \leq z_R \\ \degree G_3 V_{3,1} V_{3,2} \leq \frac{1}{8} \log_q \degree R \\ \degree G_3 V_{1,3} V_{2,3} \leq \frac{1}{8} \log_q \degree R \\ (G_i , R) , (V_{j,k} , R)=1 \; \forall i,j,k  \\ (V_{i,j} , V_{k,l}) = 1 \text{ for $(i \neq k \land j \neq l)$}}}  
	\frac{\beta (G_3 V_{3,1} V_{3,2}) \beta (G_3 V_{1,3} V_{2,3}) }{\lvert G_1 G_2 G_3 V_{1,2} V_{1,3} V_{2,1} V_{2,3} V_{3,1} V_{3,2} \rvert} \\
= & \sum_{\substack{G_3 , V_{1,3} , V_{2,3} , V_{3,1} , V_{3,2} \in \mathcal{S}_{\mathcal{M}} (X) \\ \degree G_3 V_{3,1} V_{3,2} \leq \frac{1}{8} \log_q \degree R \\ \degree G_3 V_{1,3} V_{2,3} \leq \frac{1}{8} \log_q \degree R \\ (G_3 V_{1,3} V_{2,3} V_{3,1} V_{3,2} , R) = 1 \\ (V_{1,3} V_{2,3} , V_{3,1} V_{3,2}) = 1 }}
		\frac{\beta (G_3 V_{3,1} V_{3,2}) \beta (G_3 V_{1,3} V_{2,3}) }{\lvert G_3 V_{1,3} V_{2,3} V_{3,1} V_{3,2} \rvert} \\
	& \hspace{5em} \cdot \sum_{\substack{V_{1,2} , V_{2,1} \in \mathcal{M} \\ \degree V_{1,2} V_{2,1} \leq z_R - \degree V_{1,3} V_{3,1} \\ \degree V_{1,2} V_{2,1} \leq z_R - \degree V_{2,3} V_{3,2} \\ (V_{1,2} V_{2,1} , R) = 1 \\ (V_{1,2} , V_{2,3} V_{3,1}) = 1 \\ (V_{2,1} , V_{3,2} V_{1,3}) = 1 \\ (V_{1,2} , V_{2,1}) = 1 }}
		\frac{1}{\lvert V_{1,2} V_{2,1} \rvert}
	\sum_{\substack{ G_1 , G_2 \in \mathcal{M} \\ \degree G_1 \leq \frac{z_R - \degree V_{1,2} V_{2,1} V_{1,3} V_{3,1}}{2} \\ \degree G_2 \leq \frac{z_R - \degree V_{1,2} V_{2,1} V_{2,3} V_{3,2}}{2} \\ (G_1 G_2 , R) = 1}}
		\frac{1}{\lvert G_1 G_2 \rvert} .
\end{align*}
By Lemma \ref{V = V_(1,2) V_(2,1)} we have
\begin{align*}
&\sum_{\substack{V_{1,2} , V_{2,1} \in \mathcal{M} \\ \degree V_{1,2} V_{2,1} \leq z_R - \degree V_{1,3} V_{3,1} \\ \degree V_{1,2} V_{2,1} \leq z_R - \degree V_{2,3} V_{3,2} \\ (V_{1,2} V_{2,1} , R) = 1 \\ (V_{1,2} , V_{2,3} V_{3,1}) = 1 \\ (V_{2,1} , V_{3,2} V_{1,3}) = 1 \\ (V_{1,2} , V_{2,1}) = 1 }}
		\frac{1}{\lvert V_{1,2} V_{2,1} \rvert}
	\sum_{\substack{ G_1 , G_2 \in \mathcal{M} \\ \degree G_1 \leq \frac{z_R - \degree V_{1,2} V_{2,1} V_{1,3} V_{3,1}}{2} \\ \degree G_2 \leq \frac{z_R - \degree V_{1,2} V_{2,1} V_{2,3} V_{3,2}}{2} \\ (G_1 G_2 , R) = 1}}
		\frac{1}{\lvert G_1 G_2 \rvert} \\
= &\sum_{\substack{V \in \mathcal{M} \\ \degree V \leq z_R - \degree V_{1,3} V_{3,1} \\ \degree V \leq z_R - \degree V_{2,3} V_{3,2} \\ \big( V , R (V_{1,3} V_{3,1} , V_{2,3} V_{3,2}) \big) = 1 }}
		\frac{1}{\lvert V \rvert} 
	\sum_{\substack{V_{1,2} , V_{2,1} \in \mathcal{M} \\ V_{1,2} V_{2,1} = V \\ (V_{1,2} , V_{2,1} ) = 1 \\ (V_{1,2} , V_{2,3} V_{3,1}) = 1 \\ (V_{2,1} , V_{3,2} V_{1,3}) = 1 }} 
	\sum_{\substack{ G_1 , G_2 \in \mathcal{M} \\ \degree G_1 \leq \frac{z_R - \degree V V_{1,3} V_{3,1}}{2} \\ \degree G_2 \leq \frac{z_R - \degree V V_{2,3} V_{3,2}}{2} \\ (G_1 G_2 , R) = 1}}
		\frac{1}{\lvert G_1 G_2 \rvert} \\
=&\sum_{\substack{V \in \mathcal{M} \\ \degree V \leq z_R - \degree V_{1,3} V_{3,1} \\ \degree V \leq z_R - \degree V_{2,3} V_{3,2} \\ \big( V , R (V_{1,3} V_{3,1} , V_{2,3} V_{3,2}) \big) = 1 }}
		\frac{2^{\omega(V) - \omega \Big( \big(V , V_{1,3} V_{2,3} V_{3,1} V_{3,2} \big) \Big)}}{\lvert V \rvert} 
	\sum_{\substack{ G_1 , G_2 \in \mathcal{M} \\ \degree G_1 \leq \frac{z_R - \degree V V_{1,3} V_{3,1}}{2} \\ \degree G_2 \leq \frac{z_R - \degree V V_{2,3} V_{3,2}}{2} \\ (G_1 G_2 , R) = 1}}
		\frac{1}{\lvert G_1 G_2 \rvert} .
\end{align*}
So, we have
\begin{align}
\begin{split} \label{diagonal terms as sum over Gs Vs times sum over V times sum over G1 G2}
& \sum_{\substack{A_1 , A_2 , B_1 , B_2 \in \mathcal{M} \\ A_3 , B_3 \in \mathcal{S}_{\mathcal{M}} (X) \\ \degree A_1 B_1 , \degree A_2 B_2 \leq z_R \\ \degree A_3 , \degree B_3 \leq \frac{1}{8} \log_q \degree R \\ (A_1 A_2 A_3 B_1 B_2 B_3 , R)=1 \\ A_1 A_2 A_3 = B_1 B_2 B_3 }}
	\hspace{-1em} \frac{\beta (A_3) \beta (B_3) }{\lvert A_1 A_2 A_3 B_1 B_2 B_3 \rvert^{\frac{1}{2}}} \\
= & \sum_{\substack{G_3 , V_{1,3} , V_{2,3} , V_{3,1} , V_{3,2} \in \mathcal{S}_{\mathcal{M}} (X) \\ \degree G_3 V_{3,1} V_{3,2} \leq \frac{1}{8} \log_q \degree R \\ \degree G_3 V_{1,3} V_{2,3} \leq \frac{1}{8} \log_q \degree R \\ (G_3 V_{1,3} V_{2,3} V_{3,1} V_{3,2} , R) = 1 \\ (V_{1,3} V_{2,3} , V_{3,1} V_{3,2}) = 1 }}
		\frac{\beta (G_3 V_{3,1} V_{3,2}) \beta (G_3 V_{1,3} V_{2,3}) }{\lvert G_3 V_{1,3} V_{2,3} V_{3,1} V_{3,2} \rvert} \\
	& \cdot \sum_{\substack{V \in \mathcal{M} \\ \degree V \leq z_R - \degree V_{1,3} V_{3,1} \\ \degree V \leq z_R - \degree V_{2,3} V_{3,2} \\ \big( V , R (V_{1,3} V_{3,1} , V_{2,3} V_{3,2}) \big) = 1 }}
		\frac{2^{\omega(V) - \omega \Big( \big(V , V_{1,3} V_{2,3} V_{3,1} V_{3,2} \big) \Big)}}{\lvert V \rvert} 
	\sum_{\substack{ G_1 , G_2 \in \mathcal{M} \\ \degree G_1 \leq \frac{z_R - \degree V V_{1,3} V_{3,1}}{2} \\ \degree G_2 \leq \frac{z_R - \degree V V_{2,3} V_{3,2}}{2} \\ (G_1 G_2 , R) = 1}}
		\frac{1}{\lvert G_1 G_2 \rvert} .
\end{split}
\end{align}
Now, by Corollary \ref{Sum over (A,R)=1, deg A <= a deg R of 1/A}, if
\begin{align*}
\frac{z_R - \degree V V_{1,3} V_{3,1}}{2} \geq \log_q 3^{\omega (R)} 
\end{align*}
- that is,
\begin{align*}
\degree V \leq \degree R - \log_q 18^{\omega (R)} - \degree V_{1,3} V_{3,1}
\end{align*}
- then
\begin{align}
\begin{split} \label{G_1 sum for large G_1}
\sum_{\substack{ G_1 \in \mathcal{M} \\ \degree G_1 \leq \frac{z_R - \degree V V_{1,3} V_{3,1}}{2} \\ (G_1 , R) = 1}} \frac{1}{\lvert G_1 \rvert} 
= &\frac{\phi (R)}{2 \lvert R \rvert} (z_R - \degree V V_{1,3} V_{3,1}) + O \Big( \frac{\phi (R)}{\lvert R \rvert} \log \omega (R) \Big) \\
= &\frac{\phi (R)}{2 \lvert R \rvert} \Big( \degree R  - \degree V + O \big( \log \degree R + \omega (R) \big) \Big) .
\end{split}
\end{align}
If
\begin{align*}
\degree V > \degree R - \log_q 18^{\omega (R)} - \degree V_{1,3} V_{3,1},
\end{align*}
then
\begin{align} \label{G_1 sum for small G_1}
\sum_{\substack{ G_1 \in \mathcal{M} \\ \degree G_1 \leq \frac{z_R - \degree V V_{1,3} V_{3,1}}{2} \\ (G_1 , R) = 1}} \frac{1}{\lvert G_1 \rvert} 
\leq \sum_{\substack{ G_1 \in \mathcal{M} \\ \degree G_1 \leq \log_q 3^{\omega (R)}  \\ (G_1 , R) = 1}} \frac{1}{\lvert G_1 \rvert} 
\ll \frac{\phi (R)}{\lvert R \rvert} \omega (R) .
\end{align}
Similar results hold for the sum over $G_2$. \\

So, let us define 
\begin{align*}
m_0 := &\min \big\{ \degree R - \log_q 18^{\omega (R)} - \degree V_{1,3} V_{3,1} \; , \; \degree R - \log_q 18^{\omega (R)} - \degree V_{2,3} V_{3,2} \big\} , \\
m_1 := &\max \big\{ \degree R - \log_q 18^{\omega (R)} - \degree V_{1,3} V_{3,1} \; , \; \degree R - \log_q 18^{\omega (R)} - \degree V_{2,3} V_{3,2} \big\} .
\end{align*}

Then, by (\ref{G_1 sum for large G_1}) and (\ref{G_1 sum for small G_1}), we have
\begin{align}
\begin{split} \label{statement, 2^(omega (V)...) sum with l_1 (R,V,...)}
&\sum_{\substack{V \in \mathcal{M} \\ \degree V \leq z_R - \degree V_{1,3} V_{3,1} \\ \degree V \leq z_R - \degree V_{2,3} V_{3,2} \\ \big( V , R (V_{1,3} V_{3,1} , V_{2,3} V_{3,2}) \big) = 1}} \frac{2^{\omega(V) - \omega \Big( \big(V , V_{1,3} V_{2,3} V_{3,1} V_{3,2} \big) \Big)}}{\lvert V \rvert} 
		\sum_{\substack{ G_1 , G_2 \in \mathcal{M} \\ \degree G_1 \leq \frac{z_R - \degree V V_{1,3} V_{3,1}}{2} \\ \degree G_2 \leq \frac{z_R - \degree V V_{2,3} V_{3,2}}{2} \\ (G_1 G_2 , R) = 1}} \frac{1}{\lvert G_1 G_2 \rvert} \\
& = \frac{\phi (R)^2}{4 \lvert R \rvert^2} \hspace{-3em} \sum_{\substack{V \in \mathcal{M} \\ \degree V \leq m_0 \\ \big( V , R(V_{1,3} V_{3,1} , V_{2,3} V_{3,2}) \big) = 1}} \hspace{-3em} \frac{2^{\omega(V) - \omega \Big( \big(V , V_{1,3} V_{2,3} V_{3,1} V_{3,2} \big) \Big)}}{\lvert V \rvert} 
		\Big( \degree R  - \degree V + O \big( \log \degree R + \omega (R) \big) \Big)^2 \\
	& \hspace{20em} + l_1 (R , V_{1,3} , V_{3,1} , V_{2,3} , V_{3,2}) ,
\end{split}
\end{align}
where
\begin{align}
\begin{split} \label{statement def, EH section, L_1 (R,V,...) def}
l_1 (R , V_{1,3} , V_{3,1} , V_{2,3} , V_{3,2}) 
\ll &\frac{\phi (R)^2 \omega (R) \degree R}{2 \lvert R \rvert^2} \hspace{-3em} \sum_{\substack{V \in \mathcal{M} \\ m_0 < \degree V \leq m_1 \\ \big( V , R(V_{1,3} V_{3,1} , V_{2,3} V_{3,2}) \big) = 1}} \frac{2^{\omega(V) - \omega \Big( \big(V , V_{1,3} V_{2,3} V_{3,1} V_{3,2} \big) \Big)}}{\lvert V \rvert} \\
	&+\frac{\phi (R)^2 \omega (R)^2}{\lvert R \rvert^2} \sum_{\substack{V \in \mathcal{M} \\ m_1 < \degree V \leq \degree R \\ \big( V , R(V_{1,3} V_{3,1} , V_{2,3} V_{3,2}) \big) = 1}} \frac{2^{\omega(V) - \omega \Big( \big(V , V_{1,3} V_{2,3} V_{3,1} V_{3,2} \big) \Big)}}{\lvert V \rvert} .
\end{split}
\end{align}
We now apply Corollary \ref{Corollary Sum of 2^(omega(V) - omega((V,W)) (deg V)^k / |V|} to both terms on the right side of (\ref{statement, 2^(omega (V)...) sum with l_1 (R,V,...)}). For the second term, which is (\ref{statement def, EH section, L_1 (R,V,...) def}), it is just two direct applications. For the first term, we must expand $\Big( \degree R  - \degree V + O \big( \log \degree R + \omega (R) \big) \Big)^2$ and use Corollary \ref{Corollary Sum of 2^(omega(V) - omega((V,W)) (deg V)^k / |V|} on each of the resulting terms. We obtain
\begin{align}
\begin{split} \label{sum over V of 2^(omega - omega) times sum over G1 G2}
&\sum_{\substack{V \in \mathcal{M} \\ \degree V \leq z_R - \degree V_{1,3} V_{3,1} \\ \degree V \leq z_R - \degree V_{2,3} V_{3,2} \\ \big( V , R(V_{1,3} V_{3,1} , V_{2,3} V_{3,2}) \big) = 1}} \frac{2^{\omega(V) - \omega \Big( \big(V , V_{1,3} V_{2,3} V_{3,1} V_{3,2} \big) \Big)}}{\lvert V \rvert} 
		\sum_{\substack{ G_1 , G_2 \in \mathcal{M} \\ \degree G_1 \leq \frac{z_R - \degree V V_{1,3} V_{3,1}}{2} \\ \degree G_2 \leq \frac{z_R - \degree V V_{2,3} V_{3,2}}{2} \\ (G_1 G_2 , R) = 1}} \frac{1}{\lvert G_1 G_2 \rvert} \\
= &\frac{1-q^{-1}}{48} (\degree R)^4 \bigg( 1 + O \Big( \frac{\omega (R)  + \log \degree R}{\degree R} \Big) \bigg) 
	\prod_{P \mid R} \Bigg( \frac{\big( 1-\lvert P \rvert^{-1} \big)^3}{1+\lvert P \rvert^{-1}} \Bigg) \\
	& \hspace{2em}  \cdot  \prod_{P \mid V_{1,3} V_{2,3} V_{3,1} V_{3,2} } \Bigg( \frac{1-\lvert P \rvert^{-1}}{1+\lvert P \rvert^{-1}} \Bigg)
	\prod_{\substack{P \mid V_{1,3} V_{2,3} V_{3,1} V_{3,2} \\ P \nmid (V_{1,3} , V_{2,3}) ,  (V_{3,1} , V_{3,2}) }} \Bigg( \frac{1}{1-\lvert P \rvert^{-1}} \Bigg) \\
=: & l_2 (R, V_{1,3} , V_{2,3} , V_{3,1} , V_{3,2}) .
\end{split}
\end{align}

Before proceeding let us make the following definitions: For $A \in \mathcal{A} \backslash \{ 0 \}$ and $P \in \mathcal{P}$ we define $e_{P} (A)$ to be the largest non-negative integer such that $P^{e_P (A)} \mid A$, and
\begin{align} \label{statement def, EH chapter, 4th Hada prod section, gamma (A) def}
\gamma (A)
:=  \prod_{P \mid A } \Bigg( 1 + e_{P} (A) \frac{1-\lvert P \rvert^{-1}}{1+\lvert P \rvert^{-1}} \Bigg) .
\end{align}
Then, we can see that
\begin{align}
\begin{split} \label{gamma B_3' definition}
&\sum_{\substack{V_{1,3} , V_{2,3} \in \mathcal{S}_{\mathcal{M}} (X) \\ V_{1,3} V_{2,3} = {B_3}'}}
	 \prod_{P \mid V_{1,3} V_{2,3} } \Bigg( \frac{1-\lvert P \rvert^{-1}}{1+\lvert P \rvert^{-1}} \Bigg)
	 \prod_{\substack{P \mid V_{1,3} V_{2,3} \\ P \nmid (V_{1,3} , V_{2,3}) } } \Bigg( \frac{1}{1-\lvert P \rvert^{-1}} \Bigg) \\
= & \prod_{P \mid {B_3}' } \Bigg( \frac{1-\lvert P \rvert^{-1}}{1+\lvert P \rvert^{-1}} \Bigg)
\sum_{\substack{W_1 W_2 = {B_3}' \\ (W_1 , W_2 )=1}} 
	\sum_{\substack{V_{1,3} , V_{2,3} \in \mathcal{S}_{\mathcal{M}} (X) \\ V_{1,3} V_{2,3} = {B_3}' \\ \rad (V_{1,3} , V_{2,3}) = \rad W_1 }}
		 \prod_{P \mid W_2} \Bigg( \frac{1}{1-\lvert P \rvert^{-1}} \Bigg) \\
= & \prod_{P \mid {B_3}' } \Bigg( \frac{1-\lvert P \rvert^{-1}}{1+\lvert P \rvert^{-1}} \Bigg)
\sum_{\substack{W_1 W_2 = {B_3}' \\ (W_1 , W_2 )=1}} 
		 \prod_{P \mid W_2} \Bigg( \frac{1}{1-\lvert P \rvert^{-1}} \Bigg)
		2^{\omega (W_2)}
		\prod_{P \mid W_1} \Big( e_P ({B_3}') - 1 \Big) \\
= & \prod_{P \mid {B_3}' } \Bigg( \frac{1-\lvert P \rvert^{-1}}{1+\lvert P \rvert^{-1}} \Bigg)
	\prod_{P \mid {B_3}'} \bigg( \frac{2}{1 - \lvert P \rvert^{-1}} + \Big( e_P ({B_3}') - 1 \Big) \bigg) \\
= & \prod_{P \mid {B_3}' } \Bigg( 1 + e_{P} ({B_3}') \frac{1-\lvert P \rvert^{-1}}{1+\lvert P \rvert^{-1}} \Bigg)
= \gamma ({B_3}') .
\end{split}
\end{align}
Similarly,
\begin{align}
\begin{split} \label{gamma A_3' definition}
&\sum_{\substack{V_{3,1} , V_{3,2} \in \mathcal{S}_{\mathcal{M}} (X) \\ V_{3,1} V_{3,2} = {A_3}'}}
	 \prod_{P \mid V_{3,1} V_{3,2} } \Bigg( \frac{1-\lvert P \rvert^{-1}}{1+\lvert P \rvert^{-1}} \Bigg)
	 \prod_{\substack{P \mid V_{3,1} V_{3,2} \\ P \nmid (V_{3,1} , V_{3,2}) } } \Bigg( \frac{1}{1-\lvert P \rvert^{-1}} \Bigg) 
= \gamma ({A_3}') .
\end{split}
\end{align}
We now substitute (\ref{sum over V of 2^(omega - omega) times sum over G1 G2}) to (\ref{diagonal terms as sum over Gs Vs times sum over V times sum over G1 G2}) and apply (\ref{gamma B_3' definition}) and (\ref{gamma A_3' definition}) to obtain
\begin{align}
\begin{split} \label{Diagonal sum as sum over G3 A3' B3' of betas and gammas}
& \sum_{\substack{A_1 , A_2 , B_1 , B_2 \in \mathcal{M} \\ A_3 , B_3 \in \mathcal{S}_{\mathcal{M}} (X) \\ \degree A_1 B_1 , \degree A_2 B_2 \leq z_R \\ \degree A_3 , \degree B_3 \leq \frac{1}{8} \log_q \degree R \\ (A_1 A_2 A_3 B_1 B_2 B_3 , R)=1 \\ A_1 A_2 A_3 = B_1 B_2 B_3 }}
	\hspace{-1em} \frac{\beta (A_3) \beta (B_3) }{\lvert A_1 A_2 A_3 B_1 B_2 B_3 \rvert^{\frac{1}{2}}} \\
= & \sum_{\substack{G_3 , V_{1,3} , V_{2,3} , V_{3,1} , V_{3,2} \in \mathcal{S}_{\mathcal{M}} (X) \\ \degree G_3 V_{3,1} V_{3,2} \leq \frac{1}{8} \log_q \degree R \\ \degree G_3 V_{1,3} V_{2,3} \leq \frac{1}{8} \log_q \degree R \\ (G_3 V_{1,3} V_{2,3} V_{3,1} V_{3,2} , R) = 1 \\ (V_{1,3} V_{2,3} , V_{3,1} V_{3,2}) = 1 }}
		\frac{\beta (G_3 V_{3,1} V_{3,2}) \beta (G_3 V_{1,3} V_{2,3}) }{\lvert G_3 V_{1,3} V_{2,3} V_{3,1} V_{3,2} \rvert} l_2 (R, V_{1,3} , V_{2,3} , V_{3,1} , V_{3,2}) \\
= & \sum_{\substack{G_3 , {A_3}' , {B_3}' \in \mathcal{S}_{\mathcal{M}} (X) \\ \degree G_3 {A_3}' \leq \frac{1}{8} \log_q \degree R \\ \degree G_3 {B_3}' \leq \frac{1}{8} \log_q \degree R \\ (G_3 {A_3}' {B_3}' , R) = 1 \\ ({A_3}' , {B_3}') = 1 }}
		\hspace{-1.75em} \frac{\beta (G_3 {A_3}') \beta (G_3 {B_3}') }{\lvert G_3 {A_3}' {B_3}' \rvert} 
	\hspace{-2em} \sum_{\substack{V_{3,1} , V_{3,2} \in \mathcal{S}_{\mathcal{M}} (X) \\ V_{3,1} V_{3,2} = {A_3}'}}
	\sum_{\substack{V_{1,3} , V_{2,3} \in \mathcal{S}_{\mathcal{M}} (X) \\ V_{1,3} V_{2,3} = {B_3}'}}
		\hspace{-1.5em} l_2 (R, V_{1,3} , V_{2,3} , V_{3,1} , V_{3,2}) \\
= & \frac{1-q^{-1}}{48}  \prod_{P \mid R} \Bigg( \frac{\big( 1-\lvert P \rvert^{-1} \big)^3}{1+\lvert P \rvert^{-1}} \Bigg) (\degree R)^4 
	\hspace{-1.5em} \sum_{\substack{G_3 , {A_3}' , {B_3}' \in \mathcal{S}_{\mathcal{M}} (X) \\ \degree G_3 {A_3}' \leq \frac{1}{8} \log_q \degree R \\ \degree G_3 {B_3}' \leq \frac{1}{8} \log_q \degree R \\ (G_3 {A_3}' {B_3}' , R) = 1 \\ ({A_3}' , {B_3}') = 1 }}
		\hspace{-1em} \frac{\beta (G_3 {A_3}') \beta (G_3 {B_3}') }{\lvert G_3 {A_3}' {B_3}' \rvert} \gamma ({A_3}') \gamma ({B_3}') \\
				& \hspace{2em} + l_3 (R) ,
\end{split}
\end{align}
where
\begin{align}
\begin{split} \label{statement def, EH chapter, l_3 (R) def}
&l_3 (R) \\
\ll & \prod_{P \mid R} \Bigg( \frac{\big( 1-\lvert P \rvert^{-1} \big)^3}{1+\lvert P \rvert^{-1}} \Bigg) (\degree R)^{3} \big( \omega (R) + \log \degree R \big) 
		\hspace{-1em} \sum_{\substack{G_3 , {A_3}' , {B_3}' \in \mathcal{S}_{\mathcal{M}} (X) \\ \degree G_3 {A_3}' \leq \frac{1}{8} \log_q \degree R \\ \degree G_3 {B_3}' \leq \frac{1}{8} \log_q \degree R \\ (G_3 {A_3}' {B_3}' , R) = 1 \\ ({A_3}' , {B_3}') = 1 }}
		\hspace{-1.3em} \frac{ \lvert \beta (G_3 {A_3}') \beta (G_3 {B_3}') \rvert}{\lvert G_3 {A_3}' {B_3}' \rvert} \gamma ({A_3}') \gamma ({B_3}') .
\end{split}
\end{align}

Consider the first term on the far right side of (\ref{Diagonal sum as sum over G3 A3' B3' of betas and gammas}). We recall that $\beta (A) = 0$ if $A$ is divisible by $P^3$ for any prime $P$. Hence, defining $\Pi_{\mathcal{P} , X} := \prod_{\degree P \leq X} P$, we may assume that $G_3 = IJ^2$ where $I,J \mid \Pi_{\mathcal{P} , X}$, $(IJ,R)=1$, and $(I,J)=1$. By similar reasoning, we may assume that ${A_3}' = K {A_3}''$ where $K \mid I$, $({A_3}'' , RIJ)=1$; and ${B_3}' = L {B_3}''$ where $L \mid I$, $(L,K)=1$ and $({B_3}'' , RIJ{A_3}'')=1$. Then, by the multiplicativity of $\beta$ and $\gamma$, we have
\begin{align}
\begin{split} \label{Main sum over G3 A3' B3' as sums over I J K L A3'' B3''}
&\sum_{\substack{G_3 , {A_3}' , {B_3}' \in \mathcal{S}_{\mathcal{M}} (X) \\ \degree G_3 {A_3}' \leq \frac{1}{8} \log_q \degree R \\ \degree G_3 {B_3}' \leq \frac{1}{8}\log_q \degree R \\ (G_3 {A_3}' {B_3}' , R) = 1 \\ ({A_3}' , {B_3}') = 1 }} \frac{\beta (G_3 {A_3}') \beta (G_3 {B_3}') }{\lvert G_3 {A_3}' {B_3}' \rvert} \gamma ({A_3}') \gamma ({B_3}') \\
= &\sum_{\substack{ I \mid \Pi_{\mathcal{P} , X} \\ \degree I \leq \frac{1}{8} \log_q \degree R \\ (I,R)=1}} \frac{ \beta (I)^2}{\lvert I \rvert}
	\hspace{-1em} \sum_{\substack{ J \mid \Pi_{\mathcal{P} , X} \\ \degree J \leq \frac{1}{16} \log_q \degree R - \frac{\degree I}{2} \\ (J,RI)=1}} \frac{ \beta (J^2)^2}{\lvert J \rvert^2}
	\sum_{K \mid I} \frac{\beta (K^2 ) \gamma (K)}{\beta (K) \lvert K \rvert}
	\sum_{\substack{L \mid I \\ (L,K)=1}} \frac{\beta (L^2 ) \gamma (L)}{\beta (L) \lvert L \rvert} \\
	& \cdot \sum_{\substack{{A_3}'' \mid (\Pi_{\mathcal{P} , X})^2 \\ \degree {A_3}'' \leq \frac{1}{8} \log_q \degree R - \degree IJ^2 K \\ ({A_3}'' , RIJ)=1}} \frac{\beta ({A_3}'') \gamma ({A_3}'')}{\lvert {A_3}'' \rvert}
	\sum_{\substack{{B_3}'' \mid (\Pi_{\mathcal{P} , X})^2 \\ \degree {B_3}'' \leq \frac{1}{8} \log_q \degree R - \degree IJ^2 L \\ ({B_3}'' , RIJ{A_3}'')=1}} \frac{\beta ({B_3}'') \gamma ({B_3}'')}{\lvert {B_3}'' \rvert} .
\end{split}
\end{align}

Consider the case where $\degree I > \frac{1}{64} \log_q \degree R$ or $\degree J > \frac{1}{64} \log_q \degree R$. Without loss of generality, suppose the former. Then, all the sums above, except that over $I$, can be bounded by $O \big(  (\log_q \log \deg R)^c \big)$ for some constant $c > 0$, while the sum over $I$ can be bounded by $O \big(  (\degree R)^{-\frac{1}{66}} \big)$ (this is obtained in the same way we have done several times before, such as in (\ref{2k-th moment of Euler product, Dirichlet series truncation bound})). So, with these restrictions, we have that the above is $O \big( (\degree R)^{-\frac{1}{67}} \big)$. \\

Now consider the case where $\degree I \leq \frac{1}{64} \log_q \degree R$ and $\degree J \leq \frac{1}{64} \log_q \degree R$. Then, 
\begin{align*}
\frac{1}{8} \log_q \degree R - \degree IJ^2 K \geq \frac{1}{16} \log_q \degree R
\end{align*}
and
\begin{align*}
\frac{1}{8} \log_q \degree R - \degree IJ^2 L \geq \frac{1}{16} \log_q \degree R .
\end{align*}
In particular, we can apply Lemma \ref{Proposition, extending sums over smooth numbers} to the last two summations of (\ref{Main sum over G3 A3' B3' as sums over I J K L A3'' B3''}):
\begin{align}
\begin{split} \label{Sums over A3'' and B3'' of betas and gammas as product}
&\sum_{\substack{{A_3}'' \mid (\Pi_{\mathcal{P} , X})^2 \\ \degree {A_3}'' \leq \frac{1}{8} \log_q \degree R - \degree IJ^2 K \\ ({A_3}'' , RIJ)=1}} \frac{\beta ({A_3}'') \gamma ({A_3}'')}{\lvert {A_3}'' \rvert}
	\sum_{\substack{{B_3}'' \mid (\Pi_{\mathcal{P} , X})^2 \\ \degree {B_3}'' \leq \frac{1}{8} \log_q \degree R - \degree IJ^2 L \\ ({B_3}'' , RIJ{A_3}'')=1}} \frac{\beta ({B_3}'') \gamma ({B_3}'')}{\lvert {B_3}'' \rvert} \\
=& \prod_{\substack{\degree P \leq X \\ (P,R)=1}} \bigg( 1 + \frac{\beta (P) \gamma (P)}{\lvert P \rvert} + \frac{\beta (P^2 ) \gamma (P^2 )}{\lvert P^2 \rvert} \bigg)
	\prod_{P \mid IJ} \bigg( 1 + \frac{\beta (P) \gamma (P)}{\lvert P \rvert} + \frac{\beta (P^2 ) \gamma (P^2 )}{\lvert P^2 \rvert} \bigg)^{-1} \\
	& \cdot \sum_{\substack{{A_3}'' \mid (\Pi_{\mathcal{P} , X})^2 \\ \degree {A_3}'' \leq \frac{1}{8} \log_q \degree R - \degree IJ^2 K \\ ({A_3}'' , RIJ)=1}} \frac{\beta ({A_3}'') \gamma ({A_3}'')}{\lvert {A_3}'' \rvert}
	\prod_{P \mid {A_3}''} \bigg( 1 + \frac{\beta (P) \gamma (P)}{\lvert P \rvert} + \frac{\beta (P^2 ) \gamma (P^2 )}{\lvert P^2 \rvert} \bigg)^{-1} \\
	&+ O \Big( (\degree R)^{- \frac{1}{17}} \Big)  \\
= &\prod_{\substack{\degree P \leq X \\ (P,R)=1}} \bigg( 1 + \frac{\beta (P) \gamma (P)}{\lvert P \rvert} + \frac{\beta (P^2 ) \gamma (P^2 )}{\lvert P^2 \rvert} \bigg)
	\prod_{P \mid IJ} \bigg( 1 + \frac{\beta (P) \gamma (P)}{\lvert P \rvert} + \frac{\beta (P^2 ) \gamma (P^2 )}{\lvert P^2 \rvert} \bigg)^{-1} \\
	& \cdot \prod_{\substack{ \degree P \leq X \\ (P,R)=1}} \Bigg( 1 + \bigg( \frac{\beta (P) \gamma (P)}{\lvert P \rvert} + \frac{\beta (P^2 ) \gamma (P^2 )}{\lvert P^2 \rvert} \bigg) \bigg( 1 + \frac{\beta (P) \gamma (P)}{\lvert P \rvert} + \frac{\beta (P^2 ) \gamma (P^2 )}{\lvert P^2 \rvert} \bigg)^{-1} \Bigg) \\
	& \cdot \prod_{P \mid IJ} \Bigg( 1 + \bigg( \frac{\beta (P) \gamma (P)}{\lvert P \rvert} + \frac{\beta (P^2 ) \gamma (P^2 )}{\lvert P^2 \rvert} \bigg) \bigg( 1 + \frac{\beta (P) \gamma (P)}{\lvert P \rvert} + \frac{\beta (P^2 ) \gamma (P^2 )}{\lvert P^2 \rvert} \bigg)^{-1} \Bigg)^{-1} \\
	&+O \Big( (\degree R)^{- \frac{1}{17}} \Big)  \\
= &\prod_{\substack{\degree P \leq X \\ (P,R)=1}} \bigg( 1 + \frac{2 \beta (P) \gamma (P)}{\lvert P \rvert} + \frac{2 \beta (P^2 ) \gamma (P^2 )}{\lvert P^2 \rvert} \bigg) \\
	& \cdot \prod_{P \mid IJ} \bigg( 1 + \frac{2 \beta (P) \gamma (P)}{\lvert P \rvert} + \frac{2 \beta (P^2 ) \gamma (P^2 )}{\lvert P^2 \rvert} \bigg)^{-1} \\
	& + O \Big( (\degree R)^{- \frac{1}{17}} \Big) . \\
\end{split}
\end{align}

Consider now the two middle summations on the right side of (\ref{Main sum over G3 A3' B3' as sums over I J K L A3'' B3''}). We have
\begin{align}
\begin{split} \label{Sums over K and L of betas and gammas as product}
& \sum_{K \mid I} \frac{\beta (K^2 ) \gamma (K)}{\beta (K) \lvert K \rvert}
\sum_{\substack{L \mid I \\ (L,K)=1}} \frac{\beta (L^2 ) \gamma (L)}{\beta (L) \lvert L \rvert} \\
	& = \prod_{P \mid I} \bigg( 1 + \frac{\beta (P^2 ) \gamma (P)}{\beta (P) \lvert P \rvert} \bigg)
	\sum_{K \mid I} \frac{\beta (K^2 ) \gamma (K)}{\beta (K) \lvert K \rvert}
	\prod_{P \mid K} \bigg( 1 + \frac{\beta (P^2 ) \gamma (P)}{\beta (P) \lvert P \rvert} \bigg)^{-1} \\
	& = \prod_{P \mid I} \bigg( 1 + \frac{\beta (P^2 ) \gamma (P)}{\beta (P) \lvert P \rvert} \bigg)
	\prod_{P \mid I} \Bigg( 1 + \frac{\beta (P^2 ) \gamma (P)}{\beta (P) \lvert P \rvert}
	\bigg( 1 + \frac{\beta (P^2 ) \gamma (P)}{\beta (P) \lvert P \rvert} \bigg)^{-1} \Bigg) \\
	& = \prod_{P \mid I} \bigg( 1 + \frac{2 \beta (P^2 ) \gamma (P)}{\beta (P) \lvert P \rvert} \bigg) .
\end{split}
\end{align}

Applying (\ref{Sums over A3'' and B3'' of betas and gammas as product}) and (\ref{Sums over K and L of betas and gammas as product}) to (\ref{Main sum over G3 A3' B3' as sums over I J K L A3'' B3''}), we obtain
\begin{align*}
&\sum_{\substack{G_3 , {A_3}' , {B_3}' \in \mathcal{S}_{\mathcal{M}} (X) \\ \degree G_3 {A_3}' \leq \frac{1}{8} \log_q \degree R \\ \degree G_3 {B_3}' \leq \frac{1}{8}\log_q \degree R \\ (G_3 {A_3}' {B_3}' , R) = 1 \\ ({A_3}' , {B_3}') = 1 }} \frac{\beta (G_3 {A_3}') \beta (G_3 {B_3}') }{\lvert G_3 {A_3}' {B_3}' \rvert} \gamma ({A_3}') \gamma ({B_3}') \\
= &\prod_{\substack{\degree P \leq X \\ (P,R)=1}} \bigg( 1 + \frac{2 \beta (P) \gamma (P)}{\lvert P \rvert} + \frac{2 \beta (P^2 ) \gamma (P^2 )}{\lvert P^2 \rvert} \bigg) 
	\sum_{\substack{ I \mid \Pi_{\mathcal{P} , X} \\ \degree I \leq \frac{1}{64} \log_q \degree R \\ (I,R)=1}} \frac{ \beta (I)^2}{\lvert I \rvert}
	\sum_{\substack{ J \mid \Pi_{\mathcal{P} , X} \\ \degree J \leq \frac{1}{64} \log_q \degree R \\ (J,RI)=1}} \frac{ \beta (J^2)^2}{\lvert J \rvert^2}\\
	&\hspace{5em} \cdot \prod_{P \mid I} \bigg( 1 + \frac{2 \beta (P^2 ) \gamma (P)}{\beta (P) \lvert P \rvert} \bigg)
	\prod_{P \mid IJ} \bigg( 1 + \frac{2 \beta (P) \gamma (P)}{\lvert P \rvert} + \frac{2 \beta (P^2 ) \gamma (P^2 )}{\lvert P^2 \rvert} \bigg)^{-1} \\
&+ O \Big( (\degree R)^{-\frac{1}{67}} \Big) \\
= &\prod_{\substack{\degree P \leq X \\ (P,R)=1}} \bigg( 1 + \frac{2 \beta (P) \gamma (P)}{\lvert P \rvert} + \frac{2 \beta (P^2 ) \gamma (P^2 )}{\lvert P^2 \rvert} + \frac{\beta (P^2 )^2}{\lvert P \rvert^2 } \bigg) \\
	&\hspace{1em} \cdot \sum_{\substack{ I \mid \Pi_{\mathcal{P} , X} \\ \degree I \leq \frac{1}{64} \log_q \degree R \\ (I,R)=1}} \hspace{-2em} \frac{ \beta (I)^2}{\lvert I \rvert} 
	\prod_{P \mid I} \Bigg( \bigg( 1 + \frac{2 \beta (P^2 ) \gamma (P)}{\beta (P) \lvert P \rvert} \bigg)  \bigg( 1 + \frac{2 \beta (P) \gamma (P)}{\lvert P \rvert} + \frac{2 \beta (P^2 ) \gamma (P^2 )}{\lvert P^2 \rvert} + \frac{\beta (P^2 )^2}{\lvert P \rvert^2 } \bigg)^{-1} \Bigg) \\
& + O \Big( (\degree R)^{-\frac{1}{67}} \Big) \\
= &\prod_{\substack{\degree P \leq X \\ (P,R)=1}} \hspace{-0.5em} \Bigg( 1 + \frac{2 \beta (P) \gamma (P)}{\lvert P \rvert} + \frac{2 \beta (P^2 ) \gamma (P^2 )}{\lvert P^2 \rvert} + \frac{\beta (P^2 )^2}{\lvert P \rvert^2 } + \frac{ \beta (P)^2}{\lvert P \rvert}\bigg( 1 + \frac{2 \beta (P^2 ) \gamma (P)}{\beta (P) \lvert P \rvert} \bigg) \Bigg) \\
&+ O \Big( (\degree R)^{-\frac{1}{67}} \Big) .
\end{align*}

Now, recalling the definitions of $\beta, \gamma$ (equations (\ref{statement def, EH chapter, prelim results 4th Hada prod section, beta (A) def}) and (\ref{statement def, EH chapter, 4th Hada prod section, gamma (A) def}), respectively) we see that the product above is equal to
\begin{align*}
& \prod_{\substack{\degree P \leq X \\ P \nmid R}} \bigg( \frac{\big(1 - \lvert P \rvert^{-1} \big)^3}{1 + \lvert P \rvert^{-1}} \bigg)
 \prod_{\substack{\frac{X}{2} < \degree P \leq X \\ P \nmid R}} \bigg( 1 + O \big( \lvert P \rvert^{-2} \big) \bigg) \\
\sim & \prod_{P \mid R} \bigg( \frac{\big(1 - \lvert P \rvert^{-1} \big)^3}{1 + \lvert P \rvert^{-1}} \bigg)^{-1}
 \prod_{\substack{\degree P > X \\ P \mid R}} \bigg( \frac{\big(1 - \lvert P \rvert^{-1} \big)^3}{1 + \lvert P \rvert^{-1}} \bigg)
 \prod_{\degree P \leq X } \bigg( \frac{\big(1 - \lvert P \rvert^{-1} \big)^3}{1 + \lvert P \rvert^{-1}} \bigg) \\
= & \prod_{P \mid R} \bigg( \frac{\big(1 - \lvert P \rvert^{-1} \big)^3}{1 + \lvert P \rvert^{-1}} \bigg)^{-1}
 \prod_{\substack{\degree P > X \\ P \mid R}} \bigg( \frac{\big(1 - \lvert P \rvert^{-1} \big)^3}{1 + \lvert P \rvert^{-1}} \bigg) 
	\prod_{\degree P \leq X } \big(1 - \lvert P \rvert^{-1} \big)^4
\prod_{\degree P \leq X } \big(1 - \lvert P \rvert^{-2} \big)^{-1} \\
\sim &\big( 1 - q^{-1} \big)^{-1}
 \prod_{P \mid R} \bigg( \frac{\big(1 - \lvert P \rvert^{-1} \big)^3}{1 + \lvert P \rvert^{-1}} \bigg)^{-1}
 \prod_{\substack{\degree P > X \\ P \mid R}} \bigg( \frac{\big(1 - \lvert P \rvert^{-1} \big)^3}{1 + \lvert P \rvert^{-1}} \bigg)
\Big( \frac{1}{ e^{\gamma} X } \Big)^4 ,
\end{align*}
where we have used Lemma \ref{lemma, Mertens 3rd Theorem in FF} for the last equality. Recall that the above is to be applied to the first term on the far right side of (\ref{Diagonal sum as sum over G3 A3' B3' of betas and gammas}). We now consider $l_3 (R)$: the second term on the far right side of (\ref{Diagonal sum as sum over G3 A3' B3' of betas and gammas}). By means similar to those described in the paragraph after (\ref{Main sum over G3 A3' B3' as sums over I J K L A3'' B3''}), we can show that there is some constant $c>0$ such that
\begin{align*}
\sum_{\substack{G_3 , {A_3}' , {B_3}' \in \mathcal{S}_{\mathcal{M}} (X) \\ \degree G_3 {A_3}' \leq \frac{1}{8} \log_q \degree R \\ \degree G_3 {B_3}' \leq \frac{1}{8} \log_q \degree R \\ (G_3 {A_3}' {B_3}' , R) = 1 \\ ({A_3}' , {B_3}') = 1 }} \frac{\vert \beta (G_3 {A_3}') \beta (G_3 {B_3}') \rvert }{\lvert G_3 {A_3}' {B_3}' \rvert} \gamma ({A_3}') \gamma ({B_3}') 
\ll X^c
\ll \big( \log_q \log \degree R \big)^c .
\end{align*}
We apply this to (\ref{statement def, EH chapter, l_3 (R) def}) to obtain a bound for $l_3 (R)$. \\

Hence, considering all of the above, (\ref{Diagonal sum as sum over G3 A3' B3' of betas and gammas}) becomes
\begin{align}
\begin{split} \label{statement, EH chapter, 4th Hada mom section, Step 1.2 final main term}
\sum_{\substack{A_1 , A_2 , B_1 , B_2 \in \mathcal{M} \\ A_3 , B_3 \in \mathcal{S}_{\mathcal{M}} (X) \\ \degree A_1 B_1 , \degree A_2 B_2 \leq z_R \\ \degree A_3 , \degree B_3 \leq \frac{1}{8} \log_q \degree R \\ (A_1 A_2 A_3 B_1 B_2 B_3 , R)=1 \\ A_1 A_2 A_3 = B_1 B_2 B_3 }} \hspace{-1.5em} \frac{\beta (A_3) \beta (B_3) }{\lvert A_1 A_2 A_3 B_1 B_2 B_3 \rvert^{\frac{1}{2}}} 
\sim \frac{1}{48} \Big( \frac{\degree R}{ e^{\gamma} X } \Big)^4
	 \prod_{\substack{\degree P > X \\ P \mid R}} \bigg( \frac{\big(1 - \lvert P \rvert^{-1} \big)^3}{1 + \lvert P \rvert^{-1}} \bigg)
\end{split}
\end{align}

\textbf{Step 1.2:} We consider the second term on the far right side of (\ref{fourth twisted moment, a(chi) sum, diag and off diag split, statement}): the off-diagonal terms. We have
\begin{align*}
&\sum_{EF=R} \mu (E) \phi (F)
	\sum_{\substack{A_1 , A_2 , B_1 , B_2 \in \mathcal{M} \\ A_3 , B_3 \in \mathcal{S}_{\mathcal{M}} (X) \\ \degree A_1 B_1 , \degree A_2 B_2 \leq z_R \\ \degree A_3 , \degree B_3 \leq \frac{1}{8} \log_q \degree R \\ (A_1 A_2 A_3 B_1 B_2 B_3 , R)=1 \\ A_1 A_2 A_3 \equiv B_1 B_2 B_3 (\modulus F) \\  A_1 A_2 A_3 \neq B_1 B_2 B_3 }}
	\frac{\beta (A_3) \beta (B_3) }{\lvert A_1 A_2 A_3 B_1 B_2 B_3 \rvert^{\frac{1}{2}}} \\
&\leq \sum_{\substack{A_3 , B_3 \in \mathcal{S}_{\mathcal{M}} (X) \\ \degree A_3 , \degree B_3 \leq \frac{1}{8} \log_q \degree R \\ (A_3 B_3 , R)=1 }} \hspace{-2em} \frac{\lvert \beta (A_3) \beta (B_3) \rvert}{\lvert A_3 B_3 \rvert^{\frac{1}{2}}}
	\sum_{EF=R} \lvert \mu (E) \rvert \phi (F)
	\sum_{z_1 , z_2 = 0}^{z_R} q^{-\frac{z_1 + z_2}{2}}
	\hspace{-1.5em} \sum_{\substack{A_1 , A_2 , B_1 , B_2 \in \mathcal{M} \\ \degree A_1 B_1 = z_1 \\ \degree A_2 B_2 = z_2 \\  (A_1 A_2 B_1 B_2 , R)=1 \\ A_1 A_2 A_3 \equiv B_1 B_2 B_3 (\modulus F) \\  A_1 A_2 A_3 \neq B_1 B_2 B_3 }} 1 .
\end{align*}
By Lemma \ref{Lemma, off diagonal z_1 z_2 sum, EH version} we have, for $\epsilon = \frac{1}{40}$,
\begin{align*}
&\sum_{z_1 , z_2 = 0}^{z_R} q^{-\frac{z_1 + z_2}{2}}
	\sum_{\substack{A_1 , A_2 , B_1 , B_2 \in \mathcal{M} \\ \degree A_1 B_1 = z_1 \\ \degree A_2 B_2 = z_2 \\  (A_1 A_2 B_1 B_2 , R)=1 \\ A_1 A_2 A_3 \equiv B_1 B_2 B_3 (\modulus F) \\  A_1 A_2 A_3 \neq B_1 B_2 B_3 }} 1 \\
\ll &\frac{\lvert A_3 B_3 \rvert^{1+ \frac{\epsilon}{2}}}{\lvert F \rvert} \hspace{-2em} \sum_{\substack{z_1 , z_2 = 0\\ z_1 + z_2 + \degree A_3 B_3 \leq \frac{19}{10} \degree F}}^{z_R} \hspace{-1em} q^{(z_1 + z_2 ) \big( \frac{1}{2} + \frac{\epsilon}{2} \big)} 
	+ \frac{\lvert A_3 B_3 \rvert}{\phi (F)} \hspace{-2em} \sum_{\substack{z_1 , z_2 = 0 \\ z_1 + z_2 + \degree A_3 B_3 > \frac{19}{10} \degree F}}^{z_R} \hspace{-1em} q^{\frac{z_1 + z_2 }{2}} (z_1 + z_2 + \degree A_3 B_3 )^3 \\
\ll &\frac{\lvert A_3 B_3 \rvert^{1 + \epsilon}}{\lvert F \rvert^{\frac{1}{20} - \epsilon}}
	+ \frac{\lvert A_3 B_3 \rvert}{\phi (F)} q^{z_R} (\degree R)^3 .
\end{align*}
We also have
\begin{align*}
&\sum_{EF=R} \lvert \mu (E) \rvert \phi (F) \bigg( \frac{\lvert A_3 B_3 \rvert^{1 + \epsilon}}{\lvert F \rvert^{\frac{1}{20} - \epsilon}} + \frac{\lvert A_3 B_3 \rvert}{\phi (F)} q^{z_R} (\degree R)^3 \bigg) \\
= &\lvert A_3 B_3 \rvert^{1+ \epsilon} \sum_{EF=R} \lvert \mu (E) \rvert \frac{\phi (F)}{{\lvert F \rvert^{\frac{1}{20} - \epsilon}}}
	+ \lvert A_3 B_3 \rvert q^{z_R} (\degree R)^3 \sum_{EF=R} \lvert \mu (E) \rvert \\
\ll &\lvert A_3 B_3 \rvert^{1 + \epsilon} \lvert R \rvert
	+ \lvert A_3 B_3 R \rvert (\degree R)^3 ,
\end{align*}
where the last relation uses
\begin{align*}
&\sum_{EF=R} \lvert \mu (E) \rvert \frac{\phi (F)}{\lvert F \rvert^{\frac{1}{20} - \epsilon}}
\leq \sum_{EF=R} \lvert \mu (E) \rvert \phi (F) 
= \phi (R) \sum_{EF=R} \lvert \mu (E) \rvert  \prod_{\substack{P \mid E \\ P^2 \mid R}} \bigg( \frac{1}{\lvert P \rvert} \bigg)  \prod_{\substack{P \mid E \\ P^2 \nmid R}} \bigg( \frac{1}{\lvert P \rvert -1} \bigg) \\
\leq &\phi (R) \sum_{EF=R} \lvert \mu (E) \rvert \prod_{P \mid E} \bigg( \frac{1}{\lvert P \rvert -1} \bigg) 
= \phi (R) \prod_{P \mid R} \bigg( 1 + \frac{1}{\lvert P \rvert -1} \bigg)
= \phi (R) \frac{\lvert R \rvert}{\phi (R)}
= \lvert R \rvert .
\end{align*}

Finally, using the fact that
\begin{align*}
&\sum_{\substack{A_3 , B_3 \in \mathcal{S}_{\mathcal{M}} (X) \\ \degree A_3 , \degree B_3 \leq \frac{1}{8} \log_q \degree R \\ (A_3 B_3 , R)=1 }} \lvert \beta (A_3) \beta (B_3) \rvert \lvert A_3 B_3 \rvert^{\frac{1}{2} + \epsilon}
\leq \bigg( \sum_{\substack{A \in \mathcal{M} \\ \degree A \leq \frac{1}{8} \log_q \degree R }} \lvert \beta (A) \rvert \lvert A \rvert^{\frac{1}{2} + \epsilon} \bigg)^2 \\
\leq &\bigg( \hspace{-1em} \sum_{\substack{A \in \mathcal{M} \\ \degree A \leq \frac{1}{8} \log_q \degree R }} \hspace{-1em} 2^{\omega (A)} \lvert A \rvert^{\frac{1}{2} + \epsilon} \bigg)^2 
\leq \bigg( \hspace{-1em} \sum_{\substack{A \in \mathcal{M} \\ \degree A \leq \frac{1}{8} \log_q \degree R }} \hspace{-1em} d (A) \lvert A \rvert^{\frac{1}{2} + \epsilon} \bigg)^2 
\leq \bigg( \hspace{-1em} \sum_{\substack{A \in \mathcal{M} \\ \degree A \leq \frac{1}{8} \log_q \degree R }} \hspace{-1em} \lvert A \rvert^{\frac{1}{2} + \epsilon} \bigg)^4
\leq ( \degree R )^{\frac{7}{8}} ,
\end{align*}
we see that
\begin{align*}
\frac{1}{\phi^* (R)} \sum_{EF=R} \mu (E) \phi (F)
	\hspace{-1.75em} \sum_{\substack{A_1 , A_2 , B_1 , B_2 \in \mathcal{M} \\ A_3 , B_3 \in \mathcal{S}_{\mathcal{M}} (X) \\ \degree A_1 B_1 , \degree A_2 B_2 \leq z_R \\ \degree A_3 , \degree B_3 \leq \frac{1}{8} \log_q \degree R \\ (A_1 A_2 A_3 B_1 B_2 B_3 , R)=1 \\ A_1 A_2 A_3 \equiv B_1 B_2 B_3 (\modulus F) \\  A_1 A_2 A_3 \neq B_1 B_2 B_3 }}
	\frac{ \beta (A_3) \beta (B_3) }{\lvert A_1 A_2 A_3 B_1 B_2 B_3 \rvert^{\frac{1}{2}}}
\ll \frac{\lvert R \rvert}{\phi^* (R)} (\degree R )^{3 + \frac{7}{8}} .
\end{align*}
This is indeed of lower order than (\ref{statement, EH chapter, 4th Hada mom section, Step 1.2 final main term}); Section 4 of \cite{AndradeYiasemides2021_4thPowMeanDLFuncField_Final} provides the necessary results to confirm this. \\

\textbf{Step 2; the asymptotic main term of $\frac{1}{\phi^* (R)} \sumstar_{\chi \modulus R} b(\chi )^2 \Big\lvert \widehat{P_X^{**}} \Big( \frac{1}{2} , \chi \Big) \Big\rvert^2$:} \\

We have that
\begin{align}
\begin{split} \label{fourth twisted moment, b(chi) sum, split into diag and off diag, statement}
&\frac{1}{\phi^* (R)} \sumstar_{\chi \modulus R} b(\chi )^2 \Big\lvert \widehat{P_X^{**}} \Big( \frac{1}{2} , \chi \Big) \Big\rvert^2
\leq \frac{1}{\phi^* (R)} \sum_{\chi \modulus R} b(\chi )^2 \Big\lvert \widehat{P_X^{**}} \Big( \frac{1}{2} , \chi \Big) \Big\rvert^2 \\
\leq & \frac{1}{\phi^* (R)} \sum_{\chi \modulus R} \sum_{\substack{A_1 , A_2 , B_1 , B_2 \in \mathcal{M} \\ A_3 , B_3 \in \mathcal{S}_{\mathcal{M}} (X) \\ z_R < \degree A_1 B_1 , \degree A_2 B_2 < \degree R \\ \degree A_3 , \degree B_3 \leq \frac{1}{8} \log_q \degree R}}
	\frac{\beta (A_3) \beta (B_3) \chi (A_1 A_2 A_3) \conj\chi (B_1 B_2 B_3)}{\lvert A_1 A_2 A_3 B_1 B_2 B_3 \rvert^{\frac{1}{2}}} \\
= & \frac{ \phi (R)}{\phi^* (R)} \hspace{-2em} \sum_{\substack{A_1 , A_2 , B_1 , B_2 \in \mathcal{M} \\ A_3 , B_3 \in \mathcal{S}_{\mathcal{M}} (X) \\ z_R < \degree A_1 B_1 , \degree A_2 B_2 < \degree R \\ \degree A_3 , \degree B_3 \leq \frac{1}{8} \log_q \degree R \\ (A_1 A_2 A_3 B_1 B_2 B_3 , R)=1 \\ A_1 A_2 A_3 = B_1 B_2 B_3}}
		\hspace{-1.5em} \frac{\beta (A_3) \beta (B_3)}{\lvert A_1 A_2 A_3 B_1 B_2 B_3 \rvert^{\frac{1}{2}}} 
	\hspace{1em} + \frac{ \phi (R)}{\phi^* (R)} \hspace{-2em} \sum_{\substack{A_1 , A_2 , B_1 , B_2 \in \mathcal{M} \\ A_3 , B_3 \in \mathcal{S}_{\mathcal{M}} (X) \\ z_R < \degree A_1 B_1 , \degree A_2 B_2 < \degree R \\ \degree A_3 , \degree B_3 \leq \frac{1}{8} \log_q \degree R \\ (A_1 A_2 A_3 B_1 B_2 B_3 , R)=1 \\ A_1 A_2 A_3 \equiv B_1 B_2 B_3 (\modulus R) \\ A_1 A_2 A_3 \neq B_1 B_2 B_3}}
		\hspace{-1.5em} \frac{\beta (A_3) \beta (B_3)}{\lvert A_1 A_2 A_3 B_1 B_2 B_3 \rvert^{\frac{1}{2}}} .
\end{split}
\end{align}

\textbf{Step 2.1:} For the diagonal term, by similar means as in (\ref{diagonal terms as sum over Gs Vs times sum over V times sum over G1 G2}), we obtain
\begin{align}
\begin{split}
& \sum_{\substack{A_1 , A_2 , B_1 , B_2 \in \mathcal{M} \\ A_3 , B_3 \in \mathcal{S}_{\mathcal{M}} (X) \\ z_R < \degree A_1 B_1 , \degree A_2 B_2 < \degree R \\ \degree A_3 , \degree B_3 \leq \frac{1}{8} \log_q \degree R \\ (A_1 A_2 A_3 B_1 B_2 B_3 , R)=1 \\ A_1 A_2 A_3 = B_1 B_2 B_3 }} \frac{\beta (A_3) \beta (B_3) }{\lvert A_1 A_2 A_3 B_1 B_2 B_3 \rvert^{\frac{1}{2}}} \\
= & \sum_{\substack{G_3 , V_{1,3} , V_{2,3} , V_{3,1} , V_{3,2} \in \mathcal{S}_{\mathcal{M}} (X) \\ \degree G_3 V_{3,1} V_{3,2} \leq \frac{1}{8} \log_q \degree R \\ \degree G_3 V_{1,3} V_{2,3} \leq \frac{1}{8} \log_q \degree R \\ (G_3 V_{1,3} V_{2,3} V_{3,1} V_{3,2} , R) = 1 \\ (V_{1,3} V_{2,3} , V_{3,1} V_{3,2}) = 1 }}
		\hspace{-1em} \frac{\beta (G_3 V_{3,1} V_{3,2}) \beta (G_3 V_{1,3} V_{2,3}) }{\lvert G_3 V_{1,3} V_{2,3} V_{3,1} V_{3,2} \rvert} 
	\hspace{-1em} \sum_{\substack{V \in \mathcal{M} \\ \degree V \leq \degree R - \degree V_{1,3} V_{3,1} \\ \degree V \leq \degree R - \degree V_{2,3} V_{3,2} \\ \big( V , R(V_{1,3} V_{3,1} , V_{2,3} V_{3,2}) \big) = 1}} 
			\hspace{-1em} \frac{2^{\omega(V) - \omega \Big( \big(V , V_{1,3} V_{2,3} V_{3,1} V_{3,2} \big) \Big)}}{\lvert V \rvert} \\
		& \hspace{5em} \cdot  \sum_{\substack{ G_1 , G_2 \in \mathcal{M} \\ \max \Big\{ 0 , \frac{z_R - \degree V V_{1,3} V_{3,1}}{2} \Big\} < \degree G_1 < \frac{\degree R - \degree V V_{1,3} V_{3,1}}{2} \\ \max \Big\{ 0 , \frac{z_R - \degree V V_{2,3} V_{3,2}}{2} \Big\} < \degree G_2 < \frac{\degree R - \degree V V_{2,3} V_{3,2}}{2} \\ (G_1 G_2 , R) = 1}} \frac{1}{\lvert G_1 G_2 \rvert} .
\end{split}
\end{align}
Now, if $\frac{z_R - \degree V V_{1,3} V_{3,1}}{2} \leq \log_q 3^{\omega (R)}$ then 
\begin{align*}
\frac{\degree R - \degree V V_{1,3} V_{3,1}}{2} \leq \log_q 3^{\omega (R)} + \frac{1}{2} \log_q 2^{\omega (R)} < \log_q 6^{\omega (R)},
\end{align*}
and so, by Corollary \ref{Sum over (A,R)=1, deg A <= a deg R of 1/A}, we have
\begin{align*}
\sum_{\substack{ G_1 \in \mathcal{M} \\ \max \Big\{ 0 , \frac{z_R - \degree V V_{1,3} V_{3,1}}{2} \Big\} < \degree G_1 < \frac{\degree R - \degree V V_{1,3} V_{3,1}}{2} \\ (G_1 , R) = 1}} \frac{1}{\lvert G_1 \rvert} 
\leq \sum_{\substack{ G_1 \in \mathcal{M} \\ \degree G_1 < \log_q 6^{\omega (R)} \\ (G_1 , R) = 1}} \frac{1}{\lvert G_1 \rvert}
\ll \frac{\phi (R)}{\lvert R \rvert} \omega (R) .
\end{align*}
If $\frac{z_R - \degree V V_{1,3} V_{3,1}}{2} > \log_q 3^{\omega (R)}$ then
\begin{align*}
&\sum_{\substack{ G_1 \in \mathcal{M} \\ \max \Big\{ 0 , \frac{z_R - \degree V V_{1,3} V_{3,1}}{2} \Big\} < \degree G_1 < \frac{\degree R - \degree V V_{1,3} V_{3,1}}{2} \\ (G_1 , R) = 1}} \frac{1}{\lvert G_1 \rvert} \\
= &\sum_{\substack{ G_1 \in \mathcal{M} \\ \degree G_1 < \frac{\degree R - \degree V V_{1,3} V_{3,1}}{2} \\ (G_1 , R) = 1}} \frac{1}{\lvert G_1 \rvert} 
	- \sum_{\substack{ G_1 \in \mathcal{M} \\ \degree G_1 < \frac{z_R - \degree V V_{1,3} V_{3,1}}{2} \\ (G_1 , R) = 1}} \frac{1}{\lvert G_1 \rvert} 
\ll \frac{\phi (R)}{\lvert R \rvert} \omega (R) ,
\end{align*}
where we have used Corollary \ref{Sum over (A,R)=1, deg A <= a deg R of 1/A} twice for the last relation. Similar results hold for the sum over $G_2$. Hence, proceeding similarly as we did for the diagonal terms of $\frac{1}{\phi^* (R)} \sumstar_{\chi \modulus R} a(\chi )^2 \Big\lvert \widehat{P_X^{**}} \Big( \frac{1}{2} , \chi \Big) \Big\rvert^2$, we see that there is a constant $c$ such that
\begin{align*}
&\frac{\phi (R)}{\phi^* (R)} \sum_{\substack{A_1 , A_2 , B_1 , B_2 \in \mathcal{M} \\ A_3 , B_3 \in \mathcal{S}_{\mathcal{M}} (X) \\ z < \degree A_1 B_1 , \degree A_2 B_2 < \degree R \\ \degree A_3 , \degree B_3 \leq \frac{1}{8} \log_q \degree R \\ (A_1 A_2 A_3 B_1 B_2 B_3 , R)=1 \\ A_1 A_2 A_3 = B_1 B_2 B_3 }} \hspace{-1em} \frac{\beta (A_3) \beta (B_3) }{\lvert A_1 A_2 A_3 B_1 B_2 B_3 \rvert^{\frac{1}{2}}} \\
\ll &\frac{\phi (R)^3 }{\lvert R \rvert^2 \phi^* (R) } \omega (R)^2 (\degree R)^2 \prod_{P \mid R} \bigg( \frac{(1-\lvert P \rvert^{-1})^3}{1+ \lvert P \rvert^{-1}} \bigg) (\log_q \log \degree R )^c .
\end{align*}

\textbf{Step 2.2:} We now look at the second term on the far right side of (\ref{fourth twisted moment, b(chi) sum, split into diag and off diag, statement}): the off-diagonal terms. Using Lemma \ref{Lemma, off diagonal z_1 z_2 sum, EH version}, we have
\begin{align*}
&\frac{ \phi (R)}{\phi^* (R)} \sum_{\substack{A_1 , A_2 , B_1 , B_2 \in \mathcal{M} \\ A_3 , B_3 \in \mathcal{S}_{\mathcal{M}} (X) \\ z_R < \degree A_1 B_1 , \degree A_2 B_2 < \degree R \\ \degree A_3 , \degree B_3 \leq \frac{1}{8} \log_q \degree R \\ (A_1 A_2 A_3 B_1 B_2 B_3 , R)=1 \\ A_1 A_2 A_3 \equiv B_1 B_2 B_3 (\modulus R) \\ A_1 A_2 A_3 \neq B_1 B_2 B_3}}
	\frac{\beta (A_3) \beta (B_3)}{\lvert A_1 A_2 A_3 B_1 B_2 B_3 \rvert^{\frac{1}{2}}} \\
= &\frac{ \phi (R)}{\phi^* (R)} \sum_{\substack{A_3 , B_3 \in \mathcal{S}_{\mathcal{M}} (X) \\ \degree A_3 , \degree B_3 \leq \frac{1}{8} \log_q \degree R \\ (A_3 B_3 , R)=1}} \frac{\beta (A_3) \beta (B_3)}{\lvert A_3 B_3 \rvert^{\frac{1}{2}}}
	\sum_{z_R < z_1 , z_2< \degree R} q^{-\frac{z_1 + z_2 }{2}}
	\sum_{\substack{A_1 , A_2 , B_1 , B_2 \in \mathcal{M} \\ \degree A_1 B_1 = z_1 \\ \degree A_2 B_2 = z_2 \\ (A_1 A_2 A_3 B_1 B_2 B_3 , R)=1 \\ A_1 A_2 A_3 \equiv B_1 B_2 B_3 (\modulus R) \\ A_1 A_2 A_3 \neq B_1 B_2 B_3}} 1 \\
\ll & \frac{(\degree R)^3}{\phi^* (R)} \sum_{\substack{A_3 , B_3 \in \mathcal{S}_{\mathcal{M}} (X) \\ \degree A_3 , \degree B_3 \leq \frac{1}{8} \log_q \degree R \\ (A_3 B_3 , R)=1}} \lvert \beta (A_3) \beta (B_3) \rvert \lvert A_3 B_3 \rvert^{\frac{1}{2}}
	\sum_{z_R < z_1 , z_2< \degree R} q^{\frac{z_1 + z_2 }{2}} \\
\ll & \frac{\lvert R \rvert (\degree R)^3 }{\phi^* (R)} \sum_{\substack{A_3 , B_3 \in \mathcal{S}_{\mathcal{M}} (X) \\ \degree A_3 , \degree B_3 \leq \frac{1}{8} \log_q \degree R \\ (A_3 B_3 , R)=1}} \lvert \beta (A_3) \beta (B_3) \rvert \lvert A_3 B_3 \rvert^{\frac{1}{2}} 
\ll \frac{\lvert R \rvert (\degree R)^{3 + \frac{3}{4}} }{\phi^* (R)} .
\end{align*}

\textbf{Step 3; the asymptotic main term of $\frac{1}{\phi^* (R)} \sumstar_{\chi \modulus R} c(\chi )^2 \Big\lvert \widehat{P_X^{**}} \Big( \frac{1}{2} , \chi \Big) \Big\rvert^2$:} \\

We recall that $c (\chi )$ differs, depending on whether $\chi$ is even or odd. Furthermore, if $\chi$ is even, then there are three terms to consider. However, by the Cauchy-Schwarz inequality, it suffices to bound the following for $i=0,1,2,$:
\begin{align*}
\frac{1}{\phi^* (R)} \sumstar_{\substack{\chi \modulus R \\ \chi \text{ even} }} d_i (\chi )^2 \Big\lvert \widehat{P_X^{**}} \Big( \frac{1}{2} , \chi \Big) \Big\rvert^2 ,
\end{align*}
where
\begin{align*}
d_i (\chi )
:= \sum_{\substack{A,B \in \mathcal{M} \\ \degree AB = \degree R - i}} \frac{\chi (A) \conj\chi (B)}{\lvert AB \rvert^{\frac{1}{2}}} .
\end{align*}
We will bound 
\begin{align*}
\frac{1}{\phi^* (R)} \sumstar_{\substack{\chi \modulus R \\ \chi \text{ even} }} d_0 (\chi )^2 \Big\lvert \widehat{P_X^{**}} \Big( \frac{1}{2} , \chi \Big) \Big\rvert^2 .
\end{align*}
The other cases for $d_i (\chi)$ and the odd case are similar. \\

Now, we have that
\begin{align}
\begin{split} \label{fourth twisted moment, c(chi) sum, split into diag and off diag, statement}
&\frac{1}{\phi^* (R)} \sumstar_{\substack{\chi \modulus R \\ \chi \text{ even} }} d_0 (\chi )^2 \Big\lvert \widehat{P_X^{**}} \Big( \frac{1}{2} , \chi \Big) \Big\rvert^2
\leq \frac{1}{\phi^* (R)} \sum_{\chi \modulus R} d_0 (\chi )^2 \Big\lvert \widehat{P_X^{**}} \Big( \frac{1}{2} , \chi \Big) \Big\rvert^2 \\
\leq & \frac{ 1}{\phi^* (R)} \sum_{\chi \modulus R} \sum_{\substack{A_1 , A_2 , B_1 , B_2 \in \mathcal{M} \\ A_3 , B_3 \in \mathcal{S}_{\mathcal{M}} (X) \\ \degree A_1 B_1 , \degree A_2 B_2 = \degree R \\ \degree A_3 , \degree B_3 \leq \frac{1}{8} \log_q \degree R}}
	\frac{\beta (A_3) \beta (B_3) \chi (A_1 A_2 A_3) \conj\chi (B_1 B_2 B_3)}{\lvert A_1 A_2 A_3 B_1 B_2 B_3 \rvert^{\frac{1}{2}}} \\
= & \frac{ \phi (R)}{\phi^* (R)} \hspace{-2em} \sum_{\substack{A_1 , A_2 , B_1 , B_2 \in \mathcal{M} \\ A_3 , B_3 \in \mathcal{S}_{\mathcal{M}} (X) \\ \degree A_1 B_1 , \degree A_2 B_2 = \degree R \\ \degree A_3 , \degree B_3 \leq \frac{1}{8} \log_q \degree R \\ (A_1 A_2 A_3 B_1 B_2 B_3 , R)=1 \\ A_1 A_2 A_3 = B_1 B_2 B_3}}
	\hspace{-0.5em} \frac{\beta (A_3) \beta (B_3)}{\lvert A_1 A_2 A_3 B_1 B_2 B_3 \rvert^{\frac{1}{2}}} 
	\hspace{1em} + \frac{ \phi (R)}{\phi^* (R)} \hspace{-2em} \sum_{\substack{A_1 , A_2 , B_1 , B_2 \in \mathcal{M} \\ A_3 , B_3 \in \mathcal{S}_{\mathcal{M}} (X) \\ \degree A_1 B_1 , \degree A_2 B_2 = \degree R \\ \degree A_3 , \degree B_3 \leq \frac{1}{8} \log_q \degree R \\ (A_1 A_2 A_3 B_1 B_2 B_3 , R)=1 \\ A_1 A_2 A_3 \equiv B_1 B_2 B_3 (\modulus R) \\ A_1 A_2 A_3 \neq B_1 B_2 B_3}}
	\hspace{-0.5em} \frac{\beta (A_3) \beta (B_3)}{\lvert A_1 A_2 A_3 B_1 B_2 B_3 \rvert^{\frac{1}{2}}} .
\end{split}
\end{align}
For the first term on the far right side of (\ref{fourth twisted moment, c(chi) sum, split into diag and off diag, statement}), we have, similarly to Step 2.1,
\begin{align*}
& \sum_{\substack{A_1 , A_2 , B_1 , B_2 \in \mathcal{M} \\ A_3 , B_3 \in \mathcal{S}_{\mathcal{M}} (X) \\ \degree A_1 B_1 , \degree A_2 B_2 = \degree R \\ \degree A_3 , \degree B_3 \leq \frac{1}{8} \log_q \degree R \\ (A_1 A_2 A_3 B_1 B_2 B_3 , R)=1 \\ A_1 A_2 A_3 = B_1 B_2 B_3 }} \frac{\beta (A_3) \beta (B_3) }{\lvert A_1 A_2 A_3 B_1 B_2 B_3 \rvert^{\frac{1}{2}}} \\
= & \sum_{\substack{G_3 , V_{1,3} , V_{2,3} , V_{3,1} , V_{3,2} \in \mathcal{S}_{\mathcal{M}} (X) \\ \degree G_3 V_{3,1} V_{3,2} \leq \frac{1}{8} \log_q \degree R \\ \degree G_3 V_{1,3} V_{2,3} \leq \frac{1}{8} \log_q \degree R \\ (G_3 V_{1,3} V_{2,3} V_{3,1} V_{3,2} , R) = 1 \\ (V_{1,3} V_{2,3} , V_{3,1} V_{3,2}) = 1 }}
		\hspace{-1.75em} \frac{\beta (G_3 V_{3,1} V_{3,2}) \beta (G_3 V_{1,3} V_{2,3}) }{\lvert G_3 V_{1,3} V_{2,3} V_{3,1} V_{3,2} \rvert} 
		\hspace{-0.5em} \sum_{\substack{V \in \mathcal{M} \\ \degree V \leq \degree R - \degree V_{1,3} V_{3,1} \\ \degree V \leq \degree R - \degree V_{2,3} V_{3,2} \\ \big( V , R(V_{1,3} V_{3,1} , V_{2,3} V_{3,2}) \big) = 1}} 
			\hspace{-2em} \frac{2^{\omega(V) - \omega \Big( \big(V , V_{1,3} V_{2,3} V_{3,1} V_{3,2} \big) \Big)}}{\lvert V \rvert} \\
		& \hspace{3em} \cdot  \sum_{\substack{ G_1 , G_2 \in \mathcal{M} \\ \degree G_1 = \frac{\degree R - \degree V V_{1,3} V_{3,1}}{2} \\ \degree G_2 = \frac{\degree R - \degree V V_{2,3} V_{3,2}}{2} \\ (G_1 G_2 , R) = 1}} \frac{1}{\lvert G_1 G_2 \rvert} \\
\ll & \sum_{\substack{G_3 , V_{1,3} , V_{2,3} , V_{3,1} , V_{3,2} \in \mathcal{S}_{\mathcal{M}} (X) \\ \degree G_3 V_{3,1} V_{3,2} \leq \frac{1}{8} \log_q \degree R \\ \degree G_3 V_{1,3} V_{2,3} \leq \frac{1}{8} \log_q \degree R \\ (G_3 V_{1,3} V_{2,3} V_{3,1} V_{3,2} , R) = 1 \\ (V_{1,3} V_{2,3} , V_{3,1} V_{3,2}) = 1 }} 
		\hspace{-1.75em} \frac{\lvert \beta (G_3 V_{3,1} V_{3,2}) \beta (G_3 V_{1,3} V_{2,3}) \rvert }{\lvert G_3 V_{1,3} V_{2,3} V_{3,1} V_{3,2} \rvert} 
	\hspace{-1em} \sum_{\substack{V \in \mathcal{M} \\ \degree V \leq \degree R - \degree V_{1,3} V_{3,1} \\ \degree V \leq \degree R - \degree V_{2,3} V_{3,2} \\ \big( V , R(V_{1,3} V_{3,1} , V_{2,3} V_{3,2}) \big) = 1}} 
			\hspace{-2em} \frac{2^{\omega(V) - \omega \Big( \big(V , V_{1,3} V_{2,3} V_{3,1} V_{3,2} \big) \Big)}}{\lvert V \rvert} \\
\ll &(\degree R)^2 \prod_{P \mid R} \bigg( \frac{(1-\lvert P \rvert^{-1})^3}{1+ \lvert P \rvert^{-1}} \bigg) (\log_q \log \degree R )^c ,
\end{align*}
for some positive constant $c$. \\

For the second term on the far right side of (\ref{fourth twisted moment, c(chi) sum, split into diag and off diag, statement}), we have, similarly to Step 2.2,
\begin{align*}
&\frac{ \phi (R)}{\phi^* (R)} \sum_{\substack{A_1 , A_2 , B_1 , B_2 \in \mathcal{M} \\ A_3 , B_3 \in \mathcal{S}_{\mathcal{M}} (X) \\ \degree A_1 B_1 , \degree A_2 B_2 = \degree R \\ \degree A_3 , \degree B_3 \leq \frac{1}{8} \log_q \degree R \\ (A_1 A_2 A_3 B_1 B_2 B_3 , R)=1 \\ A_1 A_2 A_3 \equiv B_1 B_2 B_3 (\modulus R) \\ A_1 A_2 A_3 \neq B_1 B_2 B_3}}
	\frac{\beta (A_3) \beta (B_3)}{\lvert A_1 A_2 A_3 B_1 B_2 B_3 \rvert^{\frac{1}{2}}} \\
= &\frac{ \phi (R)}{\lvert R \rvert \phi^* (R)} \sum_{\substack{A_3 , B_3 \in \mathcal{S}_{\mathcal{M}} (X) \\ \degree A_3 , \degree B_3 \leq \frac{1}{8} \log_q \degree R \\ (A_3 B_3 , R)=1}} \frac{\beta (A_3) \beta (B_3)}{\lvert A_3 B_3 \rvert^{\frac{1}{2}}}
	\sum_{\substack{A_1 , A_2 , B_1 , B_2 \in \mathcal{M} \\ \degree A_1 B_1 = \degree R \\ \degree A_2 B_2 = \degree R \\ (A_1 A_2 A_3 B_1 B_2 B_3 , R)=1 \\ A_1 A_2 A_3 \equiv B_1 B_2 B_3 (\modulus R) \\ A_1 A_2 A_3 \neq B_1 B_2 B_3}} 1 \\
\ll & \frac{\lvert R \rvert (\degree R)^3}{\phi^* (R)} \sum_{\substack{A_3 , B_3 \in \mathcal{S}_{\mathcal{M}} (X) \\ \degree A_3 , \degree B_3 \leq \frac{1}{8} \log_q \degree R \\ (A_3 B_3 , R)=1}} \lvert \beta (A_3) \beta (B_3) \rvert \lvert A_3 B_3 \rvert^{\frac{1}{2}} 
\ll \frac{\lvert R \rvert (\degree R)^{3 + \frac{3}{4}} }{\phi^* (R)} .
\end{align*}

The proof now follows from Steps 1, 2, and 3.

\end{proof}

\bibliography{YiasemidesBibliography1}{}
\bibliographystyle{YiaseBstNumer1}

\end{document}